\documentclass[letterpaper, oneside]{amsart}
\usepackage{layout}	

\usepackage[
]{geometry}

\usepackage{amssymb}
\usepackage{amsthm}
\usepackage{amsmath}
\usepackage{enumitem}
\usepackage{mathtools}
\usepackage{hyperref}
\usepackage{colonequals}
\usepackage{mathrsfs}
\usepackage{color}
\usepackage{subcaption}
\usepackage{adjustbox}
\usepackage{ytableau}
\usepackage{bm}
\usepackage{bbm}
\usepackage{tensor}
\usepackage{esvect}
\usepackage{tabmac}
\usepackage{comment}

\usepackage{tikz}
\usetikzlibrary{cd}
\usetikzlibrary{arrows.meta}
\usetikzlibrary{decorations.pathmorphing}
\tikzset{
 	arr/.style={
		dashed,-{Latex[length=2mm]}
	},
	tl/.style={
  		ultra thick
	},
	remo/.style={
  		line width=0
	}
}

\theoremstyle{plain}
\newtheorem{thm}{Theorem}[section]
\newtheorem{lem}[thm]{Lemma}
\newtheorem{prop}[thm]{Proposition}
\newtheorem{cor}[thm]{Corollary}

\theoremstyle{definition}

\newtheorem*{ack}{Acknowledgement}

\newtheorem{example}[thm]{Example}

\theoremstyle{remark}
\newtheorem*{rmk}{Remark}

\numberwithin{equation}{section}

\newcommand{\bN}{\mathbb{N}}

\newcommand{\bQ}{\mathbb{Q}}
\newcommand{\bR}{\mathbb{R}}

\newcommand{\bZ}{\mathbb{Z}}

\newcommand{\cC}{\mathcal{C}}

\newcommand{\cH}{\mathcal{H}}

\newcommand{\cO}{\mathcal{O}}
\newcommand{\cP}{\mathcal{P}}

\newcommand{\cT}{\mathcal{T}}

\newcommand{\fA}{\mathfrak{A}}

\newcommand{\fT}{\mathfrak{T}}
\newcommand{\fF}{\mathfrak{F}}
\newcommand{\fI}{\mathfrak{P}}

\newcommand{\sA}{\mathscr{A}}

\newcommand{\Sym}{\mathcal{S}}
\newcommand{\affSn}{\overline{\Sym_n}}
\newcommand{\affI}{\overline{I}}
\newcommand{\affS}[1]{\overline{\Sym_{#1}}}
\newcommand{\extSn}{\widetilde{\Sym_n}}

\newcommand{\shift}{\bm{\omega}}
\newcommand{\tsc}{{\underline{c}}}

\newcommand{\Tc}{{T}^{\textnormal{can}}}

\newcommand{\Tas}{{T}^{\textnormal{as}}}

\newcommand{\fw}{\textnormal{fw}}

\newcommand{\RSYT}{\textnormal{RSYT}}
\newcommand{\SSYT}{\textnormal{SSYT}}
\newcommand{\SYT}{\textnormal{SYT}}

\newcommand{\fsh}{\textnormal{FinSh}}
\newcommand{\RSK}{\textnormal{RSK}}

\newcommand{\awg}{\overline{\Gamma}}
\newcommand{\fdeg}{\mathcal{D}}
\newcommand{\adeg}{\overline{\fdeg}}
\newcommand{\des}{\textnormal{des}}
\newcommand{\ades}{\overline{\des}}

\newcommand{\pwg}{\overline{\Gamma}^\textnormal{per}}
\newcommand{\qwg}{\overline{\Gamma}^\textnormal{quot}}
\newcommand{\fund}{\mathfrak{D}}

\newcommand{\dt}[1]{\textbf{#1}}

\DeclareMathOperator{\sh}{\textnormal{Sh}}

\DeclareMathOperator{\Aut}{\textnormal{Aut}}

\DeclareMathOperator{\im}{\textnormal{im}}

\DeclareSymbolFont{extraup}{U}{zavm}{m}{n}
\DeclareMathSymbol{\varheart}{\mathalpha}{extraup}{86}
\DeclareMathSymbol{\vardiamond}{\mathalpha}{extraup}{87}

\DeclareMathOperator{\intch}{\mathrel{\text{$\nwarrow$\llap{$\searrow$}}}}
\DeclareMathOperator{\rintch}{\mathrel{\text{$\nearrow$\llap{$\swarrow$}}}}

\newcommand{\ceil}[1]{\left\lceil{#1}\right\rceil}

\newcommand{\mo}[1]{{\overline{#1}}}
\newcommand{\pint}[1]{\ulcorner{#1}\lrcorner}
\newcommand{\pres}[1]{\mathord{\downarrow}_{#1}}
\newcommand{\Nf}[4]{N^{#1\mathord{\intch}#2,#3\mathord{\intch}#4}}
\newcommand{\ra}{\triangleright}

\newgeometry{margin=1in}

\title{Two-row $W$-graphs in affine type $A$}
\author{Dongkwan Kim} 
\author{Pavlo Pylyavskyy}
\address{School of Mathematics\\
  University of Minnesota Twin Cities\\
  Minneapolis, MN 55455\\
  U.S.A.}
\email{kim00657@umn.edu}
\email{ppylyavs@umn.edu}
\date{\today}							

\begin{document}
\begin{abstract} 
For affine symmetric groups we construct finite $W$-graphs 
corresponding to two-row shapes, 
and prove their uniqueness. This gives the first non-trivial family of purely combinatorial constructions of finite $W$-graphs in an affine type. We compare our construction with quotients of periodic $W$-graphs defined by Lusztig. Under certain positivity assumption on the latter the two are shown to be isomorphic. 
\end{abstract}

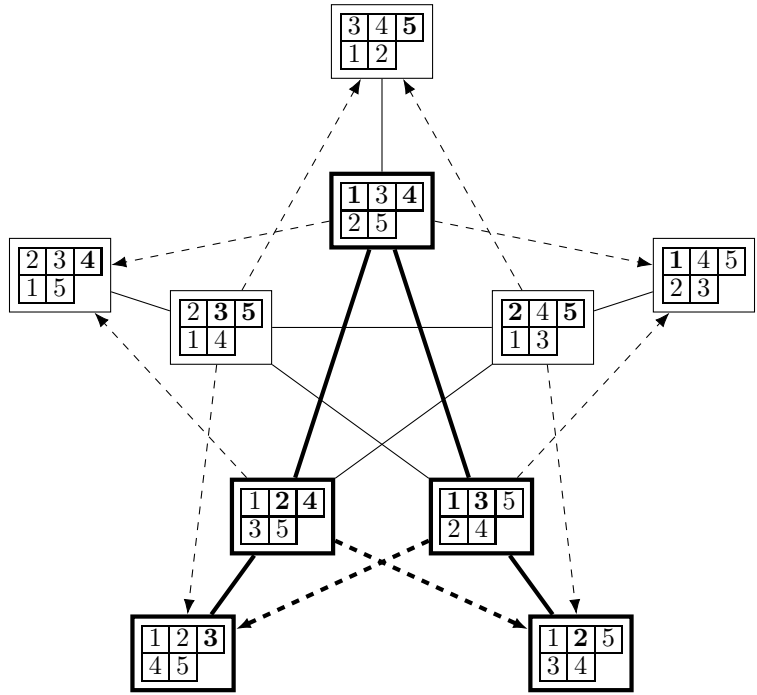
\setcounter{tocdepth}{1}
\maketitle

\renewcommand\contentsname{}
\tableofcontents

\ytableausetup{boxsize=1em, notabloids}
\begin{figure}[h] 
\begin{center}
\begin{tikzpicture}[scale=1.5]

	\node[draw, tl] (o1) at (90+144:3) {\ytableaushort{12{\dt{3}},45}};
	\node[draw] (o2) at (18+144:3) {\ytableaushort{23{\dt{4}},15}};
	\node[draw] (o3) at (306+144:3) {\ytableaushort{34{\dt{5}},12}};
	\node[draw] (o4) at (234+144:3) {\ytableaushort{{\dt{1}}45,23}};
	\node[draw, tl] (o5) at (162+144:3) {\ytableaushort{1{\dt{2}}5,34}};
	
	\node[draw, tl] (i1) at (90+144:1.5) {\ytableaushort{1{\dt{2}}{\dt{4}},35}};
	\node[draw] (i2) at (18+144:1.5) {\ytableaushort{2{\dt{3}}{\dt{5}},14}};
	\node[draw, tl] (i3) at (306+144:1.5) {\ytableaushort{{\dt{1}}3{\dt{4}},25}};
	\node[draw] (i4) at (234+144:1.5) {\ytableaushort{{\dt{2}}4{\dt{5}},13}};
	\node[draw, tl] (i5) at (162+144:1.5) {\ytableaushort{{\dt{1}}{\dt{3}}5,24}};
	
	\draw[tl] (o1) -- (i1);\draw[] (o2) -- (i2);\draw[] (o3) -- (i3);
	\draw[] (o4) -- (i4);\draw[tl] (o5) -- (i5);
	
	\draw[tl] (i1) -- (i3);\draw[] (i2) -- (i4);\draw[tl] (i3) -- (i5);
	\draw[] (i4) -- (i1);\draw[] (i5) -- (i2);
	
	\draw[arr] (i1) -- (o2);	\draw[arr, tl] (i1) -- (o5);
	\draw[arr] (i2) -- (o3);	\draw[arr] (i2) -- (o1);
	\draw[arr] (i3) -- (o4);	\draw[arr] (i3) -- (o2);
	\draw[arr] (i4) -- (o5);	\draw[arr] (i4) -- (o3);
	\draw[arr, tl] (i5) -- (o1);	\draw[arr] (i5) -- (o4);	

\end{tikzpicture}
\end{center}
\end{figure}

\newpage
\setlength{\parskip}{5pt} 		
\section{Introduction}

In their groundbreaking paper \cite{kl79} Kazhdan and Lusztig laid the groundwork for an approach to the representation theory of Hecke algebras. Since then this approach has been significantly developed, and is called \emph{Kazhdan-Lusztig theory}. 
Of special importance are the {\it {$W$-graphs}} that encode representations of Hecke algebras in a combinatorial way. Those are certain directed graphs with additional data given at vertices and edges. 
Certain $W$-graphs arise from {\it {Kazhdan-Lusztig cells}} in a canonical way, to which we refer as {\it {Kazhdan-Lusztig $W$-graphs}}. Stembridge \cite{ste08} has introduced a class of $W$-graphs called {\it {admissible}}, they include, but are not limited to, Kazhdan-Lusztig $W$-graphs. Giving an explicit elementary description of $W$-graphs in a conceptual way is an excruciatingly hard task, and constitutes one of the major problems in algebraic combinatorics and representation theory.  


There are two kinds of edges in $W$-graphs: undirected and directed. It is easier to understand the undirected edges, one could say that this problem is tame. For example, in type $A$ the undirected edges of Kazhdan-Lustig $W$-graphs are given by {\it {Knuth moves}} \cite{knu70} on permutations. If one restricts the information contained in a Kazhdan-Lusztig $W$-graph in type $A$ to undirected edges, one obtaines a {\it {dual equivalence graph}} of Haiman \cite{hai92}. The latter are well-understood, notably Assaf \cite{assaf07} has given a local characterization of dual equivalence graphs, similar to that of $W$-graphs by Stembridge \cite{ste08}. 

On the other hand, understanding the directed edges appears to be a wild problem. Chmutov \cite{chm15} has shown that in an admissible $W$-graph of an irreducible representation of type $A$ Hecke algebra the undirected edges must form one of the dual equivalence graphs, i.e. coincide with undirected edges of one of the Kazhdan-Lusztig $W$-graphs. Nguyen \cite{ngu18} further strengthened this to say that the directed edges must coincide with those of a Kazhdan-Lusztig $W$-graph as well, i.e. that in type $A$ all irreducible $W$-graphs are Kazhdan-Lusztig. This was originally a conjecture by Stembridge, see \cite{ste08}. Despite these strong results, in finite type $A$ an explicit construction of $W$-graphs is known only for hook shapes \cite{fung03} and two-row shapes, see \cite{wes95} where it is attributed to Lascoux-Schutzenberger. In Section \ref{sec:fin} we give an equivalent formulation of the latter construction in terms of tableaux, as opposed to strand diagrams. 

As one passes to affine type $A$, a lot less is known. In this case Kazhdan-Lusztig cells are labelled by tabloids, as opposed to standard Young tableaux in finite type $A$. One can still restrict Kazhdan-Lustig $W$-graphs to undirected edges, connected components of the resulting graph being {\it {Kazhdan-Lusztig molecules}} in Stembridge's terminology \cite{stem08}. A comprehensive description of those was recently given by Chmutov, Yudovina, Lewis and the second coauthor \cite{clp17, cpy18}. The majority of Kazhdan-Lusztig molecules are infinite, and the majority of Kazhdan-Lusztig cells contain infinitely many  molecules. This is in sharp contrast to type $A$, where each Kazhdan-Lusztig cell is unimolecular and finite. 

{\it {Affine dual equivalence graphs}} were introduced in \cite{cpy18} as natural quotients of affine Kazhdan-Lusztig molecules (not to be confused with a similar notion introduced by Assaf and Billey in an unrelated context \cite{assbill12}). Unlike molecules, affine dual equivalence graphs are always finite. A natural question arises of whether affine dual equivalence graphs can be enriched by directed edges to obtain genuine $W$-graphs, and whether such enrichment is unique. 

The first goal of this paper is to answer this affirmatively for two-row shapes. In Section \ref{sec:defawg} we give a concrete combinatorial rule for construction of the $W$-graphs the undirected part of which coincides with two-row affine dual equivalence graphs of \cite{cpy18}.  This constitutes the first non-trivial family of purely combinatorial constructions of finite $W$-graphs in an affine type. (Other examples, based on representation theory, can be obtained from taking certain quotients of Lusztig's periodic $W$-graphs \cite{lus80,lus97}; see Section \ref{sec:periodicWgraph} for more details.) The resulting $W$-graphs have the property that when restricted to finite Hecke algebra, one obtaines modules whose Frobenius character (in the sense of Ram \cite{ram91}) is a Hall-Littlewood symmetric function. 

The $W$-graphs constructed in this paper are manifestly non-bipartite. This is a strong indication that the bipartiteness condition often imposed in literature on $W$-graphs is not essential, and can be ignored. For example, it can be dropped from Stembridge's definition of admissible $W$-graphs, leaving the majority of the results unchanged. Similarly, recent impressive results of Nguyen \cite{ngu18} remain true if the bipartiteness requirement is omitted. In all relevant cases the original proofs carry through essentially verbatim, see Section \ref{sec:nb}.

The second goal of the paper is to show uniqueness of our construction assuming admissibility (except bipartiteness) of $W$-graphs. In Section \ref{sec:uni} we show that our construction of $W$-graphs is unique for shapes $(a,b)$, $a \not = b$, and is almost unique for shapes $(a,a)$. 

Finally, we invoke the notion of {\it {periodic $W$-graphs}} introduced by Lusztig in \cite{lus97} and how they are related to our construction. In general, periodic $W$-graphs are different from the usual Kazhdan-Lusztig $W$-graphs attached to cells and have periodicity as their name indicates. Under a certain finiteness assumption, which is proved by Varagnolo \cite{var04} for type $A$, one can take their quotients using this periodicity and obtain finite $W$-graphs of affine type. In this paper, we prove that our construction is isomorphic to such quotients of periodic $W$-graphs if we assume positivity of edge weights on the latter. We conjecture this to be true for all shapes, not just two-row ones. 

We believe that our construction of $W$-graphs provides important and useful examples in terms of representation theory. Here we discuss some possible applications to Springer theory. Firstly, Fung \cite{fun03} studied the connection between the components of Springer fibers and $W$-graphs for two-row and hook shapes (in which case the description of a $W$-graph is explicitly known). Likewise, one can consider the components of an affine Springer fiber, originally defined by Kazhdan-Lusztig \cite{kl88}, which is currently one of the central objects in geometric Langlands program. It is very interesting to ask if an analogous statement to Fung's is valid for affine Springer fibers and affine $W$-graphs.

Furthermore, it is known that periodic $W$-graphs provide a certain ``canonical basis'' of the (equivariant) $K$-theory of Springer fibers \cite{lus99}, which is in deep connection with modular representation of reductive Lie algebras and noncommutative Springer resolution \cite{bm13}. Even though the equivalence of our construction and the quotient of periodic $W$-graphs relies on the positivity conjecture which is still open as of now, we hope that the examples constructed in this paper are useful in practice when investigating such topics.

This paper is organized as follows. In Section 2, we introduce definitions and notations which are frequently used in this paper. In Section 3, we recall the notion of $W$-graphs and discuss their properties. In Section 4, we construct a graph $\awg_\lambda$ for a two-row partition $\lambda$ and study its properties. Here, we also state one of our main results that $\awg_\lambda$ is actually a $W$-graph of affine type $A$, whose proof is completed in Section 5 and 6. In Section 7, we discuss the restriction of $\awg_\lambda$ to the finite symmetric group. In Section 8 and 9, we prove that $\awg_\lambda$ satisfies certain uniqueness statement. In Section 10, we recollect the notion of Lusztig's periodic $W$-graphs and show how our construction of $\awg_\lambda$ is related to his graph under certain positivity assumption.

\begin{ack} We thank George Lusztig for his helpful comments on periodic $W$-graphs.
\end{ack}

\section{Definitions and Notations}
\subsection{Symmetric groups}
Throughout this paper we let $n \geq 3$ be a given natural number. Define $\Sym_n$ to be the symmetric group permuting $\{1, 2, \ldots, n\}$. We often regard it as a Coxeter group with the set of simple reflections $I=\{s_1, s_2, \ldots, s_{n-1}\}$ where $s_i$ is defined to be the transposition swapping $i$ and $i+1$. We define $\affSn$ and $\extSn$ to be the affine symmetric group and extended affine symmetric group, respectively. They are usually realized as
\begin{align*}
\extSn &\colonequals \{w \in \Aut(\bZ) \mid w(i+n) = w(i) \textup{ for any } i\in \bZ\},
\\\affSn &\colonequals \{w \in \extSn \mid \sum_{i=1}^n w(i) = n(n+1)/2\}.
\end{align*}
(Note that $\Sym_n$ is naturally a subgroup of both $\affSn$ and $\extSn$.)
For $w\in \extSn$, its window notation is given by $[w(1), w(2), \ldots, w(n)]$. It is clear that the window notation completely determines the element $w$. We also write $w=[w(1), w(2), \ldots, w(n)]$ to describe the element $w$. For example, we have $id = [1,2,\ldots, n]$. Also, note that $\affSn$ is a Coxeter group with the set of simple reflections $\affI=\{s_0=s_n, s_1, s_2, \ldots, s_{n-1}\}$ where $s_i=[1,2, \ldots, i-1, i+1,i,i+2, \ldots, n]$ for $1\leq i \leq n-1$ and $s_0 =s_n= [0, 2, \ldots, n-1, n+1]$. Define $\shift \in \extSn$ to be $\shift=[2, 3, \ldots, n, n+1]$, called the cyclic shift element. Then conjugation by $\shift$ defines an outer automorphism on $\affSn$ and we have $\extSn = \affSn \rtimes \langle \shift \rangle$.

\subsection{Partitions}
We say that $\lambda$ is a partition of $n$ if $\lambda$ is a finite sequence of integers, i.e. $\lambda=(\lambda_1, \lambda_2, \ldots, \lambda_l)$ where $\lambda_1, \ldots, \lambda_l \in \bZ$, which satisfies that $\lambda_1\geq \lambda_2 \geq \cdots \lambda_l >0$ and $\sum_{i=1}^l \lambda_i = n$. In this situation we also write $\lambda \vdash n$ and $|\lambda|=n$, and say that the size of $\lambda$ is $n$. The length of $\lambda$, denoted $l(\lambda)$, is its length considered as a sequence of (positive) integers. We usually identify a partition with its corresponding Young diagram (in terms of English convention) and thus its parts are often called rows.

\subsection{Young tableaux}
Let $\RSYT(n)$ (resp. $\SYT(n)$, $\SSYT(n)$) be the set of row-standard (resp. standard, semistandard) Young tableaux of size $n$. Here we require that each element of $\RSYT(n)$ consists of entries in $\{1, 2, \ldots, n\}$ and regard $\SYT(n)$ naturally as a subset of $\RSYT(n)$. For a partition $\lambda \vdash n$, we also let $\RSYT(\lambda) \subset \RSYT(n)$ (resp. $\SYT(\lambda) \subset \SYT(n)$, $\SSYT(\lambda) \subset \SSYT(n)$) be the subset of such tableaux of shape $\lambda$. In addition, for a sequence of positive integers $\mu=(\mu_1, \ldots, \mu_l)$ such that $\sum_{i=1}^l\mu_i = n$ (which is not necessarily a partition), we set $\SSYT(\lambda, \mu) \subset \SSYT(\lambda)$ to be the set of semistandard Young tableaux of shape $\lambda$ and content $\mu$. 

For a tableau $T$, define $\sh(T)$ to be the shape of $T$. We often regard a tableau $T$ as a sequence of integer sequences $(T^1, T^2, \ldots, T^{l(\lambda)})$ where each $T^a$ is an $a$-th row of $T$. For such $T$, we define the reading word of $T$ to be the concatenation $T^{l(\lambda)}\cdots T^{2}T^{1}$ (from bottom to top), considered either as a word or a sequence. Finally, for a tableau $T$ we set $T\pres{[1,i]}$ to be another tableau obtained from $T$ by removing boxes containing entries not in $\{1, 2, \ldots, i\}$.

\ytableausetup{centertableaux,notabloids}

\subsection{Robinson-Schensted-Knuth map on row-standard Young tableaux}
We define the Robinson-Schensted-Knuth map on $\RSYT(n)$ as follows. For $T\in \RSYT(n)$, consider the two-line array whose second row is the reading word of $T$ and whose first row records $l(\sh(T))+1-$(the row number) of corresponding entries. For example, the two-line array corresponding to 
$$T=\ytableaushort{2457,369,18} \quad \textup{ is given by } \quad \begin{pmatrix} 1 &1 & 2 & 2 & 2 & 3 & 3&3&3 \\ 1&8&3&6&9&2&4&5&7 \end{pmatrix}.$$
We define $\RSK(T) \colonequals (P(T), Q(T))$ to be the image of this two-line array under the usual Robinson-Schensted-Knuth correspondence, see \cite[Section 3]{knu70}. Thus in particular $P(T)$ is a standard Young tableaux and $Q(T)$ is a semistandard Young tableaux of content $\lambda^{op}$, where $\lambda^{op}$ is obtained from reversing the sequence $\lambda$. For example, if $T$ is as above then we have
$$P(T)=\ytableaushort{12457,369,8}, \qquad Q(T)= \ytableaushort{11223,233,3}.$$
We define $\fsh(T)$ to be the shape of $P(T)$. Note that $\fsh(T) \geq \sh(T)$ with respect to dominance order, and $\fsh(T)=\sh(T)$ if and only if $T$ is standard.
\ytableausetup{tabloids}

\subsection{Residues and intervals}
For $k \in \bZ$, we let $\mo{k}$ be the unique element in $\{1, 2, \ldots, n\}$ congruent to $k$ modulo $n$. For example, we have $\mo{-1} = n-1, \mo{0}=n,$ etc. For $a, b \in \bZ$, we define $[a,b] \colonequals \{x \in \bZ \mid a\leq x \leq b\}$. Similarly, for $a, b \in \{1, 2, \ldots, n\}$, we define $\pint{a,b} \subset \{1, 2, \ldots, n\}$ to be $\emptyset$ if $\mo{a}=\mo{b+1}$ and $\{\mo{a+x} \mid 0\leq x \leq \mo{b-a}\}$ otherwise.
For example, if $n=5$ then $\pint{1,5}=\pint{2,1} = \emptyset$, $\pint{3,3}=\{3\}$, and $\pint{4,2} = \{4, 5, 1, 2\}$. (Note in particular that $\emptyset=\pint{1,5} \neq[1,5]=\{1,2,3,4,5\}$. The reason for adopting such a convention for $\pint{\  ,\  }$ will be apparent in \ref{sec:defawg}.)

\subsection{Descents and Knuth moves} \label{sec:desknuth}
For $T \in \RSYT(\lambda)$ we define the (affine) descent set of $T$ to be $\ades(T) \colonequals \{i\in [1,n]\mid \mo{i} \textup{ lies in a strictly higher row of $T$ than } \mo{i+1}\}$ following \cite[Definition 3.4]{clp17}. Similarly, for $T\in \SYT(\lambda)$ we define the (finite) descent set of $T$ to be $\des(T) \colonequals \ades(T) -\{n\}$. For $T, T' \in \RSYT(\lambda)$, we say that $T'$ is obtained from $T$ by a Knuth move or $T$ and $T'$ are connected by a (single) Knuth move if $\ades(T)$ and $\ades(T')$ are not comparable and $T'$ is obtained from $T$ by interchanging $\mo{i}$ and $\mo{i+1}$ for some $\mo{i} \in [1,n]$ (and reordering entries in each row if necessary). 

\begin{rmk}
If $T, T' \in \SYT(\lambda)$, one may tempt to define a finite analogue, i.e. if $\des(T)$ and $\des(T')$ are not comparable and $T'$ is obtained from $T$ by interchanging $i$ and $i+1$ for some $i \in [1,n-1]$ (without reordering rows after). However, one can easily check that if $T, T' \in \SYT(\lambda)$ then these two notions are in fact equivalent, thus there is no need to differentiate affine and finite Knuth moves.
\end{rmk}

\section{$W$-graphs}
Here we recall the notion of $W$-graphs. Basic references are \cite{kl79} and \cite{ste08}.

\subsection{$I$-labeled graphs} Suppose for now that $W$ is a Coxeter group with the set of simple reflections $I$. We say that $\Gamma=(V, m, \tau)$ is an $I$-labeled graphs if
\begin{enumerate}
\item $m$ is a map $m: V\times V \rightarrow \bZ[q^{\pm\frac{1}{2}}]$.
\item $\tau$ is a map $\tau : V \rightarrow \cP(I)$, where $\cP(I)$ is the power set of $I$.
\item For each $v \in V$, $\{w \in V \mid m(v, w)\neq 0 \textup{ or } m(w,v)\neq 0\}$ is a finite set.
\end{enumerate}
Moreover, we say that $\Gamma$ is finite if $|V|<\infty$. Conventionally, if $m(u, v)\neq 0$ (resp. $m(u, v)=0$) for $u, v \in V$ then we say that there is an (directed) edge from $u$ to $v$ of weight $m(u, v)$ (resp. there is no edge from $u$ to $v$). In order to avoid confusion, we also write $m(u\ra v)$ instead of $m(u, v)$.

We say that $\Gamma'=(V', m', \tau')$ is a $I$-labeled subgraph (or simply subgraph) of $\Gamma$ if $V' \subset V$, $\tau'(v) = \tau(v)$ for $v\in V'$, and $m'(u\ra v)\in\{m(u\ra v),0\}$ for $u,v \in V'$. Furthermore if $m'(u\ra v) = m(u\ra v)$ for all $u,v \in V'$, then we say that $\Gamma'$ is a full subgraph of $\Gamma$.
\begin{rmk}
Note that our definition is weaker than that of \cite[p.347]{ste08} as we allow (locally finite but) infinite $W$-graphs. Indeed, a (locally finite) $W$-graph with infinite vertices will naturally appear in our paper when we discuss periodic $W$-graphs.
\end{rmk}

\subsection{$W$-graphs}\label{sec:defwgraph} Let $\cH_W$ be the Iwahori-Hecke algebra of $W$ over $\bZ[q^{\pm\frac{1}{2}}]$, which is a quotient of the braid group of $W$ with generators $T_i$ for $i\in I$ by quadratic relations $(T_i-q)(T_i+1)=0$. An $I$-labeled graph $\Gamma=(V,m,\tau)$ is called a $W$-graph if the formula
$$T_i(u) = \left\{
\begin{aligned} 
&qu& \textup{ if } i \not\in \tau(u)
\\&-u+q^{\frac{1}{2}}\sum_{v: i \not\in \tau(v)} m(u\ra v)v & \textup{ if } i \in \tau(u)
\end{aligned}\right.
$$
gives rise to a $\cH_W$-module structure on $\bigoplus_{v\in V}\bZ[q^{\pm\frac{1}{2}}]v$. Note that this is a transposed form compared to the original definition in \cite{kl79} and coincides with the one in \cite{ste08}. (Also see \cite[A.3]{lus97} for similar definition.)

\subsection{Reduced $I$-labeled graphs}
Suppose that $\Gamma = (V, m, \tau)$ is an $I$-labeled graph. We say that $\Gamma$ is reduced if $m(u\ra v)=0$ whenever $\tau(u) \subset \tau(v)$. This notion is motivated from the fact that the values $m(u\ra v)$ when $\tau(u) \subset \tau(v)$ do not appear in the above formula for $\Gamma$ being a $W$-graph. In this paper we only deal with reduced $I$-labeled graphs.

\subsection{Parabolic restriction of $I$-labeled graphs}\label{sec:parabres} For a subset $J\subset I$, the parabolic restriction of an $I$-labeled graph $\Gamma=(V, m, \tau)$, denoted $\Gamma\pres{J}=(V', m', \tau')$, is a $J$-labeled graph such that $V'=  V$, $\tau'(v)= \tau(v) \cap J$, and $m'(u\ra v)= m(u\ra v)$ if $\tau'(u) \not \subset \tau'(v)$ and $m'(u\ra v)=0$ otherwise.
Then $\Gamma'$ is clearly a $J$-labeled graph. Furthermore, if $\Gamma$ is a (reduced) $W$-graph, then it is easy to show that $\Gamma'$ is a (reduced) $W_J$-graph where $W_J \subset W$ is the parabolic subgroup generated by $J$. (cf. \cite[1.A]{ste08})

\subsection{(nb-)Admissible $I$-labeled graphs} 
For a $I$-labeled graph $\Gamma=(V, m, \tau)$, we say that $\Gamma$ is admissible if $\im m \subset \bN$, $m(u\ra v)=m(v\ra u)$ if $\tau(u)$ and $\tau(v)$ are not comparable, and $\Gamma$ is bipartite.
However, in our case it is crucial to consider $W$-graphs which are not necessarily bipartite. We say that $\Gamma$ is nb-admissible if it is admissible but possibly not bipartite. Later we will see that dropping this assumption does not cause any problem for our argument.

\subsection{Simple underlying graph} For an $I$-labeled graph $\Gamma=(V, m, \tau)$, we define its simple underlying graph $U(\Gamma)=(V', m', \tau')$ to be an $I$-labeled graph such that $V'=V, \tau'=\tau$, and $m'(u\ra v) =m'(v\ra u)=1$ if $m(u\ra v)= m(v\ra u)=1$ and $m'(u\ra v) =m'(v\ra u)=0$ otherwise. Note that $U(\Gamma)$ is canonically a subgraph of $\Gamma$ obtained by removing ``directed'' and ``non-simple'' edges. Furthermore, $U(\Gamma)$ is always a simple ($I$-labeled) graph.

\subsection{(nb-)Admissible $W$-graphs and Stembridge's theorem}
For simplicity, from now on we assume that $W$ is simply-laced. For a $W$-graph $\Gamma=(V, m, \tau)$, we introduce four combinatorial rules that it should satisfy.
\begin{enumerate}[label=\arabic*., leftmargin=*]
\item The Compatibility Rule. If $m(u\ra v)\neq 0$ for $u, v \in V$, then any $i\in \tau(u)-\tau(v)$ and any $j\in \tau(v)-\tau(u)$ are adjacent in the Dynkin diagram of $W$.
\item The Simplicity Rule. If $m(u\ra v) \neq 0$ for $u,v \in V$, then either [$\tau(u)\supset \tau(v)$ and $m(v\ra w)=0$] or [$\tau(u)$ and $\tau(v)$ are not comparable, and $m(u\ra v)=m(v\ra u)=1$.]
\item The Bonding Rule. For any $i,j \in I$ adjacent in the Dynkin diagram of $W$, if $u\in V$ satisfies $i \in \tau(u)$ and $j\not\in \tau(u)$ then there exists a unique $v\in V$ such that $i\not\in \tau(v), j\in \tau(v)$, $m(u\ra v)\neq0$, and $m(v\ra u)\neq 0$.
\item The Polygon Rule. For $i,j \in I$, we define $V_{i/j}=\{v \in V \mid i \in \tau(v), j \not\in \tau(v)\}$. For $u, v \in V$ such that $i,j \in \tau(u)$ and $i,j \not \in \tau(v)$, set
\begin{align*}
N^2_{ij}(\Gamma; u,v)&=\sum_{w \in V_{i/j}} m(u\ra w)m(w\ra v)
\\N^3_{ij}(\Gamma; u,v)&=\sum_{w_1 \in V_{i/j}, w_2 \in V_{j/i}} m(u\ra w_1)m(w_1\ra w_2)m(w_2\ra v)
\end{align*}
(These sums are well-defined due to local finiteness assumption.)
Then we have $N^r_{ij}(\Gamma;u,v)=N^r_{ji}(\Gamma;u,v)$ for such $u,v \in V$ and $i,j \in J$. Here $r=2$ or $r=3$, and the latter case is only considered when $i$ and $j$ are adjacent in the Dynkin diagram of $W$.
\end{enumerate}
The main theorem of \cite{ste08} is that these rules characterize the combinatorial properties of admissible $I$-labeled graphs being a $W$-graph. Here we generalize his theorem slightly as follows.
\begin{thm}[See {\cite[Theorem 4.9]{ste08}}]  \label{thm:combrules} Let $\Gamma$ be an nb-admissible (reduced) $I$-labeled graph. Then $\Gamma$ is a $W$-graph if and only if it satisfies the four combinatorial rules above.
\end{thm}
\begin{proof} Indeed, the original proof of Stembridge does not use the bipartition assumption, thus his proof is directly applied to our case.
\end{proof}

\subsection{Cells and simple components} \label{sec:cellsimple} For an $I$-labeled graph $\Gamma=(V,m,\tau)$, we define its cells to be its strongly connected components, which is naturally a full subgraph of $\Gamma$. Also, its simple component is defined to be a full subgraph of $\Gamma$ whose simple underlying graph is connected. Note that these two notions do not coincide for  $I$-labeled graphs; each simple component is a subgraph of a cell but not vice versa in general.

\section{$\affSn$-graph $\awg_\lambda$ for two-row partitions}
\subsection{Definition of $\awg_\lambda$} \label{sec:defawg}
From now on, we set $W=\affSn$ and $I=\{s_0=s_n, s_1, \ldots, s_{n-1}\}$. We also identify $I$-labeled graphs with $[1, n]$-labeled graphs in an obvious manner. Define the $[1,n]$-labeled graph $\awg_\lambda=(V, m, \tau)$ for the two-row partition $\lambda$ as follows. Set $V=\RSYT(\lambda)$ and $\tau=\ades$ (see \ref{sec:desknuth} for the definition of $\ades$). For any $s, t\in \RSYT(\lambda)$, $m(s\ra t)$ is equal to either 0 and 1. If $m(s\ra t)=1$, then we say that there is a move from (the source) $s$ to (the target) $t$, which falls into one of the following cases.
\begin{enumerate}[label=\arabic*., leftmargin=*]
\item (Move of the first kind) $t$ is obtained from $s$ by interchanging $\mo{i}$ and $\mo{i+1}$ when $\mo{i} \in s^1$ and $\mo{i+1} \in s^2$, i.e.
$$s=\ytableaushort{\cdots \ {\mo{i}} \ \cdots, \cdots \  {\mo{i+1}} \ \cdots} \quad \rightarrow \quad t=\ytableaushort{\cdots \ {\mo{i+1}} \ \cdots, \cdots \  {\mo{i}} \ \cdots}.$$
This move is denoted by $\mo{i}\intch \mo{i+1}$ or $\mo{i+1} \rintch \mo{i}$.
\item (Move of the second kind) $t$ is obtained by interchanging $\mo{i}$ and $\mo{j}$ when $\mo{i} \in v^2, \mo{j} \in v^1$, and $\mo{i}\neq \mo{j+1}$ (it becomes a move of the first kind if $\mo{i}=\mo{j+1}$), i.e.
$$s=\ytableaushort{\cdots {\mo{j}}  \cdots, \cdots  {\mo{i}} \cdots} \quad \rightarrow \quad t=\ytableaushort{\cdots {\mo{i}} \cdots, \cdots {\mo{j}} \cdots}.$$
This move occurs if and only if the following conditions are satisfied:
\begin{enumerate}[label=(\alph*)]
\item $\mo{j-i}$ is odd. 
\item $\mo{i+1} \in s^1$ and $\mo{j-1}\in s^2$.
\item Either $\mo{i-1}\in s^1$ or $\mo{j+1}\in s^2$.
\item $\#\{\alpha\in s^2 \mid \alpha \in \pint{\mo{j-1-2k}, \mo{j-2}}\} \geq k$ for $k\in \{1, 2, \ldots, \frac{\mo{j-i}-3}{2}\}$.
\item $\#\{\alpha \in s^2 \mid \alpha \in \pint{\mo{i+2}, \mo{j-2}}\} =\frac{\mo{j-i}-3}{2}$ when $\mo{j}\neq \mo{i+1}$.
\end{enumerate}
This move is denoted by $\mo{j}\intch \mo{i}$ or $\mo{i} \rintch \mo{j}$.
\end{enumerate}
\begin{rmk} When $\mo{j-i}=3$, then $\pint{\mo{i+2}, \mo{j-2}} = \pint{\mo{i+2}, \mo{i+1}}$ which is $\emptyset$ rather than $[1,n]$ in our convention. In such a case the condition (e) is trivially satisfied. This is the reason to set $\pint{\mo{a+1}, \mo{a}}=\emptyset$ for any $a\in [1,n]$; otherwise $\mo{j-i}=3$ case should be handled in a separate manner.
\end{rmk}

\ytableausetup{boxsize=1em, notabloids}
\begin{figure}[htbp] 
\centering
\begin{tikzpicture}[scale=1.5]

	\node[draw, tl] (o1) at (90+144:3) {\ytableaushort{12{\dt{3}},45}};
	\node[draw] (o2) at (18+144:3) {\ytableaushort{23{\dt{4}},15}};
	\node[draw] (o3) at (306+144:3) {\ytableaushort{34{\dt{5}},12}};
	\node[draw] (o4) at (234+144:3) {\ytableaushort{{\dt{1}}45,23}};
	\node[draw, tl] (o5) at (162+144:3) {\ytableaushort{1{\dt{2}}5,34}};
	
	\node[draw, tl] (i1) at (90+144:1.5) {\ytableaushort{1{\dt{2}}{\dt{4}},35}};
	\node[draw] (i2) at (18+144:1.5) {\ytableaushort{2{\dt{3}}{\dt{5}},14}};
	\node[draw, tl] (i3) at (306+144:1.5) {\ytableaushort{{\dt{1}}3{\dt{4}},25}};
	\node[draw] (i4) at (234+144:1.5) {\ytableaushort{{\dt{2}}4{\dt{5}},13}};
	\node[draw, tl] (i5) at (162+144:1.5) {\ytableaushort{{\dt{1}}{\dt{3}}5,24}};
	
	\draw[tl] (o1) -- (i1);\draw[] (o2) -- (i2);\draw[] (o3) -- (i3);
	\draw[] (o4) -- (i4);\draw[tl] (o5) -- (i5);
	
	\draw[tl] (i1) -- (i3);\draw[] (i2) -- (i4);\draw[tl] (i3) -- (i5);
	\draw[] (i4) -- (i1);\draw[] (i5) -- (i2);
	
	\draw[arr] (i1) -- (o2);	\draw[arr, tl] (i1) -- (o5);
	\draw[arr] (i2) -- (o3);	\draw[arr] (i2) -- (o1);
	\draw[arr] (i3) -- (o4);	\draw[arr] (i3) -- (o2);
	\draw[arr] (i4) -- (o5);	\draw[arr] (i4) -- (o3);
	\draw[arr, tl] (i5) -- (o1);	\draw[arr] (i5) -- (o4);	

\end{tikzpicture}
\caption{$\affS{5}$-graph $\awg_{(3,2)}$}
\label{fig:32}
\end{figure}

\begin{example}  \label{ex:32}
Figure \ref{fig:32} illustrates the $\affS{5}$-graph $\awg_{(3,2)}$. Here, $\ades(v)$ for each $v\in \awg_{(3,2)}$ is given by bold numbers on the first row of $v$. Also bold bordered vertices and bold edges are the ``standard part'' $\Gamma_{(3,2)}$ which we define later. For example consider its vertex $s=\ytableaushort{{{2}}4{{5}},13}$.  Then by applying a move of the first kind for $i=2$ we obtain an arrow pointing to vertex $\ytableaushort{345,12}$. We can also apply a move of the second kind for $i=1$, $j=4$. Indeed, $\mo{4-1} = 3$ is odd, $\mo 3 \in s^2$, $\mo 2 \in s^1$, $\mo{1-1} = \mo 5 \in s^1$, and the last two conditions are trivially true because $\mo{j-i}=3$. As a result, we get an arrow from $s$ to the vertex $\ytableaushort{{{1}}2{{5}},34}$.
\end{example}
\ytableausetup{boxsize=normal, tabloids}

\ytableausetup{boxsize=1em, notabloids}
\begin{figure}[htbp]
\centering
\begin{tikzpicture}[scale=1.5]

	\node[draw, tl] (t1) at (120+90:4) {\ytableaushort{12{\dt{3}}{\dt{5}},46}};
	\node[draw, tl] (t2) at (180+90:4) {\ytableaushort{1{\dt{2}}{\dt{4}}6,35}};
	\node[draw, tl] (t3) at (240+90:4) {\ytableaushort{{\dt{1}}{\dt{3}}56,24}};
	\node[draw] (t4) at (300+90:4) {\ytableaushort{{\dt{2}}45{\dt{6}},13}};
	\node[draw, tl] (t5) at (0+90:4) {\ytableaushort{{\dt{1}}34{\dt{5}},26}};
	\node[draw] (t6) at (60+90:4) {\ytableaushort{23{\dt{4}}{\dt{6}},15}};
	
	\node[draw, tl] (o1) at (90+90:4) {\ytableaushort{123{\dt{4}},56}};
	\node[draw, tl] (o2) at (150+90:4) {\ytableaushort{12{\dt{3}}6,45}};
	\node[draw, tl] (o3) at (210+90:4) {\ytableaushort{1{\dt{2}}56,34}};
	\node[draw] (o4) at (270+90:4)  {\ytableaushort{{\dt{1}}456,23}};
	\node[draw] (o5) at (330+90:4) {\ytableaushort{345{\dt{6}},12}};
	\node[draw] (o6) at (30+90:4) {\ytableaushort{234{\dt{5}},16}};
	
	\node[draw, tl] (t7) at (55+90:2.2) {\ytableaushort{{\dt{1}}3{\dt{4}}6,25}};
	\node[draw] (t8) at (180+90:1.8) {\ytableaushort{2{\dt{3}}5{\dt{6}},14}};
	\node[draw, tl] (t9) at (305+90:2.2) {\ytableaushort{1{\dt{2}}4{\dt{5}},36}};
	
%
	\draw[tl] (o1) -- (t1);\draw[tl] (o2) -- (t2);\draw[tl] (o3) -- (t3);
	\draw[] (o4) -- (t4);\draw[] (o5) -- (t5);\draw[] (o6) -- (t6);

	\draw[tl] (t7) -- (t2);\draw[tl] (t7) -- (t3);\draw[tl] (t7) -- (t5); \draw[] (t7) -- (t6);
	\draw[] (t8) -- (t1);\draw[] (t8) -- (t3);\draw[] (t8) -- (t4); \draw[] (t8) -- (t6);
	\draw[tl] (t9) -- (t1);\draw[tl] (t9) -- (t2);\draw[] (t9) -- (t4); \draw[tl] (t9) -- (t5);
	
	\draw[arr,tl] (t1) -- (o2);	\draw[arr] (t1) -- (o6);
	\draw[arr,tl] (t2) -- (o3);	\draw[arr,tl] (t2) -- (o1);
	\draw[arr] (t3) -- (o4);	\draw[arr,tl] (t3) -- (o2);
	\draw[arr] (t4) -- (o5);	\draw[arr] (t4) -- (o3);
	\draw[arr] (t5) -- (o6);	\draw[arr] (t5) -- (o4);
	\draw[arr] (t6) -- (o1);	\draw[arr] (t6) -- (o5);		

\end{tikzpicture}
\caption{$\affS{6}$-graph $\awg_{(4,2)}$}
\label{fig:42}
\end{figure}

\begin{example}  \label{ex:42}
Figure \ref{fig:42} illustrates the $\affS{6}$-graph $\awg_{(4,2)}$, similarly to the previous example.
\end{example}
\ytableausetup{boxsize=normal, tabloids}

\ytableausetup{boxsize=1em, notabloids}
\begin{figure}[htbp]
\centering
\begin{tikzpicture}[scale=1.2]

	\node[draw, tl] (1o1) at (0,3) {\ytableaushort{12{\dt{3}},456}};
	\node[draw] (1o2) at (3^.5/2*3,-1/2*3)  {\ytableaushort{34{\dt{5}},126}};
	\node[draw] (1o3) at (-3^.5/2*3,-1/2*3)  {\ytableaushort{{\dt{1}}56,234}};
	
	\node[draw, tl] (1i1) at (0,1.5) {\ytableaushort{1{\dt{2}}{\dt{4}},356}};
	\node[draw] (1i2) at (3^.5/2*1.5,-1/2*1.5)  {\ytableaushort{3{\dt{4}}{\dt{6}},125}};
	\node[draw] (1i3) at (-3^.5/2*1.5,-1/2*1.5)  {\ytableaushort{{\dt{2}}5{\dt{6}},134}};	

	\node[draw, tl] (1m1) at (3^.5/2*2,1/2*2){\ytableaushort{{\dt{1}}3{\dt{4}},256}};
	\node[draw] (1m2) at (0,-2){\ytableaushort{{\dt{3}}5{\dt{6}},124}};
	\node[draw, tl] (1m3) at (-3^.5/2*2,1/2*2){\ytableaushort{1{\dt{2}}{\dt{5}},346}};	
	
	\node[draw, tl] (1c) at (0,0) {\ytableaushort{{\dt{1}}{\dt{3}}{\dt{5}},246}};	
	
	\node[draw] (2o1) at (0+6,3-3) {\ytableaushort{23{\dt{4}},156}};
	\node[draw] (2o2) at (3^.5/2*3+6,-1/2*3-3)  {\ytableaushort{45{\dt{6}},123}};
	\node[draw] (2o3) at (-3^.5/2*3+6,-1/2*3-3)  {\ytableaushort{1{\dt{2}}6,345}};
	
	\node[draw] (2i1) at (0+6,1.5-3) {\ytableaushort{2{\dt{3}}{\dt{5}},146}};
	\node[draw] (2i2) at (3^.5/2*1.5+6,-1/2*1.5-3)  {\ytableaushort{{\dt{1}}4{\dt{5}},236}};
	\node[draw] (2i3) at (-3^.5/2*1.5+6,-1/2*1.5-3)  {\ytableaushort{{\dt{1}}{\dt{3}}6,245}};	

	\node[draw] (2m1) at (3^.5/2*2+6,1/2*2-3){\ytableaushort{{\dt{2}}4{\dt{5}},136}};
	\node[draw] (2m2) at (0+6,-2-3){\ytableaushort{{\dt{1}}{\dt{4}}6,235}};
	\node[draw] (2m3) at (-3^.5/2*2+6,1/2*2-3){\ytableaushort{2{\dt{3}}{\dt{6}},145}};	
	
	\node[draw] (2c) at (0+6,0-3) {\ytableaushort{{\dt{2}}{\dt{4}}{\dt{6}},135}};	
	
	\draw[tl] (1o1) -- (1i1);  \draw[] (1o2) -- (1i2); \draw[] (1o3) -- (1i3);
	\draw[tl] (1i1) -- (1m1);  \draw[] (1i2) -- (1m2); \draw[] (1i3) -- (1m3);
	\draw[tl] (1i1) -- (1m3);  \draw[] (1i2) -- (1m1); \draw[] (1i3) -- (1m2);
	\draw[tl] (1c) -- (1m1);  \draw[] (1c) -- (1m2); \draw[tl] (1c) -- (1m3);	
	
	\draw[] (2o1) -- (2i1);  \draw[] (2o2) -- (2i2); \draw[] (2o3) -- (2i3);
	\draw[] (2i1) -- (2m1);  \draw[] (2i2) -- (2m2); \draw[] (2i3) -- (2m3);
	\draw[] (2i1) -- (2m3);  \draw[] (2i2) -- (2m1); \draw[] (2i3) -- (2m2);
	\draw[] (2c) -- (2m1);  \draw[] (2c) -- (2m2); \draw[] (2c) -- (2m3);	
	
	\draw[arr] (1m3) to [out=250, in=180] (2o3); \draw[arr] (1m1) to [out=0, in=180] (2o1); \draw[arr] (1m2) to [out=0, in=170]  (2o2);
	\draw[arr] (2m3) to [out=90, in=0] (1o1); \draw[arr] (2m1) to [out=130, in=20] (1o2); \draw[arr] (2m2) to [out=180, in=270] (1o3);
	\draw[arr,tl] (1c) to [out=50,in=-50] (1o1); \draw[arr] (1c) to [out=50-120,in=-50-120]  (1o2); \draw[arr] (1c) to [out=50-240,in=-50-240]  (1o3); 
    	\draw[arr] (1c) to [out=0,in=180] (2i1); \draw[arr] (1c) to [out=0,in=90]  (2i2); \draw[arr] (1c) to [out=-70,in=150]  (2i3); 
	\draw[arr] (2c) to [out=50,in=-50] (2o1); \draw[arr] (2c) to [out=50-120,in=-50-120]  (2o2); \draw[arr] (2c) to [out=50-240,in=-50-240]  (2o3); 
    	\draw[arr] (2c) to [out=120,in=-50] (1i1); \draw[arr] (2c) to [out=170,in=-90]  (1i2); \draw[arr] (2c) to [out=170,in=-90]  (1i3); 
	
%
%
%

\end{tikzpicture}
\caption{$\affS{6}$-graph $\awg_{(3,3)}$}
\label{fig:33}
\end{figure}	

\begin{example} \label{ex:33}
Figure \ref{fig:33} illustrates the $\affS{6}$-graph $\awg_{(3,3)}$, similarly to the previous examples. Note that it is strongly connected (or consists of a single cell) but it contains two simple components. (cf. \ref{sec:cellsimple})
\end{example}
\ytableausetup{boxsize=normal, tabloids}

\subsection{Properties of $\awg_\lambda$}
Let us describe some properties of $\awg_\lambda$. First it is helpful to understand how moves change $\tau$-values in each case as described in the lemma below.
\begin{lem}
\begin{enumerate}
\item If $v\xrightarrow{\mo{i}\intch \mo{i+1}} w$ is a move of the first kind, then $\ades(v) - \ades(w)=\{\mo{i}\}$ and $\ades(w) - \ades(v) \subset \{\mo{i-1}, \mo{i+1}\}$.
\item If $v\xrightarrow{\mo{i}\intch \mo{j}} w$ is a move of the second kind, then 
\begin{enumerate}
\item $\ades(v) - \ades(w)$ is equal to one of $\{\mo{i-1}\}, \{\mo{j}\},$ or $\{\mo{i-1}, \mo{j}\}$.
\item if $\mo{j}=\mo{i+1}$, then $\ades(w)-\ades(v) = \{\mo{i}=\mo{j-1}\}$.
\item if $\mo{j}\neq\mo{i+1}$, then $\ades(w)-\ades(v)=\emptyset$.
\end{enumerate}
\end{enumerate}
\end{lem}
\begin{proof} It is clear from the definition of moves.
\end{proof}

\begin{lem} $\awg_\lambda=(\RSYT(\lambda), m, \ades)$ is reduced and nb-admissible.
\end{lem}
\begin{proof} Suppose first that $m(u\ra v)\neq 0$ for some $u,v \in \RSYT(\lambda)$, i.e. there is a move form $u$ to $v$. If it is of the first kind, then $i \in \ades(u) -\ades(v)$. Otherwise, either $\mo{i-1} \in \ades(u)-\ades(v)$ or $\mo{j} \in \ades(u)-\ades(v)$. In either case, we have $\ades(u) \not\subset \ades(v)$. This proves that $\awg_\lambda$ is reduced.

On the other hand, it is clear that $\im m \in \{0,1\}$. Now suppose that $\ades(u)$ and $\ades(v)$ are incomparable for some $u,v \in \RSYT(\lambda)$. If there is no move from either from $u$ to $v$ or from $v$ to $u$, then clearly $m(u\ra v)=m(v\ra u)=0$. Otherwise, without loss of generality we may assume that there is a move from $u$ to $v$. If this is of the first kind, say $\mo{i} \intch\mo{i+1}$, then one can easily check that there is a move of the second kind $\mo{i} \rintch\mo{i+1}$ from $v$ to $u$. (The only nontrivial condition is that either $\mo{i-1} \in v^1$ or $\mo{i+2}\in v^2$, which is true since $\ades(v) \not\subset \ades(u)$.) If the move from $u$ ot $v$ is of the second kind, then $\ades(u) \not\supset \ades(v)$ if and only if it is $\mo{i} \rintch \mo{i+1}$ for some $i \in [1,n]$. Thus there is a move of the first kind from $v$ to $u$ as well. In sum, we have $m(u\ra v)=1$ if and only if $m(v\ra u)=1$. This proves that $\awg_\lambda$ is nb-admissible.
\end{proof}
\begin{rmk} The graph $\awg_\lambda$ is not in general bipartite. For example, the $\affS{5}$-graph $\awg_{(3,2)}$ (see Figure \ref{fig:32}) cannot be bipartite even after removing all directed edges because a cycle of length 5 (a ``star'' in the figure) is embedded into $\awg_{(3,2)}$. 
\end{rmk}

Recall the cyclic shift element $\shift \in \extSn$. We consider the action of $\shift$ on $\RSYT(\lambda)$ by replacing each $i$ with $\mo{i+1}$ and reordering entries of each row if necessary. 
\begin{lem} The action of $\shift$ on $\RSYT(\lambda)$ induces that on $\awg_\lambda$.
\end{lem}
\begin{proof} It is clear that $\ades(\shift(v)) = \shift(\ades(v))$. Furthermore, it is easy to check that the description of moves on $\awg_\lambda$ is also ``invariant under $\shift$'', i.e. we have $m(u\ra v)=m(\shift(u)\ra \shift(v))$.
\end{proof}
\begin{example} In Figure \ref{fig:32} $\shift$ acts as a (clockwise) rotation by $72^\circ$. Similarly, in Figure \ref{fig:42} $\shift$ acts a s (clockwise) rotation by $60^\circ$ on the outer part and by $120^\circ$ on the inner part. On the other hand, in Figure \ref{fig:33} $\shift$ swaps two simple components and $\shift^2$ rotates each component by $120^\circ$.
\end{example}
\begin{rmk} It can be proved that $\awg_\lambda$ is also invariant under the affine evacuation defined in \cite{cfkly}, but this fact will not be used in this paper.
\end{rmk}

It is desirable to understand $U(\awg_\lambda)$ in terms of combinatorics of Young tableaux. Let $\adeg_\lambda=(V', m',\tau')$ be the Kazhdan-Lusztig affine dual equivalence graph of shape $\lambda$ as in \cite[Definition 3.21]{clp17}. (Here we use the adjective ``affine'' to differentiate it from the ``finite'' one $\fdeg_\lambda$ defined later.) It is defined as $V'=\RSYT(\lambda)$, $\tau'=\ades$, and for $u,v \in \RSYT(\lambda)$, $m'(u\ra v)=m'(v\ra u)=1$ if there exists a Knuth move connecting $u$ and $v$ and $m'(u\ra v)=m'(v\ra u)=0$ otherwise. (See \ref{sec:desknuth} for the definition of Knuth moves.)  Then we have
\begin{prop} $U(\awg_\lambda)=\adeg_\lambda$ as $[1,n]$-labeled graphs.
\end{prop}
\begin{proof} It is enough to show that $m'(u\ra v)=m(u\ra v)$ if $\ades(u) \not\supset\ades(v)$. First suppose that there is a move from $u$ to $v$, i.e. $m(u\ra v)=1$. As $\ades(u) \not\supset\ades(v)$, this move should be either $\mo{i} \intch\mo{i+1}$ or $\mo{i} \rintch\mo{i+1}$ for some $i\in [1,n]$. In any case, one may check that this is a Knuth move thus $m'(u\ra v)=1$ as well. The other direction is proved similarly.
\end{proof}

\subsection{$\awg_\lambda$ is a $\affSn$-graph}
We are ready to state the first main theorem of this paper.
\begin{thm} $\awg_\lambda$ is a $\affSn$-graph.
\end{thm}
To this end, we use Theorem \ref{thm:combrules}; our proof is purely combinatorial. Firstly, three out of four combinatorial rules of Stembridge are proved easily.
\begin{lem} $\awg_\lambda$ satisfies the Compatibility Rule, the Simplicity Rule, and the Bonding Rule.
\end{lem}
\begin{proof} The first two rules follow directly from the description of moves. Also $\awg_\lambda$ satisfies the Bonding Rule if and only if $U(\awg_\lambda)=\adeg_\lambda$ does, which follows from \cite[Proposition 3.5]{cfkly}.
\end{proof}
Thus it remains to show that $\awg_\lambda$ satisfies the Polygon Rule, which is the most technical part of our proof. First note that it is not possible to have $\mo{i}, \mo{i+1} \in \ades(v)$ for any $i \in [1,n]$ and any $v\in \RSYT(\lambda)$ since $\lambda$ is assumed to be a two-row partition. Thus we only need to show that $N^r_{i,j}(\awg_\lambda; v,w) = N^r_{j,i}(\awg_\lambda; v,w)$ where $i$ and $j$ are not adjacent in the Dynkin diagram of $\affSn$ (i.e. $\mo{i} \not\in \{\mo{j-1}, \mo{j}, \mo{j+1}\}$) and $r=2$. In such cases, we usually omit $\awg_\lambda$ and the superscript $r=2$ from the notations, and simply write $N_{i,j}(v,w)$ and $N_{j,i}(v,w)$ instead.

Furthermore, if there is a move from $s$ to $t$ then it swaps an entry in $s^1$ and another in $s^2$, which means $s$ and $t$ differ by two elements. Since if $N_{i,j}(v,w)\neq 0$ then $w$ is obtained from $v$ by two sequential moves, it follows that we only need to check the Polygon Rule when $v$ and $w$ differ by either two or four elements. In the next two sections we verify the Polygon Rule $N_{i,j}(v,w) = N_{j,i}(v,w)$ for such $v$ and $w$ case-by-case.

\section{Case 1: $v$ and $w$ differ by two elements}
Here, we check the polygon rule $N_{i,j}(v,w) = N_{j,i}(v,w)$ for a fixed $v=(v^1, v^2)$ $w=(w^1, w^2)$ and various $i,j \in [1, n]$ where $v$ and $w$ only differ by two elements. Let us denote by $a\in[1,n]$ the unique element in $v^1-w^1=w^2-v^2$ and by $b\in[1,n]$ the unique element in $v^2-w^2=w^1-v^1$.
In other words, we are in the following situation:
$$v=\ytableaushort{\cdots a \cdots,\cdots b \cdots} \quad \rightsquigarrow \ u \ \rightsquigarrow \quad w=\ytableaushort{\cdots b \cdots,\cdots a \cdots} $$

Note that $N_{i,j}(v,w) \neq 0$ only when $i,j \in [1,n]$ satisfy $i,j \in \ades(v)$ and $i,j \not\in \ades(w)$. However, it is only possible when $a\neq \mo{b-1}$ and $\{i,j\} = \{a,\mo{b-1}\}$, which we assume from now on. Moreover, it also requires that $\mo{a+1} \in v^2$ and $\mo{b-1} \in v^1$. Therefore, it suffices only to consider the following case:
$$v=\ytableaushort{\cdots a {\scriptstyle \mo{b-1}} \cdots,\cdots {\scriptstyle \mo{a+1}} b \cdots}\quad \rightsquigarrow \ u \ \rightsquigarrow \quad  w=\ytableaushort{\cdots {\scriptstyle \mo{b-1}}  b \cdots,\cdots a  {\scriptstyle \mo{a+1}} \cdots} $$
We have following two possibilities to obtain $w$ from $v$ in two steps.
\begin{enumerate}[label=$\bullet$]
\item For some element $x\in v^1-\{a\}$, interchange $x$ and $b$ and then $x$ and $a$, i.e. $x\intch b$ and $a\intch x$.
\item For some element $y \in v^2-\{b\}$, interchange $y$ and $a$ and then $y$ and $b$, i.e. $a\intch y$ and $y\intch b$.
\end{enumerate}

From now on we show that the polygon rule holds for $(v,w)$, i.e. $N_{i,j}(v,w) = N_{j,i}(v,w)$ where $\{i,j\} = \{a,\mo{b-1}\}$. We divide it into two cases, depending on whether $b=\mo{a-1}$ (\ref{b=a-1case}) or not (\ref{b!=a-1case}).

\subsection{$b=\mo{a-1}$ case}\label{b=a-1case} If $b=\mo{a-1}$, then $\{i,j\} = \{\mo{a-2}, a\}$ and we are in the following situation:
$$v=\ytableaushort{\cdots {\scriptstyle \mo{a-2}} a  \cdots,\cdots {\scriptstyle \mo{a-1}} {\scriptstyle \mo{a+1}}\cdots} \quad \rightsquigarrow \ u \ \rightsquigarrow \quad w=\ytableaushort{\cdots {\scriptstyle \mo{a-2}} {\scriptstyle \mo{a-1}} \cdots,\cdots a  {\scriptstyle \mo{a+1}} \cdots} $$
By applying cyclic shift $\shift$, we may assume that $a=3$. Thus we have:
$$v=\ytableaushort{\cdots 1 3  \cdots,\cdots 2 4 \cdots} \quad \rightsquigarrow \ u \ \rightsquigarrow  \quad w=\ytableaushort{\cdots 1 2 \cdots,\cdots 3 4 \cdots} $$
where $a=3, b=2$, and $\{i,j\}=\{1,3\}$. Here we consider two possibilities of two-step moves mentioned above, i.e. $x\intch 2$ and $3\intch x$ for some $x\in v^1-\{3\}$ or $3\intch y$ and $y\intch 2$ for some $y \in v^2-\{2\}$.

First, we consider the two-step move which performs $x\intch2$ and then $3 \intch x$ for some $x \in v^1- \{3\}$. If $x=1$, then we have:
$$v=\ytableaushort{\cdots 1 3 \cdots,\cdots 2 4 \cdots} \quad \rightsquigarrow \quad u=\ytableaushort{\cdots 2 3 \cdots,\cdots 1 4 \cdots} \quad \rightsquigarrow \quad w=\ytableaushort{\cdots 1 2 \cdots,\cdots 3 4 \cdots} $$
Otherwise, we have:
$$v=\ytableaushort{\cdots 1 3 x\cdots,\cdots 2 4 \cdots\cdots} \quad \rightsquigarrow \quad u=\ytableaushort{\cdots 1 2 3 \cdots,\cdots 4 x \cdots \cdots} \quad \rightsquigarrow \quad w=\ytableaushort{\cdots 1 2 x \cdots,\cdots 3 4 \cdots\cdots} $$
However, the second move $3 \intch x$ violates the condition (b) since $2\in u^1$.

This time we consider the move which performs $3\intch y$ and then $y\intch2$ for some $y \in v^2 - \{2\}$. If $y=4$, then we have:
$$v=\ytableaushort{\cdots 1 3 \cdots,\cdots 2 4 \cdots} \quad \rightsquigarrow \quad u=\ytableaushort{\cdots 1 4 \cdots,\cdots 2 3 \cdots} \quad \rightsquigarrow \quad w=\ytableaushort{\cdots 1 2 \cdots,\cdots 3 4 \cdots} $$
Otherwise, we have:
$$v=\ytableaushort{\cdots 1 3 \cdots \cdots,\cdots 2 4 y\cdots} \quad \rightsquigarrow \quad u=\ytableaushort{\cdots 1 y \cdots \cdots,\cdots 2 3 4 \cdots } \quad \rightsquigarrow \quad w=\ytableaushort{\cdots 1 2 \cdots \cdots,\cdots 3 4 y \cdots} $$
However, the second move $y\intch2$ violates the condition (b) since $3 \in u^2$.

Therefore, we conclude that $N_{i,j}(v,w) = N_{j,i}(v,w)=0$ and the polygon rule is trivially satisfied for $(v,w)$.

\subsection{$b\neq \mo{a-1}$ case} \label{b!=a-1case} Now let us assume that $b\neq \mo{a-1}$. By applying cyclic shift $\shift$ if necessary, we may assume that $a=1$, which implies that $4\leq b<n$. Thus, $\{i,j\} = \{1, b-1\}$ and we are in the following situation:
$$v=\ytableaushort{\cdots 1 {\scriptstyle b-1} \cdots,\cdots 2 b \cdots} \quad \rightsquigarrow \ u \ \rightsquigarrow \quad w=\ytableaushort{\cdots {\scriptstyle b-1}  b \cdots,\cdots 1 2 \cdots} $$
From now on we divide all the possibilities of two-step moves into the following four cases:
\begin{enumerate}[label=$\bullet$]
\item $(b-1)\intch b$ and $1\intch (b-1)$ (\ref{x=b-1case})
\item $a\intch 2$ and $2\intch b$ (\ref{y=2case})
\item $x\intch b$ and $1\intch x$ for some $x\neq b-1$ (\ref{x!=b-1case})
\item $a\intch y$ and $y\intch b$ for some $y\neq 2$ (\ref{y!=2case})
\end{enumerate}

\subsubsection{$(b-1)\intch b$ and $1\intch (b-1)$ case} \label{x=b-1case} First consider the case when we perform $(b-1)\intch b$ first and then $1\intch(b-1)$. It looks like:
$$v=\ytableaushort{\cdots 1 {\scriptstyle b-1} \cdots,\cdots 2 b \cdots} \quad \rightsquigarrow \quad u=\ytableaushort{\cdots 1 b \cdots,\cdots 2 {\scriptstyle b-1} \cdots} \quad \rightsquigarrow \quad w=\ytableaushort{\cdots {\scriptstyle b-1}  b \cdots,\cdots 1  2 \cdots} $$
Note that $b-1\neq 2$ since they are in different rows in $v$. Thus the move $1\intch(b-1)$ is of the second kind and the following conditions in \ref{sec:defawg} are imposed:
\begin{enumerate}[label=(\alph*)]
\item $\mo{1-(b-1)} = \mo{2-b}$ is odd, i.e. $n-b$ is odd
\item ($b \in u^1$ and) $n\in u^2$, i.e. $n \in v^2$ (note that $n \neq \mo{b-1}, b$)
\item (the third condition is satisfied since $2\in u^2$)
\item $\#(u^2 \cap [n-2k, n-1]) \geq k$ for $k\in \{1, 2, \ldots, \frac{n-b-1}{2}\}$, or equivalently
 $\#(v^2 \cap [n-2k, n-1]) \geq k$ for $k\in \{1, 2, \ldots, \frac{n-b-1}{2}\}$
\item $\#(u^2 \cap [b+1, n-1]) =\frac{n-b-1}{2}$ (note that $b = \mo{(b-1)+1} \neq 1$), or equivalently $\#(v^2 \cap [b+1, n-1]) =\frac{n-b-1}{2}$
\end{enumerate}
By part (b), we have:
$$v=\ytableaushort{\cdots 1 {\scriptstyle b-1} \cdots \cdots,\cdots 2 b n\cdots} \quad \rightsquigarrow \quad u=\ytableaushort{\cdots 1 b \cdots \cdots ,\cdots 2 {\scriptstyle b-1} n \cdots} \quad \rightsquigarrow \quad w=\ytableaushort{\cdots {\scriptstyle b-1}  b \cdots \cdots,\cdots 1  2 n \cdots} $$
Note that this path contributes $1$ to $N_{1, b-1}(v,w)$.

For later use, we set $\fI =\#(v^2\cap [b+1, n])$. By (e) combined with the fact that $n \in v^2$, it follows that $\fI = \frac{n-b-1}{2}+1=\frac{n-b+1}{2}$.

\subsubsection{$y=2$ case} \label{y=2case} Now we consider the move consisting of $1\intch 2$ and then $2 \intch b$, i.e.
$$v=\ytableaushort{\cdots 1 {\scriptstyle b-1} \cdots,\cdots 2 b \cdots} \quad \rightsquigarrow \quad u=\ytableaushort{\cdots 2  {\scriptstyle b-1} \cdots,\cdots 1 b \cdots} \quad \rightsquigarrow \quad w=\ytableaushort{\cdots {\scriptstyle b-1}  b \cdots,\cdots 1  2 \cdots} $$
Note that $b\neq 3$ since $b-1$ and $2$ are in different rows of $v$. Thus the move $2 \intch b$ is of the second kind and the following conditions in \ref{sec:defawg} are imposed:
\begin{enumerate}[label=(\alph*)]
\item $\mo{2-b}$ is odd, i.e. $n-b$ is odd
\item $b+1\in u^1$ (and $1\in u^2$), i.e. $b+1\in v^1$ (note that $\mo{b+1}\neq 1,2$)
\item (the third condition is satisfied since $b-1 \in u^1$)
\item $\#(u^2\cap [n-2k+1, n]) \geq k$ for $k\in \{1, 2, \ldots, \frac{n-b-1}{2}\}$, or equivalently $\#(v^2\cap [n-2k+1, n]) \geq k$ for $k\in \{1, 2, \ldots, \frac{n-b-1}{2}\}$ 
\item $\#(u^2\cap [b+2, n])=\frac{n-b-1}{2}$ (note that $\mo{b+1}\neq 2$), or equivalently $\#(v^2\cap [b+2, n])=\frac{n-b-1}{2}$
\end{enumerate}
By part (b), we have:
$$v=\ytableaushort{\cdots 1 {\scriptstyle b-1} {\scriptstyle b+1} \cdots,\cdots 2 b \cdots\cdots} \quad \rightsquigarrow \quad u=\ytableaushort{\cdots 2  {\scriptstyle b-1} {\scriptstyle b+1} \cdots,\cdots 1 b \cdots\cdots} \quad \rightsquigarrow \quad w=\ytableaushort{\cdots {\scriptstyle b-1}  b {\scriptstyle b+1} \cdots,\cdots 1  2 \cdots\cdots} $$
Note that this path contributes 1 to $N_{b-1,1}(v,w)$

As before, we set $\fI=\#(v^2 \cap [b+1, n])$. Then by (e) combined with the fact that $b+1 \in v^1$, we have $\fI =\frac{n-b-1}{2}.$

\subsubsection{$x\neq b-1$ case} \label{x!=b-1case} Let us now consider the case when we perform $x \intch b$ and then $1 \intch x$ for some $x\neq b-1$. Thus we have:
$$v=\ytableaushort{\cdots 1 {\scriptstyle b-1} x \cdots,\cdots 2 b \cdots\cdots} \quad \rightsquigarrow \quad u=\ytableaushort{\cdots 1 {\scriptstyle b-1} b \cdots,\cdots 2 x \cdots\cdots} \quad \rightsquigarrow \quad w=\ytableaushort{\cdots {\scriptstyle b-1}  b x \cdots,\cdots 1  2 \cdots\cdots } $$
As $x$ is neither equal to $b-1$ nor $2$, these two moves are both of the second kind. Thus the following conditions in \ref{sec:defawg} are required:
\begin{enumerate}[label=(\alph*)]
\item $x-b$ and $\mo{1-x}=n+1-x$ are both odd, thus in particular $n-b$ is odd
\item $b+1 \in v^1$, $\mo{x+1}\in u^1$, $x-1\in v^2$, $n\in u^2$, which means:
\begin{enumerate}[label=$\bullet$]
\item If $x=b+1=n$, then ($1\in v^1, b=n-1\in v^2$, and) $b+1=n\in v^1$
\item If $x=b+1\neq n$, then ($b\in v^2$ and) $b+1\in v^1, b+2\in v^1,$ and $n\in v^2$
\item If $x=n\neq b+1$, then ($1\in v^1$ and) $b+1\in v^1, n-1\in v^2, $ and $n \in v^1$
\item Otherwise, $b+1 \in v^1$, $x+1\in v^1$, $x-1\in v^2$, $n\in v^2$
\end{enumerate}
\item (the third condition is satisfied since $b-1\in v^1$ and $2\in u^2$)
\item $\#(v^2\cap [x-1-2k, x-2])\geq k$ for $k\in \{1, 2, \ldots, \frac{x-b-3}{2}\}$, and 
$\#(u^2 \cap [n-2k, n-1]) \geq k$ for $k\in \{1, 2, \ldots, \frac{n-x-2}{2}\}$ which is equivalent to $\#(v^2 \cap [n-2k, n-1]) \geq k$ for $k\in \{1, 2, \ldots, \frac{n-x-2}{2}\}$ 
\item $\#(v^2 \cap [b+2, x-2]) =\frac{x-b-3}{2}$ if $b+1\neq x$, and $\#( u^2 \cap [x+2, n-1]) =\frac{n-x-2}{2}$ if $\mo{x+1}\neq 1$, i.e. $x\neq n$ which is equivalent to  $\#(v^2 \cap [x+2, n-1]) =\frac{n-x-2}{2}$ if $x\neq n$
\end{enumerate}
We divide all the possibilities into the four cases below. By part (b), we are in the following situation in each case.
\begin{enumerate}[label=$\bullet$]
\item If $x=b+1=n$, then
$$
v=\ytableaushort{
\cdots 1 {\scriptstyle n-2} n \cdots,
\cdots 2 {\scriptstyle n-1}  \cdots\cdots
}
\quad \rightsquigarrow \quad
u=\ytableaushort{
\cdots 1 {\scriptstyle n-2} {\scriptstyle n-1} \cdots,
\cdots 2 n  \cdots\cdots
}
\quad \rightsquigarrow \quad
w=\ytableaushort{
\cdots {\scriptstyle n-2} {\scriptstyle n-1} n \cdots,
\cdots 1 2  \cdots\cdots
}$$
\item If $x=b+1\neq n$, then 
$$
v=\ytableaushort{
\cdots 1 {\scriptstyle b-1}{\scriptstyle b+1} {\scriptstyle b+2} \cdots,
\cdots 2 b n \cdots\cdots
}
\quad \rightsquigarrow \quad
u=\ytableaushort{
\cdots 1 {\scriptstyle b-1} b {\scriptstyle b+2} \cdots,
\cdots 2 {\scriptstyle b+1} n \cdots\cdots
}
\quad \rightsquigarrow \quad
w=\ytableaushort{
\cdots {\scriptstyle b-1} b {\scriptstyle b+1} {\scriptstyle b+2} \cdots,
\cdots 1 2  n \cdots\cdots
}$$
\item If $x=n\neq b+1$, then 
$$
v=\ytableaushort{
\cdots 1 {\scriptstyle b-1}{\scriptstyle b+1} n  \cdots,
\cdots 2 b {\scriptstyle n-1} \cdots\cdots
}
\quad \rightsquigarrow \quad
u=\ytableaushort{
\cdots 1 {\scriptstyle b-1} b {\scriptstyle b+1} \cdots,
\cdots 2 {\scriptstyle n-1} n \cdots\cdots
}
\quad \rightsquigarrow \quad
w=\ytableaushort{
\cdots  {\scriptstyle b-1} b {\scriptstyle b+1} n \cdots,
\cdots 1 2 {\scriptstyle n-1}  \cdots\cdots
}$$
\item Otherwise,
$$
v=\ytableaushort{
\cdots 1 {\scriptstyle b-1}{\scriptstyle b+1} x{\scriptstyle x+1} \cdots,
\cdots 2 b {\scriptstyle x-1} n \cdots\cdots
}
\quad \rightsquigarrow \quad
u=\ytableaushort{
\cdots 1 {\scriptstyle b-1} b {\scriptstyle b+1} {\scriptstyle x+1} \cdots,
\cdots 2 {\scriptstyle x-1} x n \cdots\cdots
}
\quad \rightsquigarrow \quad
w=\ytableaushort{
\cdots  {\scriptstyle b-1} b {\scriptstyle b+1} x{\scriptstyle x+1} \cdots,
\cdots 1 2 {\scriptstyle x-1}  n \cdots\cdots
}$$
\end{enumerate}
Also note that this path contributes 1 to $N_{1,b-1}(v,w)$.

As before we set $\fI\colonequals \#(v^2 \cap [b+1, n])$ and prove that $\fI=\frac{n-b-1}{2}$. Here our argument relies on part (e) and the description of each case above.
\begin{enumerate}[label=$\bullet$]
\item If $x=b+1=n$, then obviously $\fI=0=\frac{n-b-1}{2}$ as $n \in v^1$.
\item If $x=b+1\neq n$, then since $\#( u^2 \cap [b+3, n-1]) =\frac{n-b-3}{2}$ we have $\fI = \frac{n-b-3}{2}+1=\frac{n-b-1}{2}.$
\item If $x=n\neq b+1$, then since $\#(v^2 \cap [b+2, n-2])=\frac{n-b-3}{2}$ we have $\fI = \frac{n-b-3}{2}+1 = \frac{n-b-1}{2}.$
\item Otherwise, since $\#(v^2 \cap [b+2, n-1]) =\frac{x-b-3}{2}+\frac{n-x-2}{2}+1=\frac{n-b-3}{2}$ we have $\fI = \frac{n-b-3}{2}+1 = \frac{n-b-1}{2}$.
\end{enumerate}

\subsubsection{$y\neq 2$ case} \label{y!=2case} Here we consider the remaining possibility, which is to perform $1\intch y$ and then $y \intch b$ for some $y\neq 2$. Thus $b<y\leq n$ and we have:
$$v=\ytableaushort{\cdots 1 {\scriptstyle b-1} \cdots \cdots,\cdots 2 b y \cdots} \quad \rightsquigarrow \quad u=\ytableaushort{\cdots  {\scriptstyle b-1} y \cdots \cdots,\cdots 1 2 b \cdots} \quad \rightsquigarrow \quad w=\ytableaushort{\cdots  {\scriptstyle b-1} b\cdots \cdots,\cdots 1 2 y \cdots}$$
As $y$ is neither $2$ nor $b-1$, these two moves are both of the second kind. Thus the following conditions in \ref{sec:defawg} are imposed:
\begin{enumerate}[label=(\alph*)]
\item $\mo{1-y} = n+1-y$ and $y-b$ are both odd, thus in particular $n-b$ is odd
\item $\mo{y+1}\in v^1, b+1 \in u^1, n \in v^2,$ and $y-1 \in u^2$, that is:
\begin{enumerate}[label=$\bullet$]
\item If $y=b+1=n$, then ($1\in v^1, b=n-1\in v^2$ and) $b+1=n\in v^2$
\item if $y=b+1\neq n$, then ($b\in v^2$ and) $b+1 \in v^2$, $b+2 \in v^1$, and $n\in v^2$
\item If $y=n\neq b+1$, then ($1\in v^1$ and) $b+1 \in v^1$, $n-1\in v^2$, and $n\in v^2$
\item Otherwise, $y+1\in v^1, b+1 \in v^1, n \in v^2,$ and $y-1 \in v^2$
\end{enumerate}
\item (the third condition is satisfied since $2 \in v^2$ and $b-1 \in u^1$)
\item $\#(v^2 \cap [n-2k, n-1]) \geq k$ for $k\in \{1, 2, \ldots, \frac{n-y-2}{2}\}$, and $\#( u^2 \cap [y-1-2k, y-2])\geq k$ for $k\in \{1, 2, \ldots, \frac{y-b-3}{2}\}$ which is equivalent to $\#( v^2 \cap [y-1-2k, y-2])\geq k$ for $k\in \{1, 2, \ldots, \frac{y-b-3}{2}\}$
\item $\#( v^2\cap [y+2, n-1])= \frac{n-y-2}{2}$ if $y\neq n$, and $\#( u^2 \cap [b+2, y-2]) = \frac{y-b-3}{2}$ if $y\neq b+1$ which is equivalent to $\#( v^2 \cap [b+2, y-2]) = \frac{y-b-3}{2}$ if $y\neq b+1$.
\end{enumerate}
We divide all the possibilities into the four cases below. By part (b), we are in the following situation in each case.
\begin{enumerate}[label=$\bullet$]
\item If $y=b+1=n$, then
$$
v=\ytableaushort{
\cdots 1 {\scriptstyle n-2}\cdots \cdots ,
\cdots 2 {\scriptstyle n-1} n \cdots
}
\quad \rightsquigarrow \quad
u=\ytableaushort{
\cdots {\scriptstyle n-2} n\cdots \cdots ,
\cdots 1 2 {\scriptstyle n-1} \cdots
}
\quad \rightsquigarrow \quad
w=\ytableaushort{
\cdots {\scriptstyle n-2} {\scriptstyle n-1}\cdots \cdots ,
\cdots 1 2 n \cdots
}
$$
\item if $y=b+1\neq n$, then
$$
v=\ytableaushort{
\cdots 1 {\scriptstyle b-1} {\scriptstyle b+2}\cdots \cdots ,
\cdots 2 b  {\scriptstyle b+1} n \cdots
}
\quad \rightsquigarrow \quad
u=\ytableaushort{
\cdots {\scriptstyle b-1}{\scriptstyle b+1} {\scriptstyle b+2}\cdots \cdots ,
\cdots 1 2 b   n \cdots
}
\quad \rightsquigarrow \quad
w=\ytableaushort{
\cdots {\scriptstyle b-1}b {\scriptstyle b+2}\cdots \cdots ,
\cdots 1 2 {\scriptstyle b+1}   n \cdots
}
$$
\item If $y=n\neq b+1$, then
$$
v=\ytableaushort{
\cdots 1 {\scriptstyle b-1} {\scriptstyle b+1}\cdots \cdots ,
\cdots 2 b {\scriptstyle n-1} n \cdots
}
\quad \rightsquigarrow \quad
u=\ytableaushort{
\cdots {\scriptstyle b-1} {\scriptstyle b+1} n\cdots \cdots ,
\cdots 1 2 b {\scriptstyle n-1} \cdots
}
\quad \rightsquigarrow \quad
w=\ytableaushort{
\cdots {\scriptstyle b-1} b {\scriptstyle b+1} \cdots \cdots ,
\cdots 1 2 {\scriptstyle n-1} n\cdots
}
$$
\item Otherwise,
$$
v=\ytableaushort{
\cdots 1 {\scriptstyle b-1} {\scriptstyle b+1} {\scriptstyle y+1} \cdots \cdots ,
\cdots 2 b {\scriptstyle y-1} y n \cdots
}
\quad \rightsquigarrow \quad
u=\ytableaushort{
\cdots {\scriptstyle b-1} {\scriptstyle b+1} y {\scriptstyle y+1} \cdots \cdots ,
\cdots 1 2 b {\scriptstyle y-1} n \cdots
}
\quad \rightsquigarrow \quad
w=\ytableaushort{
\cdots {\scriptstyle b-1} b {\scriptstyle b+1} {\scriptstyle y+1} \cdots \cdots ,
\cdots 1 2 {\scriptstyle y-1} y n \cdots
}
$$
\end{enumerate}
Also note that this path contributes 1 to $N_{b-1,1}(v,w)$.

As before we set $\fI\colonequals \#(v^2 \cap [b+1,n])$ and prove that $\fI=\frac{n-b+1}{2}$. Here, our argument relies on part (e) and description for each case above.
\begin{enumerate}[label=$\bullet$]
\item If $y=b+1=n$, then obviously $\fI=1=\frac{n-b+1}{2}$ as $n \in v^2$.
\item If $y=b+1\neq n$, then since $\#( v^2 \cap [b+3, n-1]) =\frac{n-b-3}{2}$ we have $\fI = \frac{n-b-3}{2}+2=\frac{n-b+1}{2}.$
\item If $y=n\neq b+1$, then since $\#(v^2 \cap [b+2, n-2]) =\frac{n-b-3}{2}$ we have $\fI = \frac{n-b-3}{2}+2 = \frac{n-b+1}{2}.$
\item Otherwise, since $\#(v^2 \cap [b+2, n-1]) =\frac{y-b-3}{2}+\frac{n-y-2}{2}+2=\frac{n-b-1}{2}$ we have $\fI = \frac{n-b-1}{2}+1 = \frac{n-b+1}{2}$.
\end{enumerate}

\subsection{$b\neq \mo{a-1}$ case: verification of the polygon rule} \label{b!=a-1calc} 
Now we summarize the discussion in \ref{b!=a-1case} and verify that $N_{i,j}(v,w) = N_{j,i}(v,w)$ for $\{i,j\} = \{a, \mo{b-1}\}$. As before it suffices to consider the case when $a=1$, and thus we may assume that $i=1$ and $j=b-1$. First note that if $n-b$ is even then the polygon rule is trivially satisfied since $N_{1,b-1}(v,w) = N_{b-1,1}(v,w) =0$. (See the condition (a) in each case.) Thus from now on we assume that $n-b$ is odd.  Also from the argument above if the value $\fI= \#( v^2 \cap [b+1,n])$ is not equal to $\frac{n-b\pm1}{2}$ then again we have $N_{1,b-1}(v,w) = N_{b-1,1}(v,w) =0$. Now we consider the case $\fI = \frac{n-b+1}{2}$ and $\fI = \frac{n-b-1}{2}$ separately.

\subsubsection{$\fI = \frac{n-b+1}{2}$ case} It suffices only to consider \ref{x=b-1case} and \ref{y!=2case}. Then either $N_{1,b-1}(v,w)$ or $N_{b-1,1}(v,w)$ is not zero only when $n \in v^2$ (see condition (b)), thus we suppose that this is true. Here, $N_{1,b-1}(v,w)$ is easier to calculate; it equals 1 if ($1, b-1 \in v^1$, $2, b\in v^2$ and) $\#(v^2 \cap [n-2k, n-1])\geq k$ for $k\in \{1, 2, \ldots, \frac{n-b-1}{2}\}$ and 0 otherwise.

On the other hand, we first show that $N_{b-1,1}(v,w)\leq 1$. For the sake of contradiction, suppose that we have $y,y'\in v^2$, $b<y'<y\leq n$ which correspond to the two-step moves described in \ref{y!=2case}. First note that $y'\neq b+1$, since otherwise $b+1 \in v^2$ which forces $y=b+1$ by the condition (b) in \ref{y!=2case}, which is impossible.

We consider the case when $y\neq n$. Then from the conditions in \ref{y!=2case} we may derive that 
\begin{gather*}
\#(v^2\cap [y+2, n-1]) =\frac{n-y-2}{2},\quad \#( v^2 \cap [y'+2, n-1]) =\frac{n-y'-2}{2},
\\\#(v^2 \cap [y'+1, y-2]) \geq \frac{y-y'-2}{2},
\end{gather*}
(note that $y-y'$ is even) from which it also follows that $\#(v^2 \cap [y'+2, y+1]) =\frac{y-y'}{2}.$
However, as $y'+1, y+1 \in v^1$ and $y-1, y\in v^2$ from the description, it implies that $\#(v^2 \cap [y'+1, y-2]) =\frac{y-y'}{2}-2<\frac{y-y'-2}{2},$
which is contradiction. Now we suppose that $y=n$. We still have
\begin{gather*}
\#(v^2 \cap [y'+2, n-1]) =\frac{n-y'-2}{2}, \quad \#(v^2 \cap [y'+1, n-2]) \geq \frac{n-y'-2}{2},
\end{gather*}
but this is impossible since $y'+1 \in v^1$ and $n-1\in v^2$. This proves that $N_{b-1,1}(v,w)\leq 1$.

We are ready to prove that $N_{1,b-1}(v,w) = N_{b-1,1}(v,w)$. First suppose that $N_{1,b-1}(v,w)=1$, thus in particular 
\begin{equation}\label{ineq1b-1}
\#(v^2 \cap [n-2k, n-1]) \geq k \textup{ for }k\in \left\{1, 2, \ldots, \frac{n-b-1}{2}\right\}.\tag{$\bigstar$}
\end{equation}
Then as we proved that $N_{b-1,1}(v,w)\leq 1$, it suffices to show the existence of a two-step move corresponding to \ref{y!=2case}. First assume that $b+1 \in v^2$. Then we claim that there exists a two-step move consisting of $1\intch (b+1)$ and $(b+1) \intch b$. To this end, we check that the conditions in \ref{y!=2case} are valid as follows.
\begin{enumerate}[label=$-$, leftmargin=*]
\item If $b+1=n$, then the only nontrivial part is (b), which holds since $b+1=n\in v^2$.
\item Otherwise, we still have $b+1,n \in v^2$. We also have that $b+2 \in v^1$ and thus part (b) holds; otherwise $\{b+1,b+2, n\}\subset v^2$, which implies that $\#(v^2 \cap [b+3, n-1])=\fI-3 =  \frac{n-b-5}{2}$, contradicting (\ref{ineq1b-1}) for $k= \frac{n-b-3}{2}$. For part (d), we should have $\#(v^2 \cap [n-2k, n-1]) \geq k$ for $k\in \{1, 2, \ldots, \frac{n-b-3}{2}\}$, which follows from (\ref{ineq1b-1}). For part (e), we should have $\#(v^2 \cap [b+3, n-1]) = \frac{n-b-3}{2}$, but it follows from the fact that $\fI =  \frac{n-b+1}{2}$ together with part (b).
\end{enumerate}
It remains to consider the case when $b+1 \in v^1$. Here we first set
$$\fT= \left\{z \in v^2 \mid b+3\leq z\leq n, n-z \textup{ is even, } z-1 \in v^2, \#(v^2 \cap[b+2, z-2]) = \frac{z-b-3}{2} \right\}.$$
We claim that $\fT \neq \emptyset$; otherwise, inductive argument shows that $n-1 \in v^1, n-2\in v^2, n-3\in v^1, \ldots, b+4\in v^1, b+3 \in v^2, b+2 \in v^1$ which follows from (\ref{ineq1b-1}) and the assumption $\fI=\frac{n-b+1}{2}$, but this contradicts that $b+2\in v^2$. Now we set $y=\min \fT$. (Note that $y\neq b+1$.) We claim that there exists a two-step move consisting of $1\intch y$ and $y \intch b$. To this end, again we check that the conditions in \ref{y!=2case} hold as follows.
\begin{enumerate}[label=$-$, leftmargin=*]
\item If $y=n$, then part (b) holds since $b+1 \in v^1, n \in v^2,$ and $n-1 \in v^2$ by definition of $\fT$. For part (d), we should have $\#(v^2 \cap [n-1-2k, n-2]) \geq k$ for $k\in \{1, 2, \ldots, \frac{n-b-3}{2}\}$, thus suppose otherwise for contradiction and choose $k\in \{1, 2, \ldots, \frac{n-b-3}{2}\}$ to be maximum which satisfies $\#(v^2 \cap [n-1-2k, n-2]) < k$. By definition of $\fT$, $k< \frac{n-b-3}{2}$ and we have $\#(v^2 \cap [n-3-2k, n-2]) \geq k+1$ by maximality of $k$. This is only possible when $\#(v^2 \cap [n-1-2k, n-2]) =k-1$ and $n-2-2k, n-3-2k \in v^2$. However, it means that 
\begin{align*}
\#(v^2 \cap [b+2, n-4-2k]) &= \fI - \#(v^2 \cap [n-1-2k, n-2])-4
\\&=\frac{n-b+1}{2}-(k-1)-4=\frac{n-2k-b-5}{2},
\end{align*}
which means that $n-2-2k\in \fT$. It contradicts the assumption that $n=\min\fT$, thus we conclude that part (d) holds. For part (e), we need to check that $\#(v^2 \cap [b+2, n-2]) = \frac{n-b-3}{2}$ which follows from the assumption $\fI=\frac{n-b+1}{2}$ together with part (b).
\item Otherwise, $b+1 \in v^1, n \in v^2,$ and $y-1 \in v^2$ by definition of $\fT$, thus part (b) holds if $y+1 \in v^1$. However, if $y+1 \in v^2$ then by (\ref{ineq1b-1}) we have
$$\fI = \#(v^2 \cap [b+1, y-2])+\#(v^2 \cap [y+2, n-1])+4\geq \frac{y-b-3}{2}+\frac{n-y-2}{2}+4=\frac{n-b+3}{2}$$
which is a contradiction. Thus $y+1 \in v^1$ and part (b) holds. Now we prove part (e), i.e. $\#( v^2\cap [y+2, n-1])= \frac{n-y-2}{2}$ and $\#(v^2 \cap [b+2, y-2]) = \frac{y-b-3}{2}$. However the second equality follows from definition of $\fT$ and the first one also follows since
\begin{align*}
\#( v^2\cap [y+2, n-1])&=\fI - \#(v^2 \cap [b+2, y-2])-3=\frac{n-b+1}{2}-\frac{y-b-3}{2}-3=\frac{n-y-2}{2}.
\end{align*}
It remains to prove part (d). We should have $\#(v^2 \cap [n-2k, n-1]) \geq k$ for $k\in \{1, 2, \ldots, \frac{n-y-2}{2}\}$ and $\#(v^2 \cap [y-1-2k, y-2]) \geq k$ for $k\in \{1, 2, \ldots, \frac{y-b-3}{2}\}$. The first inequality follow directly from (\ref{ineq1b-1}), thus suppose that the second inequality does not hold and choose $k\in \{1, 2, \ldots, \frac{y-b-3}{2}\}$ to be maximum which satisfies $\#(v^2 \cap [y-1-2k, y-2]) < k$. By definition of $\fT$, $k< \frac{y-b-3}{2}$ and we have $\#(v^2 \cap [y-3-2k, y-2]) \geq k+1$ by maximality of $k$. This is only possible when $\#(v^2 \cap [y-1-2k, y-2]) =k-1$ and $y-2-2k, y-3-2k \in v^2$. However, it means that 
\begin{align*}
\#(v^2 \cap [b+2, y-4-2k]) &= \fI - \#(u^2 \cap [y-1-2k, y-2])- \#(u^2 \cap [y+2, n-1])-5
\\&=\frac{n-b+1}{2}-(k-1)-\frac{n-y-2}{2}-5=\frac{y-2k-b-5}{2},
\end{align*}
which means that $y-2-2k\in \fT$. It contradicts the assumption that $y=\min\fT$, thus we conclude that part (d) holds. 
\end{enumerate}
We cover all the possible cases and we conclude that $N_{1,b-1}(v,w) = N_{b-1,1}(v,w)=1$.

Therefore, in order to prove that $N_{1,b-1}(v,w) = N_{b-1,1}(v,w)$ it remains to show that $N_{1,b-1}(v,w)=1$ when there exists $y$ such that the two-step move $1\intch y$ and then $y \intch b$ is valid. If such $y$ exists, then the conditions $\#(v^2 \cap [n-2k, n-1]) \geq k$ for $k\in \{1, 2, \ldots, \frac{n-y-2}{2}\}$ and $\#(v^2 \cap [y-1-2k, y-2]) \geq k$ for $k\in \{1, 2, \ldots, \frac{y-b-3}{2}\}$ imply that $\#(v^2 \cap [n-2k, n-1]) \geq k$ for $k\in \{1, 2, \ldots, \frac{n-b-1}{2}\}$ as $y-1, y \in v^2$. Thus we see that $N_{1,b-1}(v,w)=1$ and again $N_{1,b-1}(v,w) = N_{b-1,1}(v,w)=1$.

As a result, the Polygon Rule holds for $(v,w)$ when $\fI = \frac{n-b+1}{2}$.

\subsubsection{$\fI = \frac{n-b-1}{2}$ case} This case is totally analogous to the previous one. It suffices only to consider \ref{y=2case} and \ref{x!=b-1case}. Then either $N_{1,b-1}(v,w)$ or $N_{b-1,1}(v,w)$ is not zero only when $b+1 \in v^1$, thus we suppose that this is true. Here, $N_{b-1,1}(v,w)$ is easier to calculate; it equals 1 if ($1, b-1 \in v^1$, $2, b\in v^2$ and) $\#(v^2\cap [n-2k+1, n]) \geq k$ for $k\in \{1, 2, \ldots, \frac{n-b-1}{2}\}$ and 0 otherwise.

On the other hand, we first show that $N_{1,b-1}(v,w)\leq 1$. For the sake of contradiction, suppose that we have $x,x'\in v^1$, $b<x'<x\leq n$ which correspond to the two-step moves described in \ref{x!=b-1case}. Note that $x\neq n$ (and thus $n\in v^2$), since otherwise $n\in v^1$ and thus $x'=n$ by the description of $v$ in \ref{x!=b-1case}. But it contradicts that $x'<x$. Now from the conditions in \ref{x!=b-1case} we may derive that 
\begin{gather*}
\#(v^2\cap [x,n-1]) \geq \frac{n-x}{2},\quad \#(v^2\cap [x+2, n-1]) =\frac{n-x-2}{2}
\end{gather*}
where the first condition comes from part (d) with respect to $x'$ (note that $x'<x$). But this is impossible since $x,x+1 \in v^1$. This proves that $N_{1,b-1}(v,w)\leq 1$.

We are ready to prove that $N_{1,b-1}(v,w) = N_{b-1,1}(v,w)$. First suppose that $N_{b-1,1}(v,w)=1$, thus in particular 
\begin{equation}\label{ineqb-11}
\#(v^2\cap [n-2k+1, n]) \geq k \textup{ for } k\in \left\{1, 2, \ldots, \frac{n-b-1}{2}\right\}. \tag{$\varheart$}
\end{equation}
Then as we proved that $N_{1,b-1}(v,w)\leq 1$, it suffices to show the existence of a two-step move corresponding to \ref{x!=b-1case}. First assume that $n \in v^1$. Then we claim that there exists a two-step move consisting of $n\intch b$ and $1 \intch n$. To this end, we check that the conditions in \ref{x!=b-1case}. are valid as follows.
\begin{enumerate}[label=$-$, leftmargin=*]
\item If $b+1=n$, then the only nontrivial part is (b), which holds since $b+1=n\in v^1$
\item Otherwise, we still have $b+1\in v^1$ and $n\in v^1$. Also, $\#(v^2\cap[n-1,n])\geq 1$ by (\ref{ineqb-11}), which forces that $n-1\in v^2$, thus part (b) holds. For part (d), we should have $\#(v^2\cap [n-1-2k, n-2])\geq k$ for $k\in \{1, 2, \ldots, \frac{n-b-3}{2}\}$, which is true by (\ref{ineqb-11}) together with part (b). For part (e), we require $\#(v^2 \cap [b+2, n-2]) =\frac{n-b-3}{2}$, which follows from the assumption that $\fI = \frac{n-b-1}{2}$ together with part (b).
\end{enumerate}
It remains to consider the case when $n \in v^2$. First note that $b+1 \in v^1$ and $\#(v^2\cap [b+2,n])=\frac{n-b-1}{2}$ because of the conditions $\fI = \frac{n-b-1}{2}$ and (\ref{ineqb-11}) for $k=\frac{n-b-1}{2}$. Now we set
$$\fT = \left\{z\in v^1\mid b+1\leq z<n, n-z \textup{ is even, } z+1\in v^1, \#(v^2 \cap [z+2, n-1]) =\frac{n-z-2}{2}\right\}.$$
We claim that $\fT \neq \emptyset$; otherwise, inductive argument shows that $b+2\in v^2, b+3\in v^1, b+4\in v^2, \ldots, n-2\in v^1, n-1\in v^2$ which follows from (\ref{ineqb-11}) and the equation  $\#(v^2\cap [b+2,n])=\frac{n-b-1}{2}$, but it contradicts the fact that $\fI = \frac{n-b-1}{2}$. Now we set $x=\max \fT$. (Note that $x\neq n$.) We claim that there exists a two-step move consisting of $x\intch b$ and $1\intch x$. To this end, again we check that the conditions in \ref{x!=b-1case} hold as follows.

\begin{enumerate}[label=$-$, leftmargin=*]
\item If $x=b+1$, then part (b) holds since $b+1 \in v^1, n \in v^2,$ and $b+2 \in v^1$ by definition of $\fT$. For part (d), we should have $\#(v^2 \cap [n-2k, n-1]) \geq k$ for $k\in \{1, 2, \ldots, \frac{n-b-3}{2}\}$ , thus suppose otherwise for contradiction and choose  $k\in \{1, 2, \ldots, \frac{n-b-3}{2}\}$ to be minimum which satisfies $\#(v^2 \cap [n-2k, n-1]) < k$. (Note that this only happens when $b+1<n-2$.) If $k=1$, then the inequality says that $n-2, n-1 \in v^1$ which implies $n-2\in \fT$, but this is impossible by the maximality of $x=b+1$ in $\fT$. Thus $k>1$ and by minimality of $k$ we have $\#(v^2 \cap [n-2k+2, n-1]) \geq k-1$. This is only possible when $\#(v^2 \cap [n-2k+2, n-1]) =k-1$ and $n-2k, n-2k+1 \in v^1$. However, it means that $n-2k\in \fT$. It contradicts the assumption that $b+1=\max\fT$, thus we conclude that part (d) holds. For part (e), we need to check that $\#(v^2 \cap [b+3, n-1]) =\frac{n-b-3}{2}$, but this follows from the definition of $\fT$.
\item Otherwise, $b+1 \in v^1, n \in v^2,$ and $x+1 \in v^1$ by definition of $\fT$, thus part (b) holds if $x-1 \in v^2$. However, if $x-1 \in v^1$ then by (\ref{ineqb-11}) we have
$$ \frac{n-x+2}{2}\leq \#(v^2 \cap [x-1,n])= \#(v^2 \cap [x+2,n-1]\}+1 = \frac{n-x}{2},
$$
which is absurd. Thus $x-1 \in v^2$ and part (b) holds. Now we prove part (e), i.e. $\#(v^2 \cap [b+2, x-2]) =\frac{x-b-3}{2}$ and $\#(v^2 \cap [x+2, n-1]) =\frac{n-x-2}{2}$. However, the second equality follows from definition of $\fT$, and also
$$\#(v^2 \cap [b+2, x-2]) =\fI - \#(v^2 \cap [x+2, n-1])-2=\frac{n-b-1}{2}-\frac{n-x-2}{2}-2=\frac{x-b-3}{2}$$
thus the first equality holds.
It remains to prove part (d), that is we should have $\#(v^2\cap [x-1-2k, x-2])\geq k$ for $k\in \{1, 2, \ldots, \frac{x-b-3}{2}\}$ and $\#(v^2 \cap [n-2k, n-1]) \geq k$ for $k\in \{1, 2, \ldots, \frac{n-x-2}{2}\}$.
By (\ref{ineqb-11}), we have
\begin{align*}
\#(v^2\cap [x-1-2k, x-2])&=\#(v^2\cap [x-1-2k, n])-\#(v^2\cap [x+2, n-1])-2
\\&\geq \frac{n-x+2k+2}{2}-\frac{n-x-2}{2}-2 = k,
\end{align*}
from which the first inequality follows. Now for contradiction suppose that there exists $k\in \{1, 2, \ldots, \frac{n-x-2}{2}\}$ such that $\#(v^2\cap [n-2k, n-1])< k$ and choose $k$ to be minimum among such values. (Note that this implies $1\leq \frac{n-x-2}{2}$, i.e. $x< n-2$.) If $k=1$, then the inequality says $n-2, n-1 \in v^1$ which implies $n-2 \in \fT$, but this contradicts the maximality of $x$. Thus $k>1$ and by minimality of $k$ we have  $\#(v^2\cap [n-2k+2, n-1])\geq k-1$. This is only possible when $\#(v^2\cap [n-2k+2, n-1])= k-1$ and $n-2k+1, n-2k\in v^1$. However, it means that $n-2k \in fT$ which again contradicts the assumption that $x=\max \fT$. Thus we conclude that part (d) holds.
\end{enumerate}
We cover all the possible cases and we conclude that $N_{1,b-1}(v,w) = N_{b-1,1}(v,w)=1$.

Therefore, in order to prove that $N_{1,b-1}(v,w) = N_{b-1,1}(v,w)$ it remains to show that $N_{b-1,1}(v,w)=1$ when there exists $x$ such that the two-step move $x\intch b$ and then $1 \intch x$ is valid. If such $x$ exists, then the conditions $\#(v^2\cap [x-1-2k, x-2])\geq k$ for $k\in \{1, 2, \ldots, \frac{x-b-3}{2}\}$ and $\#(v^2 \cap [n-2k, n-1]) \geq k$ for $k\in \{1, 2, \ldots, \frac{n-x-2}{2}\}$ imply $\#(v^2\cap [n-2k+1, n]) \geq k$ for $k\in \{1, 2, \ldots, \frac{n-b-1}{2}\}$ (If $x=n$, then it follows since $n-1\in v^2$. Otherwise, it follows since $n,x-1\in v^2$.) Thus we see that $N_{1,b-1}(v,w)=1$ and again $N_{1,b-1}(v,w) = N_{b-1,1}(v,w)=1$.

As a result, the polygon rule holds for $(v,w)$ when $\fI = \frac{n-b-1}{2}$. This suffices for the proof.

\section{Case 2: $v$ and $w$ differ by four elements}
In this section we consider the case when $v$ and $w$ differ by four elements.
Let us set $\{a,b\}=v^1-w^1$ and $\{c,d\}=v^2-w^2$. In other words, we have:
$$v=\ytableaushort{\cdots a b \cdots,\cdots c d \cdots} \quad \rightsquigarrow u \rightsquigarrow \quad w=\ytableaushort{\cdots c d \cdots,\cdots a b \cdots} $$
Then we have the following four possibilities provided that the conditions in \ref{sec:defawg} are satisfied:
\begin{enumerate}[label=$\bullet$]
\item Interchange $a$ and $c$, and interchange $b$ and $d$
\item Interchange $a$ and $d$, and interchange $b$ and $c$
\item Interchange $b$ and $c$, and interchange $a$ and $d$
\item Interchange $b$ and $d$, and interchange $a$ and $c$
\end{enumerate}
We cover all the possibilities by case-by-case argument from now on. 
Let us use the notation $\Nf{a}{c}{b}{d}_{i,j}$ for the contribution of the first way to $N_{i, j}(v,w)$, etc. Then each of those is either $0$ or $1$ and $N_{i,j}(v,w) = \Nf{a}{c}{b}{d}_{i,j}+\Nf{a}{d}{b}{c}_{i,j}+\Nf{b}{c}{a}{d}_{i,j}+\Nf{b}{d}{a}{c}_{i,j}$.

\subsection{$\{c, d\} = \{\mo{a+1}, \mo{b+1}\}$ case} Without loss of generality, we set $c=\mo{a+1}$ and $d=\mo{b+1}$.
We are in the following situation:
$$v=\ytableaushort{\cdots a b  \cdots,\cdots {\scriptstyle \mo{a+1}} {\scriptstyle \mo{b+1}} \cdots} \quad \rightsquigarrow u \rightsquigarrow \quad w=\ytableaushort{\cdots {\scriptstyle \mo{a+1}} {\scriptstyle \mo{b+1}} \cdots,\cdots a b \cdots} $$
If $(i,j)$ satisfies $i,j \in \ades(v)$ and $i,j \not\in \ades(w)$ then we have $\{i,j\} = \{a,b\}$ and $a\neq b$. Then it suffices to prove $\Nf{b}{\mo{b+1}}{a}{\mo{a+1}}_{a,b}=\Nf{a}{\mo{a+1}}{b}{\mo{b+1}}_{b,a}$.

We first consider performing $a \intch \mo{a+1}$ and then $b \intch \mo{b+1}$. This is always possible:
$$v=\ytableaushort{\cdots a b  \cdots,\cdots {\scriptstyle \mo{a+1}} {\scriptstyle \mo{b+1}} \cdots} \quad \rightsquigarrow \quad u=\ytableaushort{\cdots {\scriptstyle \mo{a+1}} b  \cdots,\cdots a {\scriptstyle \mo{b+1}} \cdots} \quad \rightsquigarrow \quad w=\ytableaushort{\cdots {\scriptstyle \mo{a+1}} {\scriptstyle \mo{b+1}} \cdots,\cdots a b \cdots} $$
Similarly, consider performing $b \intch \mo{b+1}$ and then $a \intch \mo{a+1}$. This is also always possible:
$$v=\ytableaushort{\cdots a b  \cdots,\cdots {\scriptstyle \mo{a+1}} {\scriptstyle \mo{b+1}} \cdots} \quad \rightsquigarrow \quad u=\ytableaushort{\cdots a {\scriptstyle \mo{b+1}}  \cdots,\cdots {\scriptstyle \mo{a+1}}  b \cdots} \quad \rightsquigarrow \quad w=\ytableaushort{\cdots {\scriptstyle \mo{a+1}} {\scriptstyle \mo{b+1}} \cdots,\cdots a b \cdots} $$
%
%
%
To summarize, in this case we have $\Nf{a}{\mo{a+1}}{b}{\mo{b+1}}_{b,a} = \Nf{b}{\mo{b+1}}{a}{\mo{a+1}}_{a,b}=1$, $\Nf{a}{\mo{b+1}}{b}{\mo{a+1}}_{i,j} = \Nf{b}{\mo{a+1}}{a}{\mo{b+1}}_{i,j}=0$ for $\{i,j\} = \{a,b\}$.
Thus $N_{a, b}(v,w) = 1 = N_{b, a}(v,w)$ as desired. 

\subsection{$|\{c, d\} \cap \{\mo{a+1}, \mo{b+1}\}|=1$ case} Without loss of generality, we set $c=\mo{a+1}$ and $d \not = \mo{b+1}$.
We are in the following situation:
$$v=\ytableaushort{\cdots a b  \cdots,\cdots {\scriptstyle \mo{a+1}} d \cdots} \quad \rightsquigarrow u\rightsquigarrow \quad w=\ytableaushort{\cdots {\scriptstyle \mo{a+1}} d \cdots,\cdots a b \cdots} $$
If $(i,j)$ satisfies $i,j \in \ades(v)$ and $i,j \not\in \ades(w)$ then we have $\{i,j\} = \{a,\mo{d-1}\}$ or $\{i,j\} = \{a,b\}$ or $\{i,j\} = \{b,\mo{d-1}\}$. Then after removing trivial terms it suffices to verify:
\begin{enumerate}[label=$\bullet$]
\item if $\{i,j\} = \{a,\mo{d-1}\}$, then $\Nf{b}{d}{a}{\mo{a+1}}_{a,\mo{d-1}}=\Nf{a}{\mo{a+1}}{b}{d}_{\mo{d-1},a}+\Nf{b}{\mo{a+1}}{a}{d}_{\mo{d-1},a}$
\item if $\{i,j\} = \{a,b\}$, then $\Nf{b}{d}{a}{\mo{a+1}}_{a,b}=\Nf{a}{\mo{a+1}}{b}{d}_{b,a}+\Nf{a}{d}{b}{\mo{a+1}}_{b,a}$
\item if $\{i,j\} = \{b,\mo{d-1}\}$, then $\Nf{a}{d}{b}{\mo{a+1}}_{b,\mo{d-1}}=\Nf{b}{\mo{a+1}}{a}{d}_{\mo{d-1},b}$
\end{enumerate}
From now on, let us refer to the case $b \not\in\{\mo{a-1}, \mo{a+2}, \mo{d+1}\}$ and $d \not\in\{\mo{a-1}, \mo{a+2}\}$ as the generic case, and otherwise as the special case. 

\subsubsection{Generic case, $a \in \pint{d,b}$} 
We claim that $\Nf{b}{d}{a}{\mo{a+1}}_{i,j} - \Nf{a}{\mo{a+1}}{b}{d}_{j,i} \in \{0,1\}$ for any $i,j$. Indeed, if the following sequence of moves is possible:
$$v=\ytableaushort{\cdots a b  \cdots,\cdots d {\scriptstyle \mo{a+1}}  \cdots} \quad \rightsquigarrow \quad u=\ytableaushort{\cdots {\scriptstyle \mo{a+1}} b  \cdots,\cdots d a \cdots} \quad \rightsquigarrow \quad w=\ytableaushort{\cdots d {\scriptstyle \mo{a+1}} \cdots,\cdots a b \cdots}, $$ then so is 
$$v=\ytableaushort{\cdots a b  \cdots,\cdots d {\scriptstyle \mo{a+1}}  \cdots} \quad \rightsquigarrow \quad u=\ytableaushort{\cdots d a  \cdots,\cdots {\scriptstyle \mo{a+1}} b \cdots} \quad \rightsquigarrow \quad w=\ytableaushort{\cdots d {\scriptstyle \mo{a+1}} \cdots,\cdots a b \cdots} $$
This is because swapping $a$ and $\mo{a+1}$ is always possible, although irrelevant if $\{i,j\} = \{b,\mo{d-1}\}$, in which case the above claim becomes $0-0=0$. On the other hand, having $\mo{a+1}$ instead of $a$ in $u^2$ does not prevent items (a), (b), (c), and (e) in \ref{sec:defawg} from holding - in part thanks to our assumptions such as $d \not = \mo{a+1}$, while making inequalities in (d) more likely to hold. 

Similarly, we claim that 
\begin{itemize}
 \item if $\{i,j\} = \{a,\mo{d-1}\}$ then $\Nf{b}{\mo{a+1}}{a}{d}_{i,j} - \Nf{a}{d}{b}{\mo{a+1}}_{j,i}$ takes values either $1$ or $0$;
 \item if $\{i,j\} = \{a,b\}$ then $\Nf{a}{d}{b}{\mo{a+1}}_{i,j} - \Nf{b}{\mo{a+1}}{a}{d}_{j,i}$ takes values either $1$ or $0$;
 \item finally if $\{i,j\} = \{b,\mo{d-1}\}$ then $\Nf{b}{\mo{a+1}}{a}{d}_{i,j} - \Nf{a}{d}{b}{\mo{a+1}}_{j,i} = 0$.
\end{itemize}
The first two claims are true because for the second move condition (c) will be violated. For the last claim, the two swaps are completely independent of each other, and thus either $\Nf{b}{\mo{a+1}}{a}{d}_{i,j}  = \Nf{a}{d}{b}{\mo{a+1}}_{j,i} = 0$ or $\Nf{b}{\mo{a+1}}{a}{d}_{i,j}  =\Nf{a}{d}{b}{\mo{a+1}}_{j,i}  = 1.$
Thus the case $\{i,j\} = \{b,\mo{d-1}\}$ is verified. It remains to prove the following lemma.

\begin{lem}We have:
\begin{enumerate}
\item $\Nf{b}{\mo{a+1}}{a}{d}_{\mo{d-1},a} - \Nf{a}{d}{b}{\mo{a+1}}_{a,\mo{d-1}} = 1$ iff $\Nf{b}{d}{a}{\mo{a+1}}_{a,\mo{d-1}} - \Nf{a}{\mo{a+1}}{b}{d}_{\mo{d-1},a}=1$.
\item $\Nf{a}{d}{b}{\mo{a+1}}_{b,a} - \Nf{b}{\mo{a+1}}{a}{d}_{a,b} = 1$ iff $\Nf{b}{d}{a}{\mo{a+1}}_{a,b} - \Nf{a}{\mo{a+1}}{b}{d}_{b,a}=1$.
\end{enumerate}
\end{lem}
\begin{proof} Here we give proof when $\{i,j\} = \{a,\mo{d-1}\}$, but the case of $\{i,j\} = \{a,b\}$ is essentially verbatim after replacing $\mo{d-1} \in v^1$ with $\mo{b+1} \in v^2$. 
In this case it suffices to assume $\mo{d-1} \in v^1$. We have $\Nf{b}{d}{a}{\mo{a+1}}_{a,\mo{d-1}} - \Nf{a}{\mo{a+1}}{b}{d}_{\mo{d-1},a}=1$ if and only if the following conditions of \ref{sec:defawg} hold.
\begin{enumerate}[label=(\alph*)]
\item $\mo{b-d}$ is odd.
\item $\mo{b-1} \in v^2$, $\mo{d+1} \in v^1$.
\item is mute since we already know that $\mo{d-1} \in v^1$.
\item $\#\{\alpha \in v^2 \mid \alpha \in \pint{\mo{b-1-2k}, \mo{b-2}}\} \geq k$ for $k\in \{1, 2, \ldots, \frac{\mo{b-d}-3}{2}\}$.
\item $\#\{ \alpha \in v^2 \mid \alpha \in \pint{\mo{d+2}, \mo{b-2}}\} = \frac{\mo{b-d}-3}{2}$.
\item $\mo{b-a}$ is even and $\#\{ \alpha \in v^2 \mid \alpha \in \pint{\mo{a+1}, \mo{b-2}}\} = \frac{\mo{b-a}-2}{2}$.
\end{enumerate}
The last condition comes from the fact that 2.(d) in \ref{sec:defawg} has to fail after we swap $a$ and $\mo{a+1}$. Schematically, the exchanges that contribute to $\Nf{b}{d}{a}{\mo{a+1}}_{a, \mo{d-1}}$ look as follows:
$$v=\ytableaushort{\cdots {\scriptstyle \mo{d-1}} {\scriptstyle \mo{d+1}} a {\scriptstyle \mo{a+2}} b  \cdots,\cdots d {\scriptstyle \mo{a-1}} {\scriptstyle \mo{a+1}} {\scriptstyle \mo{b-1}}  \cdots \cdots} \rightsquigarrow  u=\ytableaushort{\cdots {\scriptstyle \mo{d-1}} d {\scriptstyle \mo{d+1}} a {\scriptstyle \mo{a+2}}  \cdots,\cdots  {\scriptstyle \mo{a-1}} {\scriptstyle \mo{a+1}} {\scriptstyle \mo{b-1}} b \cdots \cdots} \rightsquigarrow w=\ytableaushort{\cdots {\scriptstyle \mo{d-1}} d {\scriptstyle \mo{d+1}} {\scriptstyle \mo{a+1}} {\scriptstyle \mo{a+2}}  \cdots,\cdots  {\scriptstyle \mo{a-1}} a {\scriptstyle \mo{b-1}} b \cdots \cdots}.$$ 

On the other hand, we have $\Nf{b}{\mo{a+1}}{a}{d}_{\mo{d-1},a} - \Nf{a}{d}{b}{\mo{a+1}}_{a,\mo{d-1}} = 1$ if and only if the following conditions of \ref{sec:defawg}  hold.
\begin{enumerate}[label=(\alph*')]
\item $\mo{b-a}$ is even; $\mo{a-d}$ is odd.
\item $\mo{b-1} \in v^2$, $\mo{a+2} \in v^1$, $\mo{a-1} \in v^2$, $\mo{d+1} \in v^1$.
\item is mute since we already know that $a, \mo{d-1} \in v^1$.
\item $\#\{ \alpha \in v^2 \mid \alpha \in \pint{\mo{b-1-2k}, \mo{b-2}}\} \geq k$ for $k\in \{1, 2, \ldots, \frac{\mo{b-a-1}-3}{2}\}$; 
\\$\#\{\alpha \in v^2 \mid \alpha \in \pint{\mo{a-1-2k}, \mo{a-2}}\} \geq k$ for $k\in \{1, 2, \ldots, \frac{\mo{a-d}-3}{2}\}$
\item $\#\{\alpha \in v^2 \mid \alpha \in \pint{\mo{a+3}, \mo{b-2}}\} = \frac{\mo{b-a-1}-3}{2}$;
\\ $\#\{ \alpha \in v^2 \mid \alpha \in \pint{\mo{d+2}, \mo{a-2}}\} = \frac{\mo{a-d}-3}{2}$.
\end{enumerate}
Schematically, the exchanges that contribute to $\Nf{b}{\mo{a+1}}{a}{d}_{\mo{d-1},a}$ look as follows:
$$v=\ytableaushort{\cdots {\scriptstyle \mo{d-1}} {\scriptstyle \mo{d+1}} a {\scriptstyle \mo{a+2}} b  \cdots,\cdots d {\scriptstyle \mo{a-1}} {\scriptstyle \mo{a+1}} {\scriptstyle \mo{b-1}}  \cdots\cdots} \rightsquigarrow  u=\ytableaushort{\cdots {\scriptstyle \mo{d-1}} {\scriptstyle \mo{d+1}} a {\scriptstyle \mo{a+1}} {\scriptstyle \mo{a+2}}   \cdots,\cdots d {\scriptstyle \mo{a-1}} {\scriptstyle \mo{b-1}} b \cdots\cdots} \rightsquigarrow w=\ytableaushort{\cdots {\scriptstyle \mo{d-1}} d {\scriptstyle \mo{d+1}} {\scriptstyle \mo{a+1}} {\scriptstyle \mo{a+2}}  \cdots,\cdots  {\scriptstyle \mo{a-1}} a {\scriptstyle \mo{b-1}} b \cdots\cdots}.$$ 

Now we observe the following. The parity part of claims (a) and (f) is equivalent to claim (a'). Claim (b') is implied by (b) as well as (d), (f). Indeed, $\mo{a+2} \in v^1$ is implied by (f) and (d) for $k=\frac{\mo{b-a-1}-3}{2}$, while $\mo{a-1} \in v^2$ is implied by (f) and (d) for $k=\frac{\mo{b-a-1}+1}{2}$. Claim (d') is implied by (d) and (f), while (e') is implied by (e) and (f). Claim (b) is trivially implied by (b'). Claim (d) is implied by (d') and the part of (b') that refers to $\mo{a-1}$ and $\mo{a+2}$. Similarly, (e) and (f) are implied by (e') and the part of (b') that refers to $\mo{a-1}$ and $\mo{a+2}$. 
\end{proof}

\subsubsection{Generic case, $a \not \in \pint{d,b}$} 
Due to our assumptions $a \not = \mo{b+1}$ and $\mo{a+1} \not = \mo{d-1}$ in this case the exchange $b \intch d$ either can or cannot be performed independently of whether we perform $a \intch \mo{a+1}$ first or not. 
Thus $\Nf{b}{d}{a}{\mo{a+1}}_{i,j} = \Nf{a}{\mo{a+1}}{b}{d}_{j,i}$ for any $i,j$. It remains to argue that $\Nf{a}{d}{b}{\mo{a+1}}_{i,j} = \Nf{b}{\mo{a+1}}{a}{d}_{j,i}$.

Consider first the case $\{i,j\} = \{a,\mo{d-1}\}$, in which case it suffices to assume $\mo{d-1} \in v^1$.
We need to argue that $\Nf{b}{\mo{a+1}}{a}{d}_{\mo{d-1},a} = 0$. If we assume otherwise, i.e. $\Nf{b}{\mo{a+1}}{a}{d}_{\mo{d-1},a} = 1$, then one can conclude, repeatedly using condition (b), that the exchanges look as follows:
$$v=\ytableaushort{\cdots {\scriptstyle \mo{d-1}} {\scriptstyle \mo{d+1}} b a \cdots,\cdots d {\scriptstyle \mo{b-1}} {\scriptstyle \mo{a-1}} {\scriptstyle \mo{a+1}}  \cdots} \rightsquigarrow  u=\ytableaushort{\cdots {\scriptstyle \mo{d-1}} {\scriptstyle \mo{d+1}} a {\scriptstyle \mo{a+1}} \cdots,\cdots d {\scriptstyle \mo{b-1}} b {\scriptstyle \mo{a-1}}   \cdots} \rightsquigarrow w=\ytableaushort{\cdots {\scriptstyle \mo{d-1}} d {\scriptstyle \mo{d+1}} {\scriptstyle \mo{a+1}} {\scriptstyle \mo{a+2}}  \cdots,\cdots  {\scriptstyle \mo{a-1}} a {\scriptstyle \mo{b-1}} b \cdots\cdots}$$ 
By condition (a) we know that $\mo{a-d}$ is odd; let $\mo{a-d}=2m+1$. The following two conditions hold.
\begin{enumerate}
\item[(d)] $\#\{\alpha \in u^2 \mid \alpha \in \pint{\mo{a-1-2k}, \mo{a-2}}\} \geq k$ for $k\in \{1, 2, \ldots, m-1\}$,
\item[(e)] $\#\{\alpha \in u^2 \mid \alpha \in \pint{\mo{d+2}, \mo{a-2}}\} = m-1$.
\end{enumerate}

Assume $\mo{b-d}$ is odd, say $\mo{b-d} = 2 \ell +1$. Then taking $k = m - \ell - 1$ we see that $\#\{\alpha \in u^2 \mid  \alpha \in \pint{\mo{b+1}, \mo{a-2}}\} \geq m - \ell - 1$. This implies that 
$\#\{\alpha \in u^2 \mid \alpha \in \pint{\mo{d+2}, \mo{b}}\} \leq \ell$, which in turn means that $\#\{\alpha \in v^2 \mid \alpha \in \pint{\mo{d+2}, \mo{b-2}}\} \leq \ell-2$. This is impossible however by $k = \ell-1$ case of the condition 
\begin{enumerate}
\item[(d)] $\#\{\alpha \in v^2\mid \alpha \in \pint{\mo{b-1-2k}, \mo{b-2}}\} \geq k$ for $k\in \{1, 2, \ldots, \frac{\mo{b-a-1}-3}{2}\}$.
\end{enumerate}

Now assume $\mo{b-d}$ is even, say $\mo{b-d} = 2 \ell$. Then taking $k = m - \ell$ we see that  $\#\{\alpha \in u^2 \mid \alpha \in \pint{\mo{b}, \mo{a-2}}\} \geq m - \ell$. This implies that 
$\#\{ \alpha \in v^2 \mid \alpha \in [\mo{d+2}, \mo{b-1}]\} \leq \ell - 1$, which in turn means that $\#\{\alpha \in v^2 \mid \alpha \in \pint{\mo{d+1}, \mo{b-2}}\} \leq \ell-2$. This is impossible however by $k = \ell-1$ case of the condition (d) above.

Consider now the case $\{i,j\} = \{a,b\}$, which means $\mo{b+1} \in v^2$. 
We need to argue that $\Nf{a}{d}{b}{\mo{a+1}}_{b,a} = 0$. If we assume otherwise, i.e. $\Nf{a}{d}{b}{\mo{a+1}}_{b,a} = 1$, then one can conclude, repeatedly using condition (b), that the exchanges look as follows:
$$v=\ytableaushort{\cdots {\scriptstyle \mo{d+1}} b a {\scriptstyle\mo{a+2}} \cdots\cdots,\cdots d {\scriptstyle \mo{b-1}} {\scriptstyle \mo{b+1}} {\scriptstyle \mo{a-1}} {\scriptstyle \mo{a+1}} \cdots} \rightsquigarrow  u=\ytableaushort{\cdots d {\scriptstyle \mo{d+1}} b {\scriptstyle \mo{a+2}} \cdots\cdots,\cdots {\scriptstyle \mo{b-1}} {\scriptstyle \mo{b+1}} {\scriptstyle \mo{a-1}} a {\scriptstyle \mo{a+1}} \cdots} \rightsquigarrow w=\ytableaushort{\cdots d {\scriptstyle \mo{d+1}} {\scriptstyle \mo{a+1}} {\scriptstyle \mo{a+2}} \cdots\cdots,\cdots {\scriptstyle \mo{b-1}} b {\scriptstyle \mo{b+1}} {\scriptstyle \mo{a-1}} a \cdots}$$ 
By condition (a) we know that $\mo{a-d}$ is odd; let $\mo{a-d}=2m+1$. The following two conditions hold.
\begin{enumerate}
\item[(d)] $\#\{\alpha \in v^2 \mid \alpha \in \pint{\mo{a-1-2k}, \mo{a-2}}\} \geq k$ for $k\in \{1, 2, \ldots, m-1\}$,
\item[(e)] $\#\{\alpha \in v^2 \mid \alpha \in \pint{\mo{d+2}, \mo{a-2}}\} = m-1$.
\end{enumerate}

Assume $\mo{b-d}$ is odd, say $\mo{b-d} = 2 \ell +1$. Then taking $k = m - \ell - 1$ we see that $\#\{ \alpha \in v^2\mid \alpha \in \pint{\mo{b+1}, \mo{a-2}}\} \geq m - \ell - 1$. This implies that 
$\#\{ \alpha \in v^2 \mid \alpha \in \pint{\mo{d+2}, \mo{b}}\} \leq \ell$, which in turn means that $\#\{\alpha \in u^2 \mid \alpha \in \pint{\mo{d}, \mo{b-2}}\} \leq \ell-1$. This is impossible however by $k = \ell$ case of the condition 
\begin{enumerate}
\item[(d)] $\#\{ \alpha \in u^2 \mid \alpha \in \pint{\mo{b-1-2k}, \mo{b-2}}\} \geq k$ for $k\in \{1, 2, \ldots, \frac{\mo{b-a-1}-3}{2}\}$.
\end{enumerate}

Now assume $\mo{b-d}$ is even, say $\mo{b-d} = 2 \ell$. Then taking $k = m - \ell$ we see that  $\#\{\alpha \in v^2 \mid \alpha \in \pint{\mo{b}, \mo{a-2}}\} \geq m - \ell$. This implies that 
$\#\{\alpha \in v^2 \mid \alpha \in \pint{\mo{d+2}, \mo{b-1}}\} \leq \ell - 1$, which in turn means that $\#\{ \alpha \in u^2 \mid \alpha \in \pint{\mo{d+1}, \mo{b-2}}\} \leq \ell-2$. This is impossible however by $k = \ell-1$ case of the condition (d) above.

Finally, consider the case $\{i,j\} = \{b, \mo{d-1}\}$, in which case we may assume that $\mo{d-1} \in v^1, \mo{b+1} \in v^2$. In this case we claim that 
$\Nf{a}{d}{b}{\mo{a+1}}_{i,j} = \Nf{b}{\mo{a+1}}{a}{d}_{j,i}=0$. The argument for $\Nf{a}{d}{b}{\mo{a+1}}_{i,j} = 0$ coincides verbatim with the same argument in the case $\{i,j\} = \{a,b\}$, while the argument for $\Nf{b}{\mo{a+1}}{a}{d}_{j,i}= 0$ coincides verbatim with the same argument in the case $\{i,j\} = \{a,\mo{d-1}\}$.

\subsubsection{Special cases}

In the $d=\mo{a+2}$ case the same argument works verbatim as in the generic, $a \not \in \pint{d,b}$, $\{i,j\} = \{a,b\}$ case. 
In the $d=\mo{a-1}$ case the same argument works verbatim as in the generic, $a \in \pint{d,b}$ case. 
In the $b=\mo{a-1}$ case the same argument works verbatim as in the generic, $a \not \in \pint{d,b}$, $\{i,j\} = \{a,\mo{d-1}\}$ case.
In the $b=\mo{a+2}$ case the same argument works verbatim as in the generic, $a \in \pint{d,b}$ case. 

Finally, consider $b=\mo{d+1}$ case. Then $\{i,j\} = \{a, \mo{d+1}\}$, or $\{i,j\} = \{a, \mo{d-1}\}$, or $\{i,j\} = \{\mo{d-1}, \mo{d+1}\}$.
In the first two cases $\Nf{a}{\mo{a+1}}{\mo{d+1}}{d}_{\mo{d\pm1},a} = \Nf{\mo{d+1}}{d}{a}{\mo{a+1}}_{a, \mo{d\pm1}} =1$. 
Also in all three cases $\Nf{a}{d}{\mo{d+1}}{\mo{a+1}}_{i,j} = \Nf{\mo{d+1}}{\mo{a+1}}{a}{d}_{i,j} =0$ because condition (b) gets violated for the second exchange. The claim follows.

\subsection{$|\{c, d\} \cap \{\mo{a+1}, \mo{b+1}\}|=0$ case}
If $(i,j)$ satisfies $i,j \in \ades(v)$ and $i,j \not\in \ades(w)$ then we have $\{i,j\} \subset \{a, b, \mo{c-1}, \mo{d-1}\}$. We divide it into two cases: the case when $\#(\pint{a, b}\cap\{c,d\})=1$, which we call the interlacing case, and the other case called the non-interlacing case. Let us list all the possibilities and the equalities needed to be proved:
\begin{enumerate}[label=$\bullet$]
\item if $\{i,j\}=\{a,b\}$, then 
$\Nf{b}{c}{a}{d}_{a,b}+\Nf{b}{d}{a}{c}_{a,b}=\Nf{a}{c}{b}{d}_{b,a}+\Nf{a}{d}{b}{c}_{b,a}$.
\item if $\{i,j\}=\{a,\mo{c-1}\}$, then 
$\Nf{b}{c}{a}{d}_{a,\mo{c-1}}=\Nf{a}{d}{b}{c}_{\mo{c-1},a}$.
\item if $\{i,j\}=\{a,\mo{d-1}\}$, then 
$\Nf{b}{d}{a}{c}_{a,\mo{d-1}}=\Nf{a}{c}{b}{d}_{\mo{d-1},a}$.
\item if $\{i,j\}=\{b,\mo{c-1}\}$, then 
$\Nf{a}{c}{b}{d}_{b,\mo{c-1}}=\Nf{b}{d}{a}{c}_{\mo{c-1},b}$.
\item if $\{i,j\}=\{b,\mo{d-1}\}$, then 
$\Nf{a}{d}{b}{c}_{b,\mo{d-1}}=\Nf{b}{c}{a}{d}_{\mo{d-1},b}$.
\item if $\{i,j\}=\{\mo{c-1},\mo{d-1}\}$, then 
$\Nf{a}{d}{b}{c}_{\mo{c-1},\mo{d-1}}+\Nf{b}{d}{a}{c}_{\mo{c-1},\mo{d-1}}=\Nf{a}{c}{b}{d}_{\mo{d-1},\mo{c-1}}+\Nf{b}{c}{a}{d}_{\mo{d-1},\mo{c-1}}$.
\end{enumerate}

\subsubsection{Non-interlacing case} 
Without loss of generality we may set $1\leq c < d < a < b \leq n$. Note that in this case it suffices to prove the following lemma.
\begin{lem}
For any $i,j$, we have $\Nf{a}{d}{b}{c}_{i,j} = \Nf{b}{c}{a}{d}_{j,i}$ and $\Nf{a}{c}{b}{d}_{i,j} = \Nf{b}{d}{a}{c}_{j,i}$.
\end{lem}

Let us start with the first equality. Note that $\Nf{a}{d}{b}{c}_{i,j}  = 1$ for $i \in \{b, \mo{c-1}\}$ and $j\in \{a,\mo{d-1}\}$ if and only if the following conditions hold. 
\begin{enumerate}[label=(\alph*)]
\item $a-d$ and $b-c$ are odd, set $b-c = 2m+1$, $a-d = 2 \ell +1$.
\item $\mo{a-1}, \mo{b-1} \in v^2$, $\mo{c+1}, \mo{d+1} \in v^1$.
\item One out of two conditions holds: $\mo{d-1} \in v^1$, $\mo{a+1} \in v^2$, and also one of the two conditions holds: $\mo{c-1} \in v^1$, $\mo{b+1} \in v^2$. 
\item $\#\{\alpha \in v^2 \mid \alpha \in \pint{\mo{a-1-2k},\mo{a-2}}\} \geq k$ for $k\in \{1, 2, \ldots, \ell-1\}$;
\\$\#\{\alpha \in u^2 \mid \alpha \in \pint{\mo{b-1-2k}, \mo{b-2}}\} \geq k$ for $k\in \{1, 2, \ldots, m-1\}$.
\item $\#\{\alpha \in v^2 \mid \alpha \in \pint{\mo{d+2},\mo{a-2}}\} = \ell-1$;
\\$\#\{\alpha \in u^2 \mid \alpha \in \pint{\mo{c+2},\mo{b-2}}\} = m-1$.
\end{enumerate}
On the other hand, $\Nf{b}{c}{a}{d}_{j,i}= 1$ for $i \in \{b, \mo{c-1}\}$ and $j\in \{a,\mo{d-1}\}$ if and only if the following conditions hold. 
\begin{enumerate}[label=(\alph*')]
\item $a-d$ and $b-c$ are odd, set $b-c = 2m+1$, $a-d = 2 \ell +1$.
\item $\mo{a-1}, \mo{b-1} \in v^2$, $\mo{c+1}, \mo{d+1} \in v^1$.
\item One out of two conditions holds: $\mo{d-1} \in v^1$, $\mo{a+1} \in v^2$, and also one of the two conditions holds: $\mo{c-1} \in v^1$, $\mo{b+1} \in v^2$. 
\item $\#\{\alpha \in u^2 \mid \alpha \in \pint{\mo{a-1-2k}, \mo{a-2}}\} \geq k$ for $k\in \{1, 2, \ldots, \ell-1\}$;
\\$\#\{\alpha \in v^2 \mid \alpha \in \pint{\mo{b-1-2k}, \mo{b-2}}\} \geq k$ for $k\in \{1, 2, \ldots, m-1\}$.
\item $\#\{ \alpha \in u^2 \mid \alpha \in \pint{\mo{d+2}, \mo{a-2}}\} = \ell-1$;
\\$\#\{\alpha \in v^2 \mid \alpha \in \pint{\mo{c+2}, \mo{b-2}}\} = m-1$.
\end{enumerate}

It is clear that we need to show equivalence between (d), (e) on one hand and (d'), (e') on the other. Assume first that $b-a$ is even. It is easy to see that for any $k\in \{1, 2, \ldots, m-1\}$ we have 
$$\#\{\alpha \in v^2 \mid \alpha \in \pint{\mo{b-1-2k}, \mo{b-2}}\} = \#\{ \alpha \in u^2 \mid \alpha \in \pint{\mo{b-1-2k}, \mo{b-2}}\},$$
where we assume that $u$ is obtained from $v$ by swapping $a$ and $d$. Indeed, in each pair $\{d, \mo{d+1}\}$ and $\{\mo{a-1}, a\}$ exactly one element belongs to $v^2$ and $u^2$, and thus the overall counts are the same, no matter what $k$ is. All other conditions needed for the equivalence are also clear. 

Assume now that $b-a$ is odd, say $b-a = 2p+1$. Recall that $v$ and $u$ differ by $a \in v^1$, $d \in v^2$, while $a \in u^2$, $d \in v^1$. It is clear that if (d') and (e') hold, then so do (d) and (e). In the opposite direction, there is only one thing that could go wrong. Namely, it is possible that 
$\#\{\alpha \in v^2 \mid \alpha \in \pint{a, \mo{b-2}}\} < p$, while at the same time $\#\{ \alpha \in u^2 \mid \alpha \in [a, \mo{b-2}]\} = p$. Thanks to the condition (e) this implies that $\#\{ \alpha \in u^2 \mid \alpha \in \pint{\mo{d+1}, \mo{b-2}}\} = \ell + p$. Then the only way one can have $\#\{ \alpha \in u^2 \mid \alpha \in [\mo{d-1}, \mo{b-2}]\} \geq \ell + p + 1$ is if $\mo{d-1} \in v^2$. By (c) this means that $\mo{a+1} \in v^2$. This however contradicts $\#\{\alpha \in v^2 \mid \alpha \in [a, \mo{b-2}]\} < p$, since we also know $\#\{ \alpha \in v^2 \mid \alpha \in \pint{\mo{a+2}, \mo{b-2}}\} \geq p-1$. Thus our assumption was wrong and (d') holds. The desired equivalence is now clear. 

Now we will prove that $\Nf{a}{c}{b}{d}_{i,j} = \Nf{b}{d}{a}{c}_{j,i}$ for any $i,j$ by proving that both of them are $0$. Indeed, assume $\Nf{a}{c}{b}{d}_{i,j}= 1$ for some $i,j$. Then 
\begin{enumerate}[label=(\alph*)]
\item $a-c$ and $b-d$ are odd, set $a-c = 2m+1$, $b-d = 2 \ell +1$.
\item $\mo{a-1}, \mo{b-1} \in v^2$, $\mo{c+1}, \mo{d+1} \in v^1$.
\item One out of two conditions holds: $\mo{c-1} \in v^1$, $\mo{a+1} \in v^2$, and also one of the two conditions holds: $\mo{d-1} \in v^1$, $\mo{b+1} \in v^2$. 
\item $\#\{\alpha \in v^2 \mid \alpha \in \pint{\mo{a-1-2k}, \mo{a-2}}\} \geq k$ for $k\in \{1, 2, \ldots, m-1\}$;
\\$\#\{ \alpha \in u^2 \mid \alpha \in \pint{\mo{b-1-2k}, \mo{b-2}}\} \geq k$ for $k\in \{1, 2, \ldots, \ell-1\}$.
\item $\#\{\alpha \in v^2 \mid \alpha \in \pint{\mo{c+2}, \mo{a-2}}\} = m-1$;
\\$\#\{ \alpha \in u^2 \mid \alpha \in \pint{\mo{d+2},\mo{b-2}}\} = \ell-1$.
\end{enumerate}

Assume $b-a$ is even, say $b-a=2p$. Then combining $\#\{ \alpha \in u^2 \mid \alpha \in \pint{\mo{a+1}, \mo{b-2}}\} \geq p-1$ with $\#\{\alpha \in u^2 \mid \alpha \in \pint{\mo{d+2},\mo{a-2}}\} = \#\{\alpha \in v^2 \mid \alpha \in \pint{\mo{d+2}, \mo{a-2}}\} \geq  \ell-p-1$ and $\mo{a-1}, a \in u^2$, we see that $\#\{\alpha \in u^2 \mid \alpha \in \pint{\mo{d+2}, \mo{b-2}}\} \geq \ell-p-1+p-1+2 = \ell$, which contradicts (e). Now assume $b-a$ is odd, say $b-a=2p+1$. Then combining $\#\{\alpha \in u^2 \mid \alpha \in \pint{a, \mo{b-2}}\} \geq  p$ with $\#\{\alpha \in u^2 \mid \alpha \in \pint{\mo{d+1},\mo{a-2}}\} = \#\{ \alpha \in v^2 \mid \alpha \in \pint{\mo{d+1}, \mo{a-2}}\} \geq  \ell-p-1$ and $\mo{a-1} \in u^2$, $\mo{d+1} \not \in u^2$ we see that $\#\{ \alpha \in u^2 \mid \alpha \in \pint{\mo{d+2},\mo{b-2}}\} \geq \ell-p-1+p+1 = \ell$, which contradicts (e). The contradictions show that $\Nf{a}{c}{b}{d}_{i,j} =0$ for any $i,j$.

Finally, assume $\Nf{b}{d}{a}{c}_{j,i} = 1$ for some $i,j$. Then 
\begin{enumerate}[label=(\alph*)]
\item $a-c$ and $b-d$ are odd, set $a-c = 2m+1$, $b-d = 2 \ell +1$.
\item $\mo{a-1}, \mo{b-1} \in v^2$, $\mo{c+1}, \mo{d+1} \in v^1$.
\item One out of two conditions holds: $\mo{c-1} \in v^1$, $\mo{a+1}\in v^2$, and also one of the two conditions holds: $\mo{d-1} \in v^1$, $\mo{b+1} \in v^2$. 
\item $\#\{ \alpha \in u^2 \mid \alpha \in \pint{\mo{a-1-2k},\mo{a-2}}\} \geq k$ for $k\in \{1, 2, \ldots, m-1\}$;
\\$\#\{ \alpha \in v^2 \mid \alpha \in \pint{\mo{b-1-2k}, \mo{b-2}}\} \geq k$ for $k\in \{1, 2, \ldots, \ell-1\}$.
\item $\#\{\alpha \in u^2 \mid \alpha \in \pint{\mo{c+2},\mo{a-2}}\} = m-1$;
\\$\#\{\alpha \in v^2 \mid \alpha \in \pint{\mo{d+2},\mo{b-2}}\} = \ell-1$.
\end{enumerate}
Assume $b-a$ is even, say $b-a=2p$. Then  $\#\{ \alpha \in v^2 \mid \alpha \in \pint{\mo{a+1},\mo{b-2}}\} \geq  p-1$ and $\mo{a-1} \in v^2, a \not \in v^2$ and $\#\{ \alpha \in v^2 \mid \alpha \in \pint{\mo{d+2}, \mo{b-2}}\} = \ell-1$ imply $\#\{\alpha \in v^2 \mid \alpha \in \pint{\mo{d+2},\mo{a-2}}\} \leq \ell-p-1$. However, since $d, \mo{d+1} \not \in u^2$, this implies $\#\{\alpha \in u^2 \mid \alpha \in\pint{d, \mo{a-2}}\} \leq \ell-p-1$, which contradicts (d).  Now assume $b-a$ is odd, say $b-a=2p+1$. Then $\#\{ \alpha \in u^2 \mid \alpha \in \pint{a, \mo{b-2}}\} \geq p$ and $\mo{a-1} \in v^2$ and $\#\{\alpha \in v^2 \mid \alpha \in \pint{\mo{d+2}, \mo{b-2}}\} = \ell-1$ imply
$\#\{\alpha \in u^2 \mid \alpha \in \pint{\mo{d+1}, \mo{a-2}}\} \leq \ell-p-2$, which contradicts (d). The contradictions show that $\Nf{b}{d}{a}{c}_{j,i} = 0$ for any $i,j$.

\subsubsection{Interlacing case} 
Without loss of generality we may set $1\leq c < a < d < b \leq n$. Similarly to above, in this case it suffices to prove the following lemma.

\begin{lem}
For any $i,j$, we have $\Nf{a}{d}{b}{c}_{i,j} = \Nf{b}{c}{a}{d}_{j,i}$ and $\Nf{a}{c}{b}{d}_{i,j} = \Nf{b}{d}{a}{c}_{j,i}$.
\end{lem}
The equality $\Nf{a}{c}{b}{d}_{i,j} = \Nf{b}{d}{a}{c}_{j,i}$ is self-evident because of our assumptions $d \not = \mo{a+1}, c \not = \mo{b+1}$. We argue that $\Nf{a}{d}{b}{c}_{i,j} = \Nf{b}{c}{a}{d}_{j,i}=0$. In fact, due to circular symmetry it is enough to just argue one of those, say $\Nf{a}{d}{b}{c}_{i,j}  = 0$. Assume otherwise, i.e. $\Nf{a}{d}{b}{c}_{i,j} = 1$. Then 
\begin{enumerate}
\item $b-c$ and $a-d$ are odd, set $b-c = 2m+1$, $a-d = 2 \ell +1$.
\item $\mo{a-1}, \mo{b-1} \in v^2$, $\mo{c+1}, \mo{d+1} \in v^1$.
\item One out of two conditions holds: $\mo{c-1} \in v^1$, $\mo{b+1} \in v^2$, and also one of the two conditions holds: $\mo{d-1} \in v^1$, $\mo{a+1} \in v^2$. 
\item $\#\{ \alpha \in v^2\mid \alpha \in \pint{\mo{a-1-2k},\mo{a-2}}\} \geq k$ for $k\in \{1, 2, \ldots, \ell-1\}$;
\\$\#\{\alpha \in u^2\mid \alpha \in \pint{\mo{b-1-2k},\mo{b-2}}\} \geq k$ for $k\in \{1, 2, \ldots, m-1\}$.
\item $\#\{\alpha \in v^2\mid \alpha \in \pint{\mo{d+2},\mo{a-2}}\} = \ell-1$;
\\$\#\{\alpha \in u^2\mid \alpha \in \pint{\mo{c+2},\mo{b-2}}\} = m-1$.
\end{enumerate}

Assume $a-b$ is even, say $a-b=2p$. Then  $\#\{ \alpha \in v^2 \mid \alpha \in \pint{\mo{b+1}, \mo{a-2}}\} \geq  p-1$, which together with $\mo{b-1} \in v^2, b \not \in v^2$ and $\#\{ \alpha \in v^2 \mid \alpha \in \pint{\mo{d+2}, \mo{a-2}}\} = \ell-1$ implies $\#\{ \alpha \in v^2 \mid \alpha \in \pint{\mo{d+2}, \mo{b-2}}\} \leq \ell-p-1$. Since $d, \mo{d+1} \in u^1$, this implies that $\#\{\alpha \in u^2 \mid \alpha \in \pint{d, \mo{b-2}}\} \leq \ell-p-1$, which contradicts (d).  Now assume $a-b$ is odd, say $a-b=2p+1$. Then  $\#\{\alpha \in v^2\mid \alpha \in \pint{b, \mo{a-2}}\} \geq  p$, which together with $\mo{b-1} \in v^2$ and $\#\{\alpha \in v^2 \mid \alpha \in \pint{\mo{d+2}, \mo{a-2}}\} = \ell-1$ implies $\#\{ \alpha \in v^2 \mid \alpha \in \pint{\mo{d+2}, \mo{b-2}}\} \leq \ell-p-2$. Since $\mo{d+1} \not \in u^2$, this implies that $\#\{ \alpha \in u^2 \mid \alpha \in \pint{\mo{d+1}, \mo{b-2}}\} \leq \ell-p-2$, which contradicts (d). The proof is complete.

\section{Restriction of $\awg_\lambda$ to $\Sym_n$}
Here we discuss the parabolic restriction of $\awg_\lambda$ to the maximal parabolic subgroup $\Sym_n$ of $\affSn$ when $\lambda$ is a two-row partition. (However, many parts in this section are still valid for general $\lambda$ when the existence of $\awg_\lambda$ is not needed.) As a result, for a two-row partition $\lambda$ we obtain an explicit description of a $\Sym_n$-graph $\Gamma_\lambda$ which is a finite analogue of $\awg_\lambda$.

\subsection{Left cells of $\Sym_n$ and $\Sym_n$-graphs} Suppose that $W$ is a Coxeter group. In \cite{kl79}, a $W$-graph is attached to each left cell of $W$. Furthermore, when $W=\Sym_n$ it is essentially proved by \cite[Theorem 1.4]{kl79} that the isomorphism class of such a $\Sym_n$-graph depends only on the two-sided cell containing the corresponding left cell. Recall that two-sided cells of $\Sym_n$ are parametrized by partitions of $n$; let $\tsc_\lambda$ be such a cell parametrized by $\lambda$. Here we adopt the convention that  if $w\in \tsc_\lambda$ then the image of $w$ under the usual Robinson-Schensted map is a pair of elements in $\SYT(\lambda)$. We define $\Gamma_\lambda$ to be the $\Sym_n$-graph attached to a left cell contained in $\tsc_\lambda$.

\begin{rmk} To be precise, the $\Sym_n$-graph $\Gamma_\lambda$ constructed in \cite{kl79} is not reduced but $m(u\ra v)=m(v\ra u)$ for any vertices $u$ and $v$. Here, we modify $\Gamma_\lambda$ to be reduced by setting $m(u\ra v)=0$ whenever $\tau(u)\subset \tau(v)$.
\end{rmk}

Recall the definition of a Kazhdan-Lusztig affine dual equivalence graph $\adeg_\lambda$. Then clearly $\adeg_\lambda\pres{[1,n-1]}$ is a $[1,n-1]$-labeled graph, and we set $\fdeg_\lambda$ to be its full subgraph whose vertices are standard Young tableaux of shape $\lambda$. In other words, $\fdeg_\lambda=(V, m, \tau)$ is a $[1,n-1]$-labeled graph such that $V=\SYT(\lambda)$, $\tau=\des$, $m(u\ra v)=m(v\ra u)=1$ if $u$ and $v$ are connected by a Knuth move, and $m(u\ra v)=m(v\ra u)=0$ otherwise. (See \ref{sec:desknuth} for the definition of $\des$ and Knuth moves, and also the remark thereafter.) The graph $\fdeg_\lambda$ is called a Kazhdan-Lusztig (finite) dual equivalence graph of shape $\lambda$. Then it is known that $U(\Gamma_\lambda) \simeq \fdeg_\lambda$, e.g. see \cite[3.5]{chm15}.

\subsection{(nb-)Admissible $\Sym_n$-graphs} \label{sec:nb} Here we discuss some properties of nb-admissible $\Sym_n$-graphs. Recall that in general cells and simple components of $W$-graphs may differ; we already observed such a phenomenon in Example \ref{ex:33} (see Figure \ref{fig:33}). However, such situations do not arise for $\Sym_n$-graphs as the following result shows.
\begin{thm}[{\cite{chm15}}] \label{thm:admklund} If $\Gamma$ is an nb-admissible $\Sym_n$-graph, then each cell consists of a simple component. Moreover, the simple underlying graph of each cell is isomorphic to $\fdeg_\mu$ for some $\mu \vdash n$.
\end{thm}
\begin{proof} The result of Chmutov is stated for admissible $\Sym_n$-graphs. However, his proof does not exploit the bipartition property and thus the statement is still valid for nb-admissible setting.
\end{proof}

In fact, more is true; the following theorem was a conjecture of Stembridge \cite[Question 2.8]{ste08}.

\begin{thm}[{\cite{ngu18}}] \label{thm:admkl} If $\Gamma$ is an nb-admissible $\Sym_n$-graph, then each cell is isomorphic to $\Gamma_\mu$ for some $\mu \vdash n$.
\end{thm}
\begin{proof} Again, the proof of Nguyen is still applicable to our setting as his proof does not use the bipartition property of admissible $\Sym_n$-graphs.
\end{proof}
To this end, Nguyen studied some property of (nb-)admissible $\Sym_n$-graphs called orderedness, which we now explain. Suppose that $\Gamma$ is an (nb-)admissible $\Sym_n$-graph and let $\Gamma'$, $\Gamma''$ be (possibly identical) cells of $\Gamma$. Then by the theorem above, there exist $\mu,\nu \vdash n$ such that $\Gamma' \simeq \Gamma_\mu$ and $\Gamma'' \simeq \Gamma_\nu$ (or equivalently $U(\Gamma')\simeq \fdeg_\mu$ and $U(\Gamma'') \simeq \fdeg_\nu$). Let $u \in \Gamma'$ and $v \in \Gamma''$. Then under the previous isomorphisms, $u$ and $v$ corresponds to $T_u \in \SYT(\mu)$ and $T_v \in \SYT(\nu)$. We say that $\Gamma$ is ordered if $m(u\ra v) \neq 0$ for such $u,v$ then either [$T_u<T_v$] or [$\Gamma'=\Gamma''$, $T_u>T_v$, and $T_u$ is obtained from $T_v$ by switching $i$ and $i+1$ for some $i\in [1,n-1]$]. Here for two tableaux $T, T' \in \SYT(n)$ we write $T\leq T'$ if $\sh(T\pres{[1,i]})$ is less than or equal to $\sh(T'\pres{[1,i]})$ with respect to dominance order for all $i \in [1,n]$. (See \cite{ngu18} for actual statement.) Now we have:

\begin{thm}[{\cite[Theorem 8.1.]{ngu18}}] \label{thm:admord} Every nb-admissible $\Sym_n$-graph is ordered.
\end{thm}
\begin{proof} Similarly to the theorems above, the proof of \cite{ngu18} is still valid in our case as it does not use the bipartition assumption.
\end{proof}

\subsection{Description of $\awg_\lambda\pres{[1,n-1]}$}

As $\awg_\lambda$ is a $\affSn$-graph, its restriction $\awg_\lambda\pres{[1,n-1]}$ is a $\Sym_n$-graph where $\Sym_n$ is considered as a parabolic subgroup of $\affSn$ generated by $\{s_1, s_2, \ldots, s_{n-1}\}$. Let us investigate each cell of $\awg_\lambda\pres{[1,n-1]}$. We start with the following proposition.

\begin{prop} Let $\mu$ be a partition of $n$. Recall the Robinson-Schensted-Knuth map $\RSK: T \mapsto (P(T), Q(T))$ defined on $\RSYT(n)$.
\begin{enumerate}
\item For $T \in \RSYT(n)$, we have $\ades(T)-\{n\}=\des(T)= \des(P(T))$.
\item For $T \in \RSYT(\mu)$, we have $\fsh(T) = \mu$ if and only if $T$ is standard if and only if $T=P(T)$.
\item If $\des(T)$ and $\des(T')$ are not comparable, then $T,T' \in \RSYT(n)$ are connected by a dual Knuth move if and only if $Q(T)=Q(T')$ and $P(T)$ and $P(T')$ are connected by a Knuth move.
\end{enumerate}
\end{prop}
\begin{proof} (1) holds since the reading words of $T$ and $P(T)$ are Knuth equivalent. For (2), first it is clear from the construction that $T$ is standard only if $T=P(T)$ only if $\fsh(T)=\mu$. Now observe that $\fsh(T)=\mu$ if and only if $Q(T)$ is the unique standard Young tableaux of shape $\mu$ and content $\mu^{op}$. Therefore, (2) follows from the fact that $\RSK$ is an injective map. For (3), we set $\tilde{T}$ (resp. $\tilde{T}'$) to be the standard Young tableau of some skew-shape which is obtained from pushing each row of $T$ (resp. $T'$) to the right so that no two boxes are in the same column. Then it is clear that $\tilde{T}$ and $P(T)$ (resp. $\tilde{T}'$ and $P(T')$ are jeu-de-taquin equivalent, and also $T$ and $T'$ are connected by a dual Knuth move if and only if $\tilde{T}$ and $\tilde{T'}$ are. Now the result follows from \cite[Lemma 2.3]{hai92}.
%
\end{proof}

As $\awg_\lambda$ is nb-admissible, so is $\awg_\lambda\pres{[1,n-1]}$, which means that we may apply Theorem \ref{thm:admklund}, \ref{thm:admkl}, and \ref{thm:admord}. In particular, each cell of $\awg_\lambda\pres{[1,n-1]}$ is a simple component and isomorphic to $\fdeg_\mu$ for some $\mu \vdash n$. Therefore, if $u, v \in \awg_\lambda\pres{[1,n-1]}$ are in the same cell then they are linked by undirected edges, which means that $Q(u)=Q(v)$ by the preceding proposition. Conversely, if $Q(u)=Q(v)$ for some $u,v$ then it is clear that $P(u)$ and $P(v)$ are linked by a series of Knuth moves, which means that $u$ and $v$ are in the same cell of $\awg_\lambda\pres{[1,n-1]}$ again by the preceding proposition. 

Recall that the $\tau$-function of $\awg_\lambda\pres{[1,n-1]}$ is obtained from that of $\awg_\lambda$ by removing $n \in [1,n]$ from the image of each $v\in \awg_\lambda$. Therefore, if we regard $v\in \awg_\lambda$ as an element in $\RSYT(\lambda)$ then its $\tau$ value in $\awg_\lambda\pres{[1,n-1]}$ is equal to $\tau(P(v))$ by the preceding proposition. Together with the paragraph above, we proved the following proposition:
\begin{prop} Cells of $\awg_\lambda\pres{[1,n-1]}$ are parametrized by $\bigsqcup_{\mu \vdash n} \SSYT(\mu, \lambda^{op})$. If $\cC \subset \awg_\lambda\pres{[1,n-1]}$ is a cell parametrized by $Q$, then $\cC$ is isomorphic to $\Gamma_{\sh(Q)}$. In particular, there exists a unique cell which is isomorphic to $\Gamma_\lambda$ and it is parametrized by the unique element in $\SSYT(\lambda, \lambda^{op})$.
\end{prop}

\ytableausetup{boxsize=1em, notabloids}
\begin{figure}[htbp]
\centering
\begin{tikzpicture}[scale=1.5]

	\node[draw] (o1) at (90+144:3) {\ytableaushort{12{\dt{3}},45}};
	\node[draw] (o2) at (18+144:3) {\ytableaushort{23{\dt{4}},15}};
	\node[draw] (o3) at (306+144:3) {\ytableaushort{345,12}};
	\node[draw] (o4) at (234+144:3) {\ytableaushort{{\dt{1}}45,23}};
	\node[draw] (o5) at (162+144:3) {\ytableaushort{1{\dt{2}}5,34}};
	
	\node[draw] (i1) at (90+144:1.5) {\ytableaushort{1{\dt{2}}{\dt{4}},35}};
	\node[draw] (i2) at (18+144:1.5) {\ytableaushort{2{\dt{3}}5,14}};
	\node[draw] (i3) at (306+144:1.5) {\ytableaushort{{\dt{1}}3{\dt{4}},25}};
	\node[draw] (i4) at (234+144:1.5) {\ytableaushort{{\dt{2}}45,13}};
	\node[draw] (i5) at (162+144:1.5) {\ytableaushort{{\dt{1}}{\dt{3}}5,24}};
	
	\draw[tl] (o1) -- (i1);\draw[tl] (o2) -- (i2);\draw[arr] (i3) -- (o3);
	\draw[tl] (o4) -- (i4);\draw[tl] (o5) -- (i5);
	
	\draw[tl] (i1) -- (i3); \draw[tl] (i2) -- (i4);\draw[tl] (i3) -- (i5);
	\draw[arr] (i1) -- (i4); \draw[arr] (i5) -- (i2);
	
	\draw[arr] (i1) -- (o2);	\draw[arr, tl] (i1) -- (o5);
	\draw[arr] (i2) -- (o3);	\path[arr] (i2) -- (o1);
	\draw[arr] (i3) -- (o4);	\draw[arr] (i3) -- (o2);
	\path[arr] (i4) -- (o5);	\draw[arr] (i4) -- (o3);
	\draw[arr, tl] (i5) -- (o1);	\draw[arr] (i5) -- (o4);	
	
	
	\node at (0,-3) {$\Gamma_{(3,2)}$};
	\node at (-3.6,0.8) {$\Gamma_{(4,1)}$};
	\node at (0,3.6) {$\Gamma_{(5)}$};
	
	\draw[thick] plot [smooth] coordinates {(-.4,-3) (-2.4,-2.8) (-2.4,-2) (-0.8,0) (-0.6,2) (0.6,2) (0.8,0) (2.4,-2) (2.4,-2.8) (.4,-3)} ;
	\draw[thick] plot [smooth] coordinates {(-3.6,0.6) (-3.4,0.2) (-2,0) (2,0) (3.4,0.6) (3.4,1.4) (2.4,1.4) (0,0.6) (-2.4,1.4) (-3.4,1.4) (-3.6,1)};
	\draw[thick] plot [smooth] coordinates {(-0.3,3.6) (-0.6,3.4) (-0.6, 2.6) (0.6, 2.6) (0.6, 3.4) (0.3, 3.6)};

\end{tikzpicture}
\caption{Parabolic restriction $\awg_{(3,2)}\pres{[1,4]}$}
\label{fig:32res}
\end{figure}	
\ytableausetup{boxsize=normal, tabloids}

\begin{example}
Figure \ref{fig:32res} illustrates the parabolic restriction $\awg_{(3,2)}\pres{[1,4]}$. Here, thick edges are the ones between vertices in the same cell. Compared to Figure \ref{fig:32}, there are less directed edges and also some undirected edges become directed. It consists of three cells isomorphic to $\Gamma_{(3,2)}, \Gamma_{(4,1)},$ and $\Gamma_{(5)}$, respectively, as indicated in the figure.
\end{example}
%
%
%
%
%
%

\subsection{Description of $\Gamma_\lambda$} \label{sec:fin}
From now on we enforce that $\lambda$ is a two-row partition and identify $\Gamma_\lambda$ with the full subgraph of $\awg_\lambda\pres{[1,n-1]}$ isomorphic to it. Then similarly to $\awg_\lambda$ it is possible to give a simple combinatorial description of $\Gamma_\lambda$. (Note that the description of $\Gamma_\lambda$ can also be given in terms of the language of Temperley-Lieb algebras, see \cite{wes95}.) First we observe the following.

\begin{lem} Let $u, v\in \Gamma_\lambda$ and suppose that we have an edge $u\xrightarrow{j\intch i} v$ in $\Gamma_\lambda\subset \awg_\lambda\pres{[1,n-1]}$ for some $1\leq i,j\leq n$. Then we have $j\geq i-1$, i.e. either $j=i-1$ (a move of the first kind) or $i<j$ (a move of the second kind).
\end{lem}

\begin{proof} For contradiction suppose that $j<i-1$. Then $j$ cannot be 1 since $j \in v^2$ and $v$ is standard. Thus $j-1\geq 1$ and we require that $j-1 \in u^2$. Since $v$ is standard and $j-1, j \in v^2$, it implies that $ \#(u^2 \cap [1,j-2])+2\leq \#(u^1 \cap [1, j-2])$, or equivalently $ \#(u^2 \cap [2,j-2])+1\leq \#(u^1 \cap [2, j-2])$ as $1\in u^1$. But this violates the inequality of part 2.(d) in \ref{sec:defawg}, thus the result follows.
\end{proof}

\begin{thm} Let $\lambda \vdash n$ be a two-row partition. Then the weight function $m$ of $\Gamma_\lambda=(\SYT(\lambda), m, \des)$ is defined as follows.
\begin{enumerate}[label=\arabic*)]
\item (Move of the first kind) $m(s\ra t)=1$ if $t$ is obtained from $s$ by interchanging $i \in s^1$ and $i+1\in s^2$ for some $1\leq i \leq n-1$, i.e.
$$s=\ytableaushort{\cdots \ i \ \cdots, \cdots \  {i+1} \ \cdots} \quad \rightarrow \quad t=\ytableaushort{\cdots \ {i+1} \ \cdots, \cdots \  {i} \ \cdots}.$$
\item (Move of the second kind) $m(s\ra t)=1$ if $t$ is obtained from $s$ by interchanging $i \in s^2$ and $j \in s^1$, i.e.
$$s=\ytableaushort{\cdots j  \cdots, \cdots  i \cdots}  \quad \rightarrow \quad t=\ytableaushort{\cdots i \cdots, \cdots  j \cdots}$$
where the following conditions hold:
\begin{enumerate}[label=(\alph*)]
\item $1<i<j\leq n$ and $j-i$ is odd.
\item $i+1 \in v^1$ and $j-1\in v^2$.
\item Either $i-1 \in s^1$ or $j+1\in s^2$. (If $j=n$, then $j+1\not \in s^2$ by convention.)
\item $\#\{\alpha \in s^2 \mid  \alpha\in [j-1-2m, j-2]\} \geq m$ for $m\in [1, \frac{j-i-3}{2}]$.
\item $\#\{\alpha \in s^2 \mid \alpha \in [i+2, j-2]\} =\frac{j-i-3}{2}$ when $j\neq i+1$.
\end{enumerate}
\item Otherwise, $m(s\ra t)=0$.
\end{enumerate}
\end{thm}
\begin{proof} This directly follows from the lemma above together with the definition of $\awg_\lambda$ in \ref{sec:defawg}.
\end{proof}

\section{Uniqueness of $\awg_\lambda$ in unequal length cases} \label{sec:uni}
In this section, $\lambda$ is a partition of $n$ consisting of two rows of unequal lengths. The main goal here is to show that $\awg_\lambda$ is the unique nb-admissible $\affSn$-graph (up to isomorphism) such that $U(\awg_\lambda) \simeq \adeg_\lambda$. In equal length cases, i.e. if $\lambda=(a,a)$ for some $a$, the corresponding nb-admissible $\affSn$-graph is not unique --- it is discussed in the next section.

\subsection{Robinson-Schensted-Knuth and $\shift$}
First we consider the action of $\shift \in \extSn$ on $\RSYT(n)$ by changing each entry $\mo{i}$ to $\mo{i+1}$ (and reordering entries in each row if necessary). Here we describe $\RSK(\shift(T))$ in terms of $\RSK(T)$.
\begin{lem} \label{rskandshift} Suppose that $T\in \RSYT(n)$ and set $\RSK(T) = (P,Q)$. From these we construct $P'$ and $Q'$ as follows.
\begin{enumerate}[label=$\bullet$]
\item Find the position of a corner box of $P$ containing $n$ (which is unique since $P$ is standard).
\item Apply the inverse of the bumping process to $Q$ starting from the corner box of $Q$ at the position found above. Denote the result tableau by $\tilde{Q}$ and the entry which is bumped out from the process by $x$.
\item Column-bump $x$ into $\tilde{Q}$ and let $Q'$ be its result. Or equivalently, insert $x$ to the transpose of $\tilde{Q}$ using the ``dual'' bumping process and let $Q'$ be the transpose of its result. (See \cite[Section 5]{knu70} or \cite[Chapter 7.14]{sta86} for the definition of dual bumping process.)
\item Let $\tilde{P}$ be the unique tableau such that $\sh(\tilde{P})=\sh(Q')$ and $\tilde{P}\pres{[1,n-1]}= P\pres{[1,n-1]}$. In other words, $\tilde{P}$ is obtained from $P$ by moving a box containing $n$ if necessary so that $\sh(\tilde{P}) = \sh(Q')$. (In particular, if $\sh(Q) = \sh(Q')$ then $\tilde{P}=P$.)
\item Do the inverse of the promotion operator on $\tilde{P}$ with respect to $n$, and define $P'$ to be its result. (See \cite[Section 7]{sag11} for the definition of the promotion operator.)
\end{enumerate}
Then we have $\RSK(\shift(T)) = (P',Q')$.
\end{lem}
\ytableausetup{smalltableaux}
\begin{proof} Let $\sA$ be the two-line array corresponding to $T$, and $\tilde{\sA}$ be the one obtained from $\sA$ by switching the first and the second rows and reordering the entries if necessary so that the first row becomes $1, 2, \ldots, n$. For example, if $T=\ytableaushort{2457,369,18}$, then 
$$\sA = \begin{pmatrix} 1 &1 & 2 & 2 & 2 & 3 & 3&3&3 \\ 1&8&3&6&9&2&4&5&7 \end{pmatrix} \textup{ and }\tilde{\sA} = \begin{pmatrix} 1&2&3&4&5&6&7&8&9\\ 1&3&2&3&3&2&3&1&2\end{pmatrix}.$$
It is clear that the image of $\tilde{\sA}$ under $\RSK$ is equal to $(Q, P)$. Also, for $T'=\shift(T)$ we similarly define $\sA'$ and $\tilde{\sA}'$. For example, if $T$ is as above then $T' = \ytableaushort{3568,147,29}$ and
$$\sA' = \begin{pmatrix} 1 &1 & 2 & 2 & 2 & 3 & 3&3&3 \\ 2&9&1&4&7&3&5&6&8 \end{pmatrix} \textup{ and }\tilde{\sA}' = \begin{pmatrix} 1&2&3&4&5&6&7&8&9\\ 2&1&3&2&3&3&2&3&1\end{pmatrix}.$$
Note that $\tilde{\sA}'$ is obtained from $\tilde{\sA}$ by applying cyclic shift on the second row. Under this description, the $Q$ part of the claim is well-known; here $Q$ (resp. $P$) is considered as an insertion tableau (resp. a recording tableau) of $\tilde{\sA}$ under $\RSK$.

On the other hand, $P'$ is the unique standard Young tableau which satisfies that $\sh(P')=\sh(Q')$ and that the reading word of $P'|_{[2,n]}$ is Knuth equivalent to that of $\shift(T)|_{[2,n]}=\shift(T|_{[1,n-1]})$, which follows from the definition of (the inverse of) the promotion operator in terms of jeu-de-taquin procedure. Therefore the $P$ part of the claim also follows.
\end{proof}

From the lemma above, it follows that either $\fsh(T)$ and $\fsh(\shift(T))$ coincide or differ by one box. The next lemma shows how $\fsh(\shift(T))$ differs from $\fsh(T)$ in (possibly equal) two-row cases.

%

\begin{lem} \label{lem:rsktworow} Assume that $\lambda=(\lambda_1,\lambda_2)\vdash n$, $T\in \RSYT(\lambda)$, and $\RSK(T) = (P,Q)$.
\begin{enumerate}[label=\textup{(\arabic*)}]
\item Suppose that $\fsh(T)=(a,b)$ is not the same as $\lambda$ or $(n)$. If $n \in P^1$, then $\fsh(\shift(T))=(a-1,b+1)$. If $n \in P^2$, then $\fsh(\shift(T))=(a+1,b-1)$.
\item Suppose that $\fsh(v)=\lambda$. If $n \in P^1$, then $\fsh(\shift(T)) = \fsh(v)=\lambda$. If $n \in P^2$, then $\fsh(\shift(T)) = (\lambda_1+1, \lambda_2-1)$.
\item Suppose that $\fsh(T)=(n)$. Then (always $n\in P^1$ and) $\fsh(\shift(T)) = (n-1,1)$.
\end{enumerate}
In particular, $\fsh(T)=\fsh(\shift(T))$ only when $T$ and $\shift(T)$ are both standard.
\end{lem}
\begin{proof} Since $\#\SSYT(\mu, \lambda^{op}) \leq 1$ for any $\mu\vdash n$, it follows that $Q$ is uniquely determined by its shape $\fsh(T)$. Now the lemma follows from Lemma \ref{rskandshift} by case-by-case analysis.
\end{proof}
\begin{rmk} Note that $n \in P^1$ if and only if  $n$ is not bumped under the RSK insertion process with input $T$ if and only if $n\in T^1$. Therefore, Lemma \ref{lem:rsktworow} remains valid if one replaces ``$n\in P^1$'' and ``$n\in P^2$'' therein with ``$n\in T^1$'' and ``$n\in T^2$'', respectively.
\end{rmk}
If $\lambda$ consists of two unequal rows, the previous lemma implies the following statement. Also one can easily observe that its proof is not valid for equal length cases.
\begin{lem} \label{lem:standuneq} Suppose that $\lambda=(\lambda_1,\lambda_2)$ where $\lambda_1>\lambda_2$. Then for any $T \in \RSYT(\lambda)$, there exists $k \in [1,n]$ such that $\shift^k(T)$ and $\shift^{k+1}(T)$ are both standard.
\end{lem}
\begin{proof} Suppose that the claim is false. Then Lemma \ref{lem:rsktworow} and its remark shows that if $n \in T^1$ (resp. $n \in T^2$) then $\fsh(\shift(T))_1=\fsh(T)_1-1$ (resp. $\fsh(\shift(T))_1=\fsh(T)_1+1$). As $\shift^n(T)=T$, this means that we have $\fsh(T) = \fsh(\shift^n(T)) = (\lambda_1+(\lambda_2-\lambda_1), \lambda_2+(\lambda_1-\lambda_2))=(\lambda_2, \lambda_1)$, which is impossible.
\end{proof}
As a result, we have the following property that is our main tool for uniqueness statement.
\begin{prop}\label{prop:fshcomp1} Let $\lambda=(\lambda_1,\lambda_2) \vdash$ where $\lambda_1>\lambda_2$ and assume that $u, v\in \RSYT(\lambda)$ where $\ades(u) \supsetneq \ades(v)$. Then there exists $k \in [1,n]$ such that $\ades(\shift^k(u)) -\{k\} \supsetneq \ades(\shift^k(v))-\{k\}$ and $\fsh(\shift^k(u)) \geq \fsh(\shift^k(v))=\lambda$ with respect to dominance order.
\end{prop}
\begin{proof} By Lemma \ref{lem:standuneq}, there exist at least two $k\in [1,n]$ such that $\shift^k(v)$ is standard, in which case we have $\fsh(\shift^k(u)) \geq \fsh(\shift^k(v))=\lambda$. As $\ades(u) \supsetneq \ades(v)$, at least one of such $k$ should satisfy  $\ades(\shift^k(u)) -\{k\} \supsetneq \ades(\shift^k(v))-\{k\}$, thus the result follows.
\end{proof}

\subsection{Uniqueness of $\awg_\lambda$ in unequal length cases}
We are ready to prove the uniqueness statement of $\awg_\lambda$ for $\lambda=(\lambda_1, \lambda_2)$ such that $\lambda_1>\lambda_2$. We start with the following lemma.
\begin{lem}\label{lem:fshcomp2} Suppose that $\Gamma$ is an nb-admissible $\affSn$-graph and $u, v \in \Gamma$ satisfy $m(u\ra v) \neq 0$ and $m(v\ra u)=0$, i.e. there exists a directed edge from $u$ to $v$ in $\Gamma$. If this edge survives in $\Gamma\pres{[1,n-1]}$ after parabolic restriction to $[1,n-1]$, then we have $\fsh(u) \leq \fsh(v)$ in terms of dominance order.
\end{lem}
\begin{proof} This follows directly from Theorem \ref{thm:admord}.
\end{proof}

\begin{thm} \label{thm:unequniq}Let $\Gamma, \Gamma'$ be nb-admissible $\affSn$-graphs such that $U(\Gamma) \simeq U(\Gamma') \simeq \adeg_\lambda$. Then $\Gamma\simeq\Gamma'$ as $\affSn$-graphs. As a result, they are also isomorphic to $\awg_\lambda$.
\end{thm}
\begin{proof} By assumption, we may identify vertices and undirected edges of $\Gamma$ and $\Gamma'$. Now suppose that there is a directed edge $u\rightarrow v$ of weight $p>0$ in $\Gamma$. Then it suffices to show that the same directed edge appears in $\Gamma'$. By Proposition \ref{prop:fshcomp1}, there exists $k \in [1,n]$ such that this edge survives in $\Gamma\pres{\pint{\mo{k+1}, \mo{k-1}}}$ and $\shift^k(v)$ is standard. Using the cyclic symmetry of $\affSn$, we may assume that $k=n$ which means that this edge survives in $\Gamma\pres{[1,n-1]}$ and that $v$ is standard. Now by Lemma \ref{lem:fshcomp2}, it forces that $\fsh(u) = \fsh(v)=\lambda$, i.e. $u$ and $v$ are both standard. However, it means that both $u$ and $v$ are in the same cell of $\Gamma$ isomorphic to $\Gamma_\lambda$, thus by Theorem \ref{thm:admklund} and \ref{thm:admkl} this directed edge should appear in $\Gamma'$ with the same weight $p$ as well.
\end{proof}
\begin{rmk} In the proof we do not assume that $\Gamma$ is $\shift$-invariant. However, as a result of the theorem such graphs should be $\shift$-invariant since so is $\awg_\lambda$.
\end{rmk}

\section{Equal length cases} 
The uniqueness statement of the previous section does not hold in equal length cases, i.e. when $\lambda=(a,a)$ for some $a\in \bZ_{>0}$. In fact, there are more than one (up to isomorphism) whose undirected part is isomorphic to $\adeg_\lambda$. Let us start with finding another such $\affSn$-graph.

\subsection{$\affSn$-graph $\awg_\lambda'$}
From now on we assume that $\lambda=(a,a)$ is a partition of two rows of the same length. Let $\adeg_\lambda^0$ and $\adeg_\lambda^1$ be the full subgraphs of $\adeg_\lambda$ whose sets of vertices are
\begin{gather*}
\{T \in \RSYT(\lambda) \mid \fsh(T) \in \{(a,a), (a+2,a-2), (a+4, a-4), \ldots \}\} \textup{ and }
\\\{T \in \RSYT(\lambda) \mid \fsh(T) \in \{(a+1,a-1), (a+3,a-3), (a+5, a-5), \ldots \}\},
\end{gather*}
respectively. Then we have:
\begin{lem} The graph $\adeg_\lambda$ consists of two connected components $\adeg_\lambda^0$ and $\adeg_\lambda^1$.
\end{lem}
\begin{proof} By \cite[Theorem 8.6]{clp17}, there are two connected components in $\adeg_\lambda$ and each component consists of row-standard Young tableaux of shape $\lambda$ with the same ``charge'' modulo 2, where the charge statistic is defined as in \cite[Definition 8.3]{clp17}. However, it is easily proved using the definition of Robinson-Schensted correspondence that in our case the charge of $T\in \RSYT(\lambda)$ is equal to the length of the second row of $\fsh(T)$.
\end{proof}
\begin{rmk} By the same reason, if $\lambda$ consists of two unequal rows then $\adeg_\lambda$ is connected, which also implies that $\awg_\lambda$ is (strongly) connected.
\end{rmk}

Let us define $\awg_\lambda'$ to be the subgraph of $\awg_\lambda$ obtained by removing all the directed edges connecting $\adeg_\lambda^0$ and $\adeg_\lambda^1$. In other words, $\awg_\lambda'$ is a (disjoint) union of two simple components of $\awg_\lambda$.
\ytableausetup{boxsize=1em, notabloids}
\begin{figure}[htbp]
\centering
\begin{tikzpicture}[scale=1.2]

	\node[draw, tl] (1o1) at (0,3) {\ytableaushort{12{\dt{3}},456}};
	\node[draw] (1o2) at (3^.5/2*3,-1/2*3)  {\ytableaushort{34{\dt{5}},126}};
	\node[draw] (1o3) at (-3^.5/2*3,-1/2*3)  {\ytableaushort{{\dt{1}}56,234}};
	
	\node[draw, tl] (1i1) at (0,1.5) {\ytableaushort{1{\dt{2}}{\dt{4}},356}};
	\node[draw] (1i2) at (3^.5/2*1.5,-1/2*1.5)  {\ytableaushort{3{\dt{4}}{\dt{6}},125}};
	\node[draw] (1i3) at (-3^.5/2*1.5,-1/2*1.5)  {\ytableaushort{{\dt{2}}5{\dt{6}},134}};	

	\node[draw, tl] (1m1) at (3^.5/2*2,1/2*2){\ytableaushort{{\dt{1}}3{\dt{4}},256}};
	\node[draw] (1m2) at (0,-2){\ytableaushort{{\dt{3}}5{\dt{6}},124}};
	\node[draw, tl] (1m3) at (-3^.5/2*2,1/2*2){\ytableaushort{1{\dt{2}}{\dt{5}},346}};	
	
	\node[draw, tl] (1c) at (0,0) {\ytableaushort{{\dt{1}}{\dt{3}}{\dt{5}},246}};	
	
	\node[draw] (2o1) at (0+6,3-3) {\ytableaushort{23{\dt{4}},156}};
	\node[draw] (2o2) at (3^.5/2*3+6,-1/2*3-3)  {\ytableaushort{45{\dt{6}},123}};
	\node[draw] (2o3) at (-3^.5/2*3+6,-1/2*3-3)  {\ytableaushort{1{\dt{2}}6,345}};
	
	\node[draw] (2i1) at (0+6,1.5-3) {\ytableaushort{2{\dt{3}}{\dt{5}},146}};
	\node[draw] (2i2) at (3^.5/2*1.5+6,-1/2*1.5-3)  {\ytableaushort{{\dt{1}}4{\dt{5}},236}};
	\node[draw] (2i3) at (-3^.5/2*1.5+6,-1/2*1.5-3)  {\ytableaushort{{\dt{1}}{\dt{3}}6,245}};	

	\node[draw] (2m1) at (3^.5/2*2+6,1/2*2-3){\ytableaushort{{\dt{2}}4{\dt{5}},136}};
	\node[draw] (2m2) at (0+6,-2-3){\ytableaushort{{\dt{1}}{\dt{4}}6,235}};
	\node[draw] (2m3) at (-3^.5/2*2+6,1/2*2-3){\ytableaushort{2{\dt{3}}{\dt{6}},145}};	
	
	\node[draw] (2c) at (0+6,0-3) {\ytableaushort{{\dt{2}}{\dt{4}}{\dt{6}},135}};	
	
	\draw[tl] (1o1) -- (1i1);  \draw[] (1o2) -- (1i2); \draw[] (1o3) -- (1i3);
	\draw[tl] (1i1) -- (1m1);  \draw[] (1i2) -- (1m2); \draw[] (1i3) -- (1m3);
	\draw[tl] (1i1) -- (1m3);  \draw[] (1i2) -- (1m1); \draw[] (1i3) -- (1m2);
	\draw[tl] (1c) -- (1m1);  \draw[] (1c) -- (1m2); \draw[tl] (1c) -- (1m3);	
	
	\draw[] (2o1) -- (2i1);  \draw[] (2o2) -- (2i2); \draw[] (2o3) -- (2i3);
	\draw[] (2i1) -- (2m1);  \draw[] (2i2) -- (2m2); \draw[] (2i3) -- (2m3);
	\draw[] (2i1) -- (2m3);  \draw[] (2i2) -- (2m1); \draw[] (2i3) -- (2m2);
	\draw[] (2c) -- (2m1);  \draw[] (2c) -- (2m2); \draw[] (2c) -- (2m3);	
	
	\draw[arr,tl] (1c) to [out=50,in=-50] (1o1); \draw[arr] (1c) to [out=50-120,in=-50-120]  (1o2); \draw[arr] (1c) to [out=50-240,in=-50-240]  (1o3); 
	\draw[arr] (2c) to [out=50,in=-50] (2o1); \draw[arr] (2c) to [out=50-120,in=-50-120]  (2o2); \draw[arr] (2c) to [out=50-240,in=-50-240]  (2o3); 
	
\end{tikzpicture}
\caption{$\affS{6}$-graph $\awg_{(3,3)}'$}
\label{fig:33'}
\end{figure}
\begin{example} Figure \ref{fig:33'} illustrates the $\affS{6}$-graph $\awg_{(3,3)}'$.
\end{example}

\ytableausetup{boxsize=normal, tabloids}	
	
We show that $\awg_\lambda'$ is also a $\affSn$-graph. First, the following lemma is a substitute of Lemma \ref{lem:standuneq} in equal length cases.
\begin{lem} \label{lem:standeq} There exists $k \in [1,n]$ such that $\shift^k(T)$ is standard.
\end{lem}
\begin{proof}\ We use induction on $n$. Choose $\mo{i} \in [1,n]$ such that $\mo{i} \in T^1$ and $\mo{i+1}\in T^2$, which always exists. Let us regard the rows of $T$ as words with alphabets in $[1,n]$, and let $w_1, w_2, w_3, w_4$ be words such that $T^1 = w_1\mo{i}w_2$ and $T^2 = w_3 \mo{i+1} w_4$. By induction hypothesis, replacing $T$ with $\shift^k(T)$ for some $k$ if necessary, we may assume that $\tilde{T} = (w_1w_2, w_3w_4)$ is standard. Furthermore, if $\mo{i}=n$ then we apply $\shift$ to $T$ which changes $n$ to $1$ but keeps $\tilde{T}$ to be standard. Thus it suffices to consider the case when $1\leq \mo{i}\leq n-1$. Now since entries in $w_1$ and $w_3$ are smaller than $\mo{i}$ and those in $w_2$ and $w_4$ are larger than $\mo{i+1}$, it follows that $\tilde{T}$ is standard only when the length of $w_1$ is not smaller than that of $w_2$. From this it is easy to see that $T$ is also standard.
\end{proof}

\begin{lem} \label{lem:eqfk}Suppose that $u$ and $v$ are in different simple components of $\awg_\lambda$ and there exists an (necessarily directed) edge $u\rightarrow v$. Then,
\begin{enumerate}
\item the move from $u$ to $v$ is of the first kind.
\item if $\fsh(u)\geq \fsh(v)$ with respect to dominance order then $\fsh(u) = (a+1, a-1)$, $\fsh(v) = (a,a)=\lambda$, and the move $u\rightarrow v$ is $n\intch 1$.
\end{enumerate}
\end{lem}
\begin{proof} We prove (1). Let us regard $u, v$ as elements in $\RSYT(\lambda)$, and note that $\ades(u) \supset \ades(v)$ as $u\rightarrow v$ is a directed edge. By Lemma \ref{lem:standeq}, we may assume that $v$ is standard. As $\fsh(u)$ cannot be equal to $\fsh(v)$ by assumption, we should have $\fsh(u)>\fsh(v)$. Therefore, Theorem \ref{thm:admord} implies that the edge $u\rightarrow v$ must be deleted in the parabolic restriction $\awg_\lambda\pres{[1,n-1]}$. This means that $\ades(u) =\ades(v) \sqcup\{n\}$ and thus $1\in u^2$ and $n \in u^1$. Since $1 \in v^1$ and $n \in v^2$ ($v$ is standard), (1) follows.

Now we prove (2). As $\fsh(u)\neq \fsh(v)$ we should have $\fsh(u)>\fsh(v)$, which means that this directed edge should be deleted in the parabolic restriction $\awg_\lambda\pres{[1,n-1]}$ by Theorem \ref{thm:admord}. Thus $\ades(u) = \ades(v) \sqcup \{n\}$, and $n\intch 1$ is the only possible move of the first kind from $u$ to $v$. Now if $\fsh(u)> (a+1, a-1)$, then direct calculation shows that $\fsh(u)=(\fsh(v)_1+2,\fsh(v)_2-2)$, which contradicts that $u$ and $v$ are in different simple components. Thus we should have $\fsh(u)= (a+1, a-1)$ and $\fsh(v)=(a,a)$ as desired.
\end{proof}

%
%

\begin{thm} $\awg_\lambda'$ is a $\affSn$-graph.
\end{thm}
\begin{proof} We use Theorem \ref{thm:combrules}. It is clear from the definition that $\awg_\lambda'$ satisfies the Compatibility Rule, the Simplicity Rule, and the Bonding Rule. Thus we only need to check that $N_{i,j}(\awg_\lambda';u,v)=N_{j,i}(\awg_\lambda';u,v)$ for $i,j \in [1,n]$ not adjacent to each other in the Dynkin diagram of $\affSn$ and for $u,v \in \awg_\lambda'$. If $u$ and $v$ are in different connected components then clearly $N_{i,j}(\awg_\lambda';u,v)=N_{j,i}(\awg_\lambda';u,v)=0$, thus we only need to consider the case when they are in the same component. As we already proved $N_{i,j}(\awg_\lambda;u,v)=N_{j,i}(\awg_\lambda;u,v)$, it suffices to show that $N_{i,j}(\awg_\lambda';u,v)=N_{i,j}(\awg_\lambda;u,v)$.

If $N_{i,j}(\awg_\lambda';u,v)\neq N_{i,j}(\awg_\lambda;u,v)$ then there exist $w\in \awg_\lambda$ and directed edges $u\rightarrow w$, $w\rightarrow v$ in $\awg_\lambda$ such that $i,j \in \ades(u)$, $\{i,j\} \cap \ades(v) =\emptyset$,  $i\in \ades(w)$, and $w$ is in the different simple component of $\awg_\lambda$ from that of $u$ and $v$. By applying $\shift$ repeatedly if necessary, we may assume that $v$ is standard (Lemma \ref{lem:standeq}). Then by Lemma \ref{lem:eqfk} we have $\fsh(w) = (a+1, a-1)$, $\fsh(v) = \lambda$, the move from $w$ to $v$ is $1\rintch n$, and the move from $u$ to $w$ is of the first kind. In particular, we have $1 \in w^2$ and $n \in w^1$. However, in such a case there is no standard tableau $u$ from which $w$ is obtained by a move of the first kind, which is a contradiction. It follows that $N_{i,j}(\awg_\lambda';v,w)=N_{i,j}(\awg_\lambda;v,w)$ which implies the claim.
\end{proof}

\subsection{Minimality of $\awg_\lambda'$} Here we prove the minimality of $\awg_\lambda'$. More precisely, we have the following theorem.
\begin{thm} \label{thm:mineq} Suppose that $\Gamma$, $\Gamma'$ are two nb-admissible $\affSn$-graphs such that $U(\Gamma) \simeq U(\Gamma') \simeq \adeg_\lambda$. If $\Gamma$ is disconnected, then there exists an embedding from $\Gamma$ to $\Gamma'$.
\end{thm}
\begin{proof} As $U(\Gamma)\simeq U(\Gamma')$, it suffices to show that if there exists an edge $u\rightarrow v$ of weight $p>0$ in $\Gamma$ then the same edge exists in $\Gamma'$. To this end we choose $k, l \in [1,n]$ such that $\shift^k(u)$ and $\shift^{l}(v)$ are standard, which exist by Lemma \ref{lem:standeq}. Note that $u$ and $v$ are in the same component of $\adeg_\lambda\simeq U(\Gamma)$ as $\Gamma$ is disconnected.

First suppose that $\shift^l(u)$ is also standard. In our situation, Lemma \ref{lem:rsktworow} implies that $\fsh(T)_2$ and $\fsh(\shift(T))_2$ always differ by 1 for any $T\in \RSYT(\lambda)$. (Note that if $T$ is standard then $n \in T^2$ as we consider equal length cases.) Therefore, $\fsh(\shift^l(u))=\fsh(\shift^{l}(v))=\lambda=(a,a)$ and $\fsh(\shift^{l\pm 1}(u))=\fsh(\shift^{l\pm 1}(v))=(a+1, a-1).$ If we identify $\Sym_n$ with the finite maximal parabolic subgroup of $\affSn$ generated by $I-\{s_{t}\}$ for each $t\in \{\mo{l-1}, \mo{l}, \mo{l+1}\}$, then $u$ and $w$ are in the same simple component of $\Gamma\pres{\pint{\mo{t+1}, \mo{t-1}}}$ and there exists at least one $t$ such that the edge $u\rightarrow w$ of weight $p$ survives in the parabolic restriction, i.e. $\ades(u)-\{\mo{t}\} \supsetneq \ades(v) -\{\mo{t}\}$. Now by Theorem \ref{thm:admkl}, this edge should also appear in $\Gamma'$ with weight $p$ as desired.

Now assume that $\shift^l(u)$ is not standard. Since $u$ and $v$ are in the same connected component of $\adeg_\lambda$, we have $\fsh(u)_2\equiv \fsh(v)_2 \pmod 2$. Therefore, there exists $t \in [1,n]$ different from $l$ such that $\fsh(\shift^t(u))=\fsh(\shift^t(v))$. On the other hand, by Theorem \ref{thm:admord} the edge $u\rightarrow v$ vanishes on the parabolic restriction $\Gamma\pres{\pint{\mo{l+1}, \mo{l-1}}}$, which means that $\ades(u) = \ades(v) \sqcup \{l\}$. Thus $\ades(u)-\{t\} \supsetneq \ades(v) - \{t\}$ and this edge survives in $\Gamma\pres{\pint{\mo{t+1}, \mo{t-1}}}$. Again by Theorem \ref{thm:admkl}, this edge should also appear in $\Gamma'$ with weight $p$ as needed.
\end{proof}

\begin{rmk} Note that we do not assume that $\Gamma$ is stable under the action of $\shift$ in the proof of the above theorem. Instead, we choose a maximal parabolic subgroup of $\affSn$ which may be different from the conventional choice and apply Theorem \ref{thm:admord} with respect to this parabolic subgroup.
\end{rmk}

\begin{cor} \label{cor:eqmin} If $\Gamma$ is a $\affSn$-graph such that $U(\Gamma) \simeq \adeg_\lambda$, then there exists an embedding from $\awg_\lambda'$ to $\Gamma$. In other words, $\awg_\lambda'$ is the minimal $\affSn$-graph such that $\awg_\lambda' \simeq \adeg_\lambda$.
\end{cor}
\begin{proof} It is clear from the theorem above.
\end{proof}

\begin{rmk} There are more than two $\affSn$-graphs, $\awg_\lambda'$ and $\awg_\lambda$, whose simple underlying graph is isomorphic to $\adeg_\lambda$. For example, if we remove the directed edges from $\adeg_\lambda^0$ to $\adeg_\lambda^1$ but keeps the ones from $\adeg_\lambda^1$ to $\adeg_\lambda^0$ in $\awg_\lambda$, then it is easy to show that this is also a $\affSn$-graph which is ``between $\awg_\lambda'$ and $\awg_\lambda$''. This graph is not $\shift$-invariant as $\shift$ swaps two simple components.
\end{rmk}

%
%
%
%
%
%
%
%
\subsection{Maximality of $\awg_\lambda$} Here we prove the maximality of $\awg_\lambda$. To this end, first we recall the notion of arc transport in \cite{chm15}.
\begin{lem}[{\cite[2.3, Lemma 1]{chm15}}] \label{lem:arctrans} Let $W$ be a Coxeter group whose Dynkin diagram is simply-laced, $\Gamma=(V, m, \tau)$ is an nb-admissible $W$-graph, and $x, y, x', y' \in \Gamma$. Suppose that $i, j, k$ are simple reflections of $W$ such that $k \in (\tau(x)\cap\tau(x'))-(\tau(y)\cup\tau(y'))$, $i \in (\tau(x)\cap\tau(y)) - (\tau(x') \cup \tau(y'))$, and $j \in (\tau(x')\cap\tau(y')) - (\tau(x) \cup \tau(y))$. (Thus in particular $i$ and $j$ are adjacent in the Dynkin diagram of $W$ by the Compatibility Rule.) If $m(x,x')=m(x',x)=m(y,y')=m(y',y)=1$, then $m(x, y) = m(x',y')$. Pictorially, we have:
$$\begin{tikzcd}
\framebox{k,i}_x \ar[rr,dash]\ar[d,dashrightarrow,"{m(x,y)}"]&&\framebox{k,j}_{x'}\ar[d, dashrightarrow,"{m(x',y')}"]\\
\framebox{i}_{y}\ar[rr,dash]&&\framebox{j}_{y'}
\end{tikzcd}$$
\end{lem}
\begin{proof} Again, the proof of \cite[2.3, Lemma 1]{chm15} does not use the bipartition property, thus it applies to our setting. Also, {\cite[2.3, Lemma 1]{chm15}} only assumes that $\Gamma$ is a ``$W$-molecular'' graph which is weaker than being a $W$-graph.
\end{proof}
\ytableausetup{notabloids}

\begin{lem} \label{lem:arctworow} Let $\Gamma=(V,m,\tau)$ be an nb-admissible $\affSn$-graph such that $U(\Gamma) \simeq \adeg_\lambda$. (Thus in particular we may set $V=\RSYT(\lambda)$ and $\tau=\ades$.) Suppose that $u$ and $v$ are in different simple components of $\Gamma$, there exists a directed edge $u\rightarrow v$ of weight $p>0$ in $\Gamma$, and $v$ is standard. Then it is a move of the first kind (of weight $p$) and $p$ is equal to the weight of the edge from $\ytableaushort{24\cdots {\scriptstyle n-4}{\scriptstyle n-2}n,135\cdots{\scriptstyle n-3}{\scriptstyle n-1}}$ to $\ytableaushort{124\cdots {\scriptstyle n-4}{\scriptstyle n-2},35\cdots{\scriptstyle n-3}{\scriptstyle n-1}n}$.
\end{lem} 
\ytableausetup{tabloids}
\begin{proof} First note that $\fsh(u) \neq \fsh(v)$ by assumption, thus by Theorem \ref{thm:admord} we have $\ades(u) = \ades(v) \sqcup\{n\}$. On the other hand, if $\fsh(u) \geq (a+3, a-3)$, then $\fsh(\shift(u))>\fsh(\shift(v))$ which means that the edge $u\rightarrow v$ is removed in the parabolic restriction $\Gamma\pres{[2,n]}$ again by Theorem \ref{thm:admord}. However, this contradicts the fact that $\ades(u) = \ades(v) \sqcup\{n\}$, thus we should have $\fsh(u) = (a+1,a-1)$. ($\fsh(u)\neq (a+2, a-2)$ since $u$ and $v$ are in different simple components.)

Furthermore, $n\in \ades(u)$ if and only if $1 \in u^2$ and $n\in u^1$, thus $n-1, 1 \notin \ades(v)\subset \ades(u)$. As $1 \in v^1$ ($v$ is standard), this means that $2\in v^1$ as well. Also, if $2\in u^2$ then direct calculation shows that $\fsh(u) \geq (a+2, a-2)$, thus we should have $2\in u^1$. Now let $x \in [2,n-1]$ be the smallest entry of $\ades(v)$. Then $[2,x] \subset u^1 \cap v^1$ and $x+1 \in u^2\cap v^2$, i.e. we have
$$u=\ytableaushort{2\cdots x\cdots,1{\scriptstyle x+1}\cdots\cdots} \quad \rightsquigarrow \quad v=\ytableaushort{12\cdots x\cdots,{\scriptstyle x+1}\cdots\cdots\cdots\cdots}$$

Suppose that $x>2$. Then we set $u'$ (resp, $v'$) to be the tableau obtained from $u$ (resp. $v$) by swapping $x$ and $x+1$. Then these are allowed moves in \ref{sec:defawg} of the first kind and also $\ades(u)$ and $\ades(u')$ (resp. $\ades(v)$ and $\ades(v')$) are incomparable, thus there exist undirected edges $u - u'$ and $v -v'$. Now we use Lemma \ref{lem:arctrans} with $(i,j,k) = (x-1, x, n)$ and thus we have $m(u'\ra v') = m(u\ra v) = p$. Furthermore, it is clear that $\fsh(u) = \fsh(u')$, $\fsh(v)=\fsh(v')$, and $u\rightarrow v$ is a move of the first kind if and only if $u' \rightarrow v'$ is a move of the first kind. Thus by iterating this process, we only need to consider the case when $2\in \ades(v)$, i.e. we have
$$u=\ytableaushort{2\cdots\cdots,13\cdots} \quad \rightsquigarrow \quad v=\ytableaushort{12\cdots,3\cdots\cdots},$$

By direct calculation, $\fsh(u) =(a+1,a-1)$ implies that $4 \in u^1$. If $n=4$, then $4 \in v^2$ and we are done. Otherwise, if $4\in v^2$ then let us set $u'$ (resp. $v'$) to be the tableau obtained from $u$ (resp. $v$) by swapping 3 and 4 (resp. 2 and 3). These are allowed moves in \ref{sec:defawg} and $\ades(u)$ and $\ades(u')$ (resp. $\ades(v)$ and $\ades(v')$) are incomparable, thus there exist undirected edge $u-u'$ and $v-v'$. Pictorially, we have
$$\begin{tikzcd}
u=\ytableaushort{24\cdots,13\cdots}\ar[r,rightsquigarrow]\ar[d,dash]& v=\ytableaushort{12\cdots,34\cdots}\ar[d, dash]\\
u'=\ytableaushort{23\cdots,14\cdots}\ar[r,rightsquigarrow]&v'=\ytableaushort{13\cdots,24\cdots}
\end{tikzcd}$$
Thus by Lemma \ref{lem:arctrans} with $(i,j,k) = (2,3,n)$, we should have $m(u'\ra v') = m(u\ra v) =p>0$. However, this is impossible as $1\in \ades(v') -\ades(u')$. It follows that $4 \in v^1$, i.e. we have
$$u=\ytableaushort{24\cdots,13\cdots} \quad \rightsquigarrow \quad v=\ytableaushort{124\cdots,3\cdots\cdots\cdots},$$

Now we choose $x\in [4,n-1]$ to be the smallest entry of $\ades(v)$. By the same argument as above, it suffices to consider the case when $x=4$. Then $5 \in u^2\cap v^2$ and $6 \in u^1$ as $\fsh(u) = (a+1, a-1)$. Now if $n=6$ then $6 \in v^2$ and we are done. Otherwise, we iterate the argument above, and eventually we only need to consider the case when
$$u=\ytableaushort{24\cdots {\scriptstyle n-4}{\scriptstyle n-2}n,135\cdots{\scriptstyle n-3}{\scriptstyle n-1}} \quad \rightsquigarrow \quad v=\ytableaushort{124\cdots {\scriptstyle n-4}{\scriptstyle n-2},35\cdots{\scriptstyle n-3}{\scriptstyle n-1}n}.$$
Now the statement follows from the fact that $u\rightarrow v$ is a move of the first kind $n\intch 1$.
\end{proof}

From the lemma above we deduce the maximality of $\awg_\lambda$.
\begin{thm} \label{thm:eqmax} If $\Gamma$ is an nb-admissible $\affSn$-graph such that $U(\Gamma) \simeq \adeg_\lambda$ and there exists an embedding from $\awg_\lambda$ to $\Gamma$, then this embedding is an isomorphism.
\end{thm}
\begin{proof} It suffices to show that if there exists a directed edge $u\rightarrow v$ of weight $p>0$ in $\Gamma$ and the same edge appears in $\awg_\lambda$. If $u$ and $v$ are in the same simple component, then it follows from the proof of Theorem \ref{thm:mineq}. Otherwise if $u$ and $v$ are in different simple components, then by Lemma \ref{lem:arctworow} this is a move of the first kind and also $p$ is equal to the weight of the directed edge
$$\shift^k\left(\ytableaushort{24\cdots {\scriptstyle n-4}{\scriptstyle n-2}n,135\cdots{\scriptstyle n-3}{\scriptstyle n-1}}\right) \rightarrow \shift^k\left(\ytableaushort{124\cdots {\scriptstyle n-4}{\scriptstyle n-2},35\cdots{\scriptstyle n-3}{\scriptstyle n-1}n}\right),$$ 
for some $k\in [1,n]$, which is always 1 by assumption. Thus the edge $u\rightarrow v$ is already contained in the image of $\awg_\lambda$ with the same weight $p=1$, which implies the statement. 
\end{proof}
\begin{rmk} Suppose that $\Gamma$ is an nb-admissible $\affSn$-graph such that $U(\Gamma)\simeq \adeg_\lambda$ and it is invariant under $\shift$. Then from the results above it is easy to show that there exists $p\geq 0$ such that every directed edge between two simple components is a move of the first kind of weight $p$. If $p=0$ (resp. $p=1$) then $\Gamma \simeq \awg_\lambda'$ (resp. $\Gamma \simeq \awg_\lambda$). In general, one can prove that $\Gamma$ is an nb-admissible $\affSn$-graph for any $p \in \bN$.
\end{rmk}

\section{Periodic $W$-graphs}\label{sec:periodicWgraph}
Here, we discuss how $\awg_\lambda$ is related to a periodic $W$-graph originally defined by Lusztig.
To this end, first we recollect the notion of a periodic $W$-graph focusing on affine type $A$. For reference see \cite{lus97} and \cite{var04}.

\subsection{Periodic $W$-graph}
We recall the root system of type $A_{n-1}$. Let $E$ be an $(n-1)$-dimensional real vector space equipped with an inner product $(\ , \ ): E\times E \rightarrow \bR$. Let $\Pi\colonequals\{\alpha_1, \ldots, \alpha_{n-1}\} \subset E$ be a fixed set of simple positive roots such that $(\alpha_i, \alpha_i) = 2$ for $1\leq i \leq n-1$, $(\alpha_i,\alpha_{i+1}) = -1$ for $1\leq i \leq n-2$, and $(\alpha_i, \alpha_j) = 0$ if $|i-j|>1$. Then the set of roots $R\subset E$ and positive roots $R^+\subset R$ are well-defined. Let $P$ be a root lattice, i.e. a free abelian group generated by $\Pi$ as a subgroup of $E$. Usually we realize this root system by letting $E=\{(x_1, \ldots, x_n)\subset \bR^n \mid \sum x_i=0\}$, $\Pi = \{e_1-e_2, \ldots, e_{n-1}-e_n\}$, etc.

We set $F_{{\alpha},k}\colonequals \{v \in E \mid ({\alpha},v)=k\}$ and $\fF\colonequals\{F_{{\alpha},k} \mid \alpha \in R, k \in \bZ\}.$
(As we only deal with type $A$ root system, we do not differentiate a root and its corresponding coroot.) Let $\fA$ be the set of all the connected components of $E-\cup_{F \in \fF} F$, each of which is called an alcove. Let $A_{id} \in \fA$ be the unique alcove which is in the dominant chamber and whose closure contains $0 \in E$.

For a partition $\lambda \vdash n$, we let $\Pi_\lambda \colonequals \{\alpha_i \in \Pi \mid i \neq n-\sum_{j=1}^k\lambda_j \textup{ for all } 1\leq k \leq l(\lambda)-1\}$. (The reason for adopting this definition rather than ``the opposite one'' will become clear as we proceed our argument.) Also let $R_\lambda$ (resp. $R^+_\lambda$) be the intersection of $R$ (resp. $R^+$) with the $\bZ$-span of $\Pi_\lambda$. Define $\fF_\lambda\colonequals\{F_{{\alpha},k} \in \fF \mid \alpha \in R_\lambda\}$. Then there exists a unique connected component of $E-\cup_{F\in \fF_\lambda} F$ which contains $A_{id}$; $v\in E$ is in this component if and only if $0< (\alpha,v)<1$ for all $\alpha \in R_\lambda^+$. Let $\fA_\lambda\subset \fA$ be the set of alcoves contained in this connected component. This will become a set of vertices of a periodic $\affSn$-graph we construct.

For $F\in \fF$, let $\mathbf{r}_F: E\rightarrow E$ be the reflection along $F$. We identify $\affSn$ with the group generated by $\mathbf{r}_F$ for $F \in \fF$. Under this correspondence, each $s_i$ for $1\leq i \leq n-1$ is assigned to $\mathbf{r}_{F_{\alpha_i,0}}$, and $s_0$ is assigned to $\mathbf{r}_{F_{\tilde{\alpha},1}}$ where $\tilde{\alpha}\colonequals \alpha_1+\alpha_2+ \cdots +\alpha_{n-1}\in R$ is the highest root. We regard $\affSn$ as acting on the right of $E, \fA,\fF,$ etc. For $v \in P$, we define $\mathbf{t}_v: E\rightarrow E$ to be the translation by $v$, which is naturally an element of $\affSn$.

Note that $\affSn$ acts simply on $\fF$ and $\{F_{\alpha_1, 0}, F_{\alpha_2, 0}, \ldots, F_{\alpha_{n-1}, 0}, F_{\tilde{\alpha}, 1}\}$ is the set of representatives of orbits. We say that $F\in \fF$ is of type $s_i$ for $1\leq i \leq n-1$ (resp. of type $s_0$) if $F$ is in the orbit of $F_{\alpha_i,0}$ (resp. $F_{\tilde{\alpha},1}$). For each $A\in \fA$ and each simple reflection $s$, there exists a unique $F\in \fF$ of type $s$ which is adjacent to $A$.

Let $\affS{\lambda^{op}}$ be the subgroup of $\affSn$ generated by reflections along $F \in \fF_\lambda$, which is isomorphic to and often identified with $\affS{\lambda_{l(\lambda)}}\times \cdots \times \affS{\lambda_2}\times \affS{\lambda_1}$.  Then $\affS{\lambda^{op}}$ acts simply on $\fA$ and each orbit meets $\fA_\lambda$ exactly once, thus $\fA_\lambda$ is the set of representatives of $\fA/\affS{\lambda^{op}}$. Let $\cT\colonequals\{\mathbf{t}_v \in \affSn \mid v\in P\}$ and define $\cT_\lambda$ to be the subgroup of $\cT$ generated by the translations by $\alpha_i \in \Pi_\lambda$. Note that $\cT_\lambda = \cT \cap \affS{\lambda^{op}}$ where the intersection is taken inside $\affSn$.

There is another (left) action of $\affSn$ on $\fA$ described as follows. Recall that for any $A\in \fA$ and a simple reflection $s\in \affSn$, there exists a unique hyperplane $F$ adjacent to $A$ which is of type $s$. We define $s \cdot A \colonequals A\cdot \mathbf{r}_F$ to be the image of $A$ under the reflection along $F$. It generates a well-defined left $\affSn$-action on $\fA$ which commutes with the right $\affSn$-action described above; indeed, it is not hard to show that $w\cdot A_{id}=A_{id}\cdot w$ for any $w\in \affSn$. Furthermore, if we set $A_w \colonequals w\cdot A_{id}=A_{id} \cdot w$, then the map $\affSn \rightarrow \fA: w \mapsto A_w$ is a bijection. (This is not the same convention as in \cite[1.1]{lus97} but in \cite[13.12]{lus97}.)

For each $F\in \fF$, there are two connected components of $E-F$. We denote one of such by $E_F^+$ (resp. $E_F^-$) where there exists $\mathbf{t}\in \cT$ such that $E_F^+\cdot \mathbf{t}$ contains the dominant Weyl chamber (resp. there does not exist such $\mathbf{t}\in \cT$). We also call $E_F^+$ (resp. $E_F^-$) the positive (resp. negative) upper half-space with respect to $F$. Now for $A, B\in \fA$, we define $d(A,B)$ by
$$d(A,B) = \left(\sum_{F\in \fF, A\in E_F^-, B\in E_F^+}1\right)-\left(\sum_{F\in \fF, A\in E_F^+, B\in E_F^-}1\right).$$
Note that each sum in the formula is finite and thus it is well-defined. Furthermore, it satisfies that $d(A,B)+d(B,C)+d(C,A) = 0$ for any $A, B, C\in \fA$. Now we define an order $\leq$ on $\fA$ as follows. For $A, B\in \fA$, we write $A\leq B$ if there exists $A_0, A_1, \ldots, A_k \in \fA$ such that $A_0=A$, $A_k=B$, $d(A_i,A_{i+1})=1$, and $A_{i+1}$ is the image of $A_i$ under reflection along some hyperplane in $\fF$ for $0\leq i \leq k-1$. Clearly $A<B$ implies $d(A,B)>0$, but not vice versa.

In \cite[Section 11]{lus97}, for any alcove $A \in \fA_\lambda$ a corresponding ``canonical basis'' $A^\flat$ is introduced which is an element of $\bZ[q^{\pm1}][\fA_\lambda]$ where $q$ is an indeterminate. (The element $A^\flat$ is originally defined to be contained in a certain completion of $\bZ[q^{\pm1}][\fA_\lambda]$. For type $A$, it was proved later by \cite{var04} that this is indeed an element of $\bZ[q^{\pm1}][\fA_\lambda]$.) It can be written as
$$B^\flat = \sum_{A \in \fA_\lambda, A\leq B} p_{A, B}A,$$
where  $p_{A, B}$ is a polynomial in $q^{-1}$. Furthermore, it is known that $p_{A,A}=1$ and $p_{A,B} \in q^{-1}\bZ[q^{-1}]$ if $A\neq B$.

For $A\in \fA_{\lambda}$, we let $\mathfrak{I}(A)$ be the set of simple reflections $s$ such that $sA \in \fA_\lambda$ and $sA>A$. Now for $A, B\in \fA_\lambda$ such that $\mathfrak{I}(A)\not\subset\mathfrak{I}(B)$, we define $\mu(A,B)=\mu(A\ra B)$ to be 
$$\mu(A,B)=\left[
\begin{aligned}
&\textup{the coefficient of } q^{-1} \textup{ in } p_{A,B} &&\textup{ if } A\leq B,
\\&1 &&\textup{ if } B< A=sB \textup{ for some simple reflection } s, 
\\&0 &&\textup{ otherwise.}
\end{aligned}\right.
$$
If $\mathfrak{I}(A)\subset\mathfrak{I}(B)$, we set $\mu(A,B)=\mu(A\ra B)=0$.
Let $\pwg_{\lambda}\colonequals (\fA_{\lambda}, \mu,\mathfrak{I})$ be the corresponding $[1,n]$-graph, where we identify the set of simple reflections of $\affSn$ with $[1,n]$.
Then it is proved that $\pwg_\lambda$ is a $\affSn$-graph, conventionally called a periodic $W$-graph.  
\begin{rmk} There are two twists in this definition compared to the original one \cite[11.13]{lus97}. First, this definition is taken from \cite[12.3]{lus97}, which is a $W$-graph complementary (in the sense of \cite[A.6]{lus97}) to \cite[11.13]{lus97}. In particular, the $\tau$-function $\mathfrak{I}$ here is not the same as $\mathfrak{I}$ but $\tilde{\mathfrak{I}}$ therein.  On the other hand, our definition of $\mu(A,B)$ is the same as that of \cite[11.13]{lus97} instead of \cite[12.3]{lus97}. This is because the definition of a $W$-graph in \cite[A.2]{lus97} is the transpose of our convention. (cf. \cite[Remark 1.1(a)]{ste08})
\end{rmk}

\subsection{Action of $\cT$ on $\pwg_\lambda$} \label{sec:gammaact}
We recall the result in \cite[2.12]{lus97}. The action of $\cT$ permutes $\affS{\lambda^{op}}$-orbits in $\fA$. Thus there is a well-defined action of $\cT$ on $\fA/\affS{\lambda^{op}}$, and under the identification $\fA/\affS{\lambda^{op}}\simeq \fA_\lambda$ we regard it as an action on $\fA_\lambda$. For $\mathbf{t}\in \cT$, we write $\gamma(\mathbf{t}):\fA_\lambda \rightarrow \fA_\lambda$ to denote such an action. (Note that this is in general different from the (right or left) action of $\mathbf{t}$ on $\fA$.) Then  the kernel of this action is $\cT_\lambda$. Furthermore, if we let $\fund_\lambda$ be the set of alcoves in $\fA_\lambda$ adjacent to $0\in E$, then $\fund_\lambda$ is the set of representatives of such $\cT/\cT_\lambda$-orbits. (This follows from \cite[2.12(f)]{lus97}.)

For $\alpha_i \in \Pi-\Pi_\lambda$, we describe $\gamma(\mathbf{t}_{\alpha_i}): \fA_\lambda \rightarrow \fA_\lambda$ explicitly as follows. According to \cite[2.12]{lus97}, there exists a unique $w\in \affS{\lambda^{op}}$ (which depends on $\alpha_i$) such that $A\cdot(\mathbf{t}_{\alpha_i}w) \in \fA_\lambda$ for any $A \in \fA_\lambda$, in which case we have $\gamma(\mathbf{t}_{\alpha_i})(A) = A\cdot (\mathbf{t}_{\alpha_i}w)$ by definition of $\gamma$. Thus it suffices to find $w\in \affS{\lambda^{op}}$ such that $A_{id}\cdot \mathbf{t}_{\alpha_i}w\in \fA_\lambda$. To this end, let $\rho \in E$ be the sum of fundamental weights, i.e. $\rho=\sum_{i=1}^{n-1} \frac{i(n-i)}{2}\alpha_i$. Then $\frac{\rho}{n} \in A_{id}$, thus it suffices to find $w\in \affS{\lambda^{op}}$ such that $(\alpha_i+\frac{\rho}{n})\cdot w \in \bigcup_{A\in \fA_\lambda} A$, i.e. $0<(\alpha, (\alpha_i+\frac{\rho}{n})\cdot w)<1$ for all $\alpha\in R_\lambda^+$. Let $ j, k \in[0,n]$ be such that $\alpha_j,\alpha_k \not \in \Pi_\lambda$, $j<i<k$, and $\alpha_l \in \Pi$ if $j<l<k$ and $l\neq i$. (Here we adopt the convention that $\alpha_0, \alpha_n \notin\Pi_\lambda$.) In other words, if $i=n-\sum_{x=1}^a \lambda_x$ then $j=n-\sum_{x=1}^{a+1} \lambda_x$ and $k=n-\sum_{x=1}^{a-1} \lambda_x$. We claim that $w=(s_{i-1}\cdots s_{j+1})(s_{i+1}\cdots s_{k-1})=(s_{i+1}\cdots s_{k-1})(s_{i-1}\cdots s_{j+1})$. Indeed, if $\alpha_l \in \Pi_\lambda$ then direct calculation shows that
$$
\left((\alpha_i+\frac{\rho}{n})\cdot w, \alpha_l\right)=\left\{
\begin{aligned}
&\frac{n-i+j}{n} &&\textup{ if } l=j+1,
\\&\frac{n-k+i}{n} &&\textup{ if } l=k-1,
\\&\frac{1}{n} &&\textup{ otherwise.}
\end{aligned}\right.
$$
From this it easily follows that $0<((\alpha_i+\frac{\rho}{n})\cdot w,\alpha)<1$ for all $\alpha\in R_\lambda^+$.

\subsection{A bijection between $\fund_\lambda$ and $\RSYT(\lambda)$}
Let $\Sym_n^\lambda \subset \Sym_n$ be the set of minimal coset representatives of $\Sym_n/\Sym_{\lambda^{op}}$, where $\Sym_{\lambda^{op}} = \Sym_{\lambda_{l(\lambda)}}\cdots \times \Sym_{\lambda_2} \times \Sym_{\lambda_1}$ naturally considered as a parabolic subgroup of $\Sym_n$. Then it is easy to show that $\fund_\lambda = \{A_w \mid w\in \Sym_n^\lambda\}$. Using this, we define a bijection $\Upsilon: \fund_\lambda \rightarrow \RSYT(\lambda)$ to be $\Upsilon(A_w) = w\cdot \Tc$ where $\Tc=\Tc_\lambda$ is the unique row-standard Young tableau of shape $\lambda$ whose reading word is $[1,2,\ldots,n]$ and $\Sym_n$ acts on $\RSYT(\lambda)$ by simply permuting entries (and reordering entries in each row if necessary). Since the stabilizer of $\Tc$ in $\Sym_n$ is $\Sym_{\lambda^{op}}$, this is indeed a bijection. Now we prove the following.
\begin{lem} \label{lem:tauinv} $i \in \ades(\Upsilon(A_w))$ if and only if $s_i \in \mathfrak{I}(A_w)$, i.e. $\Upsilon$ ``preserves the $\tau$-invariant''.
\end{lem}
\begin{proof}
Let us first show that $\mathfrak{I}(A_w)\cap\{s_1, \ldots, s_{n-1}\}=\des(\Upsilon(A_w))$. If $s=s_i$ for some $1\leq i \leq n-1$, then $s \in \mathfrak{I}(A_w)$ if and only if $sA_w=A_{sw} \in \fA_\lambda$ and $A_w<A_{sw}$. However, as $A_{sw}$ and $A_w$ are both in $D_{(1^n)}=\{A_w \mid w\in \Sym_n=\Sym_n^{(1^n)}\}$, $A_w<A_{sw}$ if and only if $w>sw$ with respect to the usual Bruhat order on $\Sym_n$. Also if $w>sw$ then $w\in \Sym_n^\lambda$ implies $sw \in \Sym_n^\lambda$. Therefore, we see that $s\in \mathfrak{I}(A_w)$ if and only if $w>sw$ if and only if  $s$ is in the left descent set $L(w)$ of $w$. On the other hand, the reading word of $w\cdot \Tc$ is equal to $[w(1), w(2), \ldots, w(n)]$ by definition (no reordering is necessary as $w\in \Sym_n^\lambda$), which means that $i\in \des(w\cdot \Tc)$ if and only if $s_i\in L(w)$ for $1\leq i \leq n-1$.

It remains to show that $s_n \in \mathfrak{I}(A_w)$ if and only if $n \in \ades(w\cdot \Tc)$. Let $\rho \in E$ be the sum of fundamental weights. Then $\frac{\rho}{n} \in A_{id}$, thus $\frac{\rho}{n}\cdot w \in A_w$ and $\frac{\rho}{n}\cdot s_0w \in A_{s_0w}$. Therefore, $s_0 \in \mathfrak{I}(A_w)$ if and only if:
\begin{enumerate}[label=$\bullet$]
\item $s_0\cdot A_w \in \fA_\lambda$, i.e. $0<(\frac{\rho}{n}\cdot s_0w, \alpha)<1$ for $\alpha \in R^+_\lambda$
\item $A_{s_0w}>A_w$, i.e. $\frac{\rho}{n}\cdot s_0w-\frac{\rho}{n}\cdot w=\frac{2}{n} \tilde{\alpha}\cdot w \in \bQ_{>0}\cdot \alpha$ for some $\alpha \in R^+$ where $\tilde{\alpha} \in R^+$ is the highest root
\end{enumerate}
By direct calculation, we see that the first condition is satisfied if and only if there is no $t$ such that $\sum_{k=t+1}^l\lambda_k< w^{-1}(1),w^{-1}(n)\leq \sum_{k=t}^l\lambda_j$, which is equivalent to that 1 and $n$ are not in the same row of $w\cdot \Tc$. Moreover, the second condition is satisfied if and only if $w^{-1}(1)<w^{-1}(n)$. Thus $w$ satisfies both conditions if and only if $1$ is in the lower row than $n$ in $w\cdot \Tc$, which is also equivalent to $n \in \ades(w\cdot \Tc)$.
\end{proof}

Let us extend $\Upsilon$ to $\Upsilon: \fA_\lambda \rightarrow \RSYT(\lambda)$ in a way that for any $\mathbf{t}\in \cT$ and $w \in \Sym_n^\lambda$ we have $\Upsilon(\gamma(\mathbf{t})(A_w)) \colonequals \Upsilon(A_w)$. This is well-defined since $\fund_\lambda$ is the set of representatives of the $\gamma$-action of $\cT$ on $\fA_\lambda$. On the other hand, we may also extend the action of $\Sym_n$ on $\RSYT(\lambda)$ to $\affSn$ where $s_0$ acts on $\RSYT(\lambda)$ by switching 1 and $n$ and reordering entries of each row if necessary. (This action is well-defined.) Then we have the following.

\begin{lem} \label{lem:upsilontab}For any $w\in \affSn$ such that $A_w\in \fA_\lambda$, we have $\Upsilon(A_w) = w\cdot\Tc$.
\end{lem}
\begin{proof} It is apparent when $w\in \Sym_n^\lambda$ (or $A_w \in \fund_\lambda$) by definition of $\Upsilon$. First we consider the situation when $A_w = \gamma(\mathbf{t}_{\alpha_i})(A_{w'})$ for some $i \in [1,n-1]$ and $w' \in \Sym_n^\lambda$ and prove $\Upsilon(A_w) = w\cdot \Tc$. Since $\gamma(\mathbf{t}_{\alpha_i})$ is trivial when $\alpha_i \in \cT_\lambda$, it suffices to assume otherwise. (The argument below also works, mutatis mutandis, for $A_w = \gamma(\mathbf{t}_{-\alpha_i})(A_{w'})$ case.)

By direct calculation, we have $\mathbf{t}_{\alpha_i} = s_i \cdot (s_{i-1}\cdots s_1) \cdot (s_{i+1}\cdots s_{n-1})\cdot s_0 \cdot (s_1\cdots s_{i-1}) \cdot (s_{n-1}\cdots s_{i+1})$ as an element in $\affSn$. Therefore, from the result in \ref{sec:gammaact} we deduce that $\gamma(\mathbf{t}_{\alpha_i})(A_{w'})=A_{w'} \cdot s_i \cdot (s_{i-1}\cdots s_1) \cdot (s_{i+1}\cdots s_{n-1})\cdot s_0 \cdot (s_1\cdots s_{j}) \cdot (s_{n-1}\cdots s_{k})$, where $j,k\in[0,n]$ are chosen such that if $i=n-\sum_{x=1}^a \lambda_x$ for some $a$ then $j=n-\sum_{x=1}^{a+1} \lambda_x$ and $k=n-\sum_{x=1}^{a-1} \lambda_x$. Thus for the claim it suffices to show that $w'\cdot (s_i \cdot (s_{i-1}\cdots s_1) \cdot (s_{i+1}\cdots s_{n-1})\cdot s_0 \cdot (s_1\cdots s_{j}) \cdot (s_{n-1}\cdots s_{k}))\cdot \Tc = w'\cdot \Tc$, i.e. $s_i \cdot (s_{i-1}\cdots s_1) \cdot (s_{i+1}\cdots s_{n-1})\cdot s_0 \cdot (s_1\cdots s_{j}) \cdot (s_{n-1}\cdots s_{k}) \cdot \Tc = \Tc$ or equivalently $s_0 \cdot (s_1\cdots s_{j}) \cdot (s_{n-1}\cdots s_{k})\cdot \Tc = (s_i \cdot (s_{i-1}\cdots s_1) \cdot (s_{i+1}\cdots s_{n-1}))^{-1} \cdot \Tc$.

It is easy to show that $(s_1\cdots s_{j}) \cdot (s_{n-1}\cdots s_{k}) = [2,3,\ldots,j+1,1,j+2,\ldots,k-1,n,k,\ldots,n-1]$ and $ (s_i \cdot (s_{i-1}\cdots s_1) \cdot (s_{i+1}\cdots s_{n-1}))^{-1} = [2,3, \ldots, j+1, j+2, \ldots, i, n, 1, i+1, \ldots, k-1, k, \ldots, n-1]$. Therefore, $(s_1\cdots s_{j}) \cdot (s_{n-1}\cdots s_{k}) \cdot \Tc$ and $(s_i \cdot (s_{i-1}\cdots s_1) \cdot (s_{i+1}\cdots s_{n-1}))^{-1}  \cdot \Tc$ are the same except two rows $\{i+1,\ldots, k-1,n\}, \{1, j+2, \ldots, i\}$ in the former and $\{1, i+1, \ldots, k-1\},\{j+2, \ldots, i, n\}$ in the latter. Now it is clear that $s_0$ interchanges these two tableaux, which implies the claim. 

Let us now consider a general case, i.e. when $A_w = \gamma(\mathbf{t})(A_{w'})$ for some $w' \in \Sym_n^\lambda$ and $\mathbf{t}\in \cT$. As $\cT$ is a free abelian group generated by $\mathbf{t}_{\alpha_i}$ for $i \in [1,n-1]$, we may write $\mathbf{t} = \sum_{i=1}^{n-1} c_i \mathbf{t}_{\alpha_i}$ for some $c_i \in \bZ$. Then the statement follows from induction on $\sum_{i=1}^{n-1} |c_i|$. 
\end{proof}

%

\ytableausetup{notabloids, smalltableaux}
\subsection{Lusztig's conjecture}
Here we prove \cite[Conjecture 13.13(b)]{lus97} for type $A$, one of the conjectures of Lusztig relating periodic $W$-graphs and left cells of $W$, using affine matrix-ball construction (\cite{clp17}, \cite{cpy18}). 
For a partition $\lambda$, we set $\Tas=\Tas_\lambda$ to be the standard Young tableau obtained from $\Tc$ by flipping it along the horizontal axis and pushing boxes up so that the shape becomes $\lambda$ again. For example, we have $\Tas_{(4,3,1)} = \ytableaushort{1348,267,5}$.

For $1\leq i \leq l(\lambda)$, define $r_i(id)$ to be an element in $\extSn$ whose window notation is $[1, 2, \ldots, s-1, s+1, \ldots, t-1, t, s+n, t+1,\ldots, n]$ where $s=1+\sum_{j=i+1}^{l(\lambda)} \lambda_j$ and $t=\sum_{j=i}^{l(\lambda)} \lambda_j$. In other words, $r_i(id)$ sends $s, s+1, \ldots, t-1$ to $s+1, s+2, \ldots, t$ respectively, and $t$ to $s+n$. Now we set $r_i :\extSn \rightarrow \extSn$ to be $r_i(w)= w\cdot r_i(id)$. (As a result, the two definitions of $r_i(id)$ coincide.) If $T_w$ is a Young tableau of shape $\lambda$ whose reading word is the same as $[w(1), w(2), \ldots, w(n)]$ for some $w\in \extSn$, then the action $r_i$ corresponds to replacing the $i$-th row of $T_w$, say $(a_1, a_2, \ldots, a_k)$, with $(a_2, \ldots, a_k, a_1+n)$. Also, the $\gamma$-action of $\cT/\cT_\lambda$ on $\fA_\lambda$ is equivalent to the action of $\{a_1r_1+\cdots+a_{l(\lambda)}r_{l(\lambda)} \mid a_1+\cdots+a_{l(\lambda)}=0\}$ on $\{w \in \affSn \mid A_w \in \fA_\lambda\}$.


Note that $u\in \Sym_n^\lambda$ if and only if $u\in \Sym_n$ and $u(i)<u(j)$ for any $i,j$ such that $\sum_{k=t+1}^{l(\lambda)}\lambda_k< i<j\leq \sum_{k=t}^{l(\lambda)}\lambda_j$ for some $t \in [1,l(\lambda)]$. Set $w \in \extSn$ to be $w = (a_1r_1+\cdots+ a_{l(\lambda)}r_{l(\lambda)})\cdot u$ for some $a_1, \ldots, a_{l(\lambda)}$. (Here we allow $w$ to be in $\extSn-\affSn$.) Let $T_w$ be a Young tableau whose reading word is the same as the window notation $[w(1), w(2), \ldots, w(n)]$ of $w$. Then entries of $T_w$ are increasing along rows, if $a, b$ are entries of $T_w$ contained in the same row then $|a-b|<n$, and the residues modulo $n$ of the entries of the $i$-th row of $T_w$ are the same as those of $\Upsilon(A_u)$, since these properties are preserved by the action of $r_i$ for any $i\in [1, l(\lambda)]$. (When $w \in \affSn$ we also have $A_w \in \fA_\lambda$ and $\Upsilon(A_w) = \Upsilon(A_u)$.) Now we prove the following theorem.

\begin{thm} Suppose that entries of $T_w$ are also increasing along columns. If $(P, Q, \vv{\rho})$ is the image of  $w$ under affine matrix-ball construction (defined in \cite{clp17}, \cite{cpy18}), then $P=\Upsilon(A_u)$, $Q=\Tas$, and $\vv{\rho}=(a_1, a_2,  \ldots, a_{l(\lambda)})$.
\end{thm}
\begin{proof} As entries of $T_w$ are increasing along columns, it is easy to show that if $b$ is on the lower row than $a$ in $T_w$ then $b+n>a$ (regardless of the columns in which $a$ and $b$ are contained). Now let us consider asymptotic realization of affine matrix-ball construction \cite[Section 7]{cpy18} and note that $P$ can be obtained by taking the (asymptotic) residue modulo $n$ of the insertion tableau of the infinite sequence $(w(1), w(2), \ldots)$ under the usual Robinson-Schensted correspondence. However, the observation above implies that if $\ceil{i/n}<\ceil{j/n}$ then $w(j)$ does not bump $w(i)$ in the column insertion process. (Here, $\ceil{\alpha}$ is the smallest integer which is not smaller than $\alpha$.) Therefore, the input of each period $(w(an+1), w(an+2), \ldots, w(an+n-1))$ for any $a\in \bN$ under column insertion becomes the same as $T_w$ shifted by $an$, i.e. $an+T_w$. By \cite[Corollary 7.5]{cpy18}, this means that $P$ and $T_w$ have the same residues modulo $n$ in each row and thus $P=\Upsilon(A_u)$.

We argue similarly for $Q$ part. Remark (\cite[Remark 7.7]{cpy18}) that $Q$ can be obtained by taking the (asymptotic) residue modulo $n$ of the insertion tableau under the Robinson-Schensted correspondence, say $\tilde{Q}$, of the infinite sequence $(w^{-1}(x), w^{-1}(x+1), \ldots)$ for any $x\in \bZ$, or more precisely the two-line array
$$\begin{pmatrix} x&x+1&x+2&x+3&\cdots
\\w^{-1}(x)&w^{-1}(x+1)&w^{-1}(x+2)&w^{-1}(x+3)&\cdots
\end{pmatrix}.
$$
Choose $x$ such that $\{w^{-1}(x), w^{-1}(x+1), \ldots, \} \supset \bZ_{>0}$ (which is true for any sufficiently small $x$). Then by flipping the array above and reordering if necessary so that the first row becomes increasing, we see that all but finite number of entries of $\tilde{Q}$ are the same as those of the recording tableau of the two-row array
$$\begin{pmatrix} 1&2&3&4&\cdots
\\w(1)&w(2)&w(3)&w(4)&\cdots
\end{pmatrix}.
$$
It follows that $Q$ is obtained by taking the (asymptotic) residue module $n$ of the recording tableau of $(w(1), w(2), \ldots)$. Then one can show that $Q=\Tas$ using the argument similar to the $P$ part as above. (Note that $\Tas=\Tas_\lambda$ is the recording tableau of the reading word of any tableau of shape $\lambda$ whose entries are increasing along both rows and columns.)

It remains to discuss the $\vv{\rho}$ part. To this end we freely use notations and results in \cite{cpy18}. From the assumptions on $w$, we see that $\{(x, w(x)) \mid n-\lambda_1< x \leq n\}$ and its translates by $\bZ\cdot (n,n)$ in $\bZ\times \bZ$ are the southwest channel of $w$ and each zigzag consists of balls corresponding to each column of $T_w$ and its translates by $\bZ\cdot (n,n)$. This means that the window notation of $\fw(w)$ is obtained from inserting some $\emptyset$ in the sequence $(w(1), w(2), \ldots, w(n-\lambda_1))$. Thus we may use induction on the number of rows to conclude that $\vv{\rho} = (y, a_2, \ldots, a_{l(\lambda)})$ for some $y \in \bZ$. Now by \cite[Lemma 10.6]{cpy18} and the comment thereof we have $y+\sum_{i=2}^{l(\lambda)}a_i = \frac{1}{n}\sum_{j=1}^n(w(j)-j)$, where the latter term is equal to $\sum_{i=1}^{l(\lambda)}a_i$. (The action of each $r_i$ increases $\frac{1}{n}\sum_{j=1}^n(w(j)-j)$ exactly by 1.) Thus $y=a_1$ and the result follows.
\end{proof}

\begin{rmk} This confirms \cite[Conjecture 13.13(b)]{lus97} for type $A$. Indeed, $\mathbf{t} \in \cT$ is ``large'' as described therein if and only if $a_1\ll a_2\ll \cdots \ll a_{l(\lambda)}$, which implies that entries of $T_w$ are increasing along columns.
\end{rmk}


\subsection{Quotient of $\pwg_\lambda$ by $\gamma(\cT)$} \label{sec:quotWgraphs}Here we construct the quotient of $\pwg_\lambda=(\fA_{\lambda}, \mu, \mathfrak{I})$ by the action of $\gamma(\cT)$, denoted by $\qwg_\lambda$. 
To this end, first observe that (the complementary version of) \cite[Proposition 11.15]{lus97} shows that $\mu(A, B) = \mu(\gamma(\mathbf{t})(A), \gamma(\mathbf{t})(B))$ for any $\mathbf{t} \in \cT$ and $A, B \in \fA_\lambda$. Also, we need the following lemma.
\begin{lem} \label{lem:gammatau} For $\mathbf{t} \in \cT$ and $A \in \fA_\lambda$ we have $\mathfrak{I}(A) = \mathfrak{I}(\gamma(\mathbf{t})(A))$.
\end{lem}
\begin{proof} Recall that $\mathfrak{I}(A)$ is the set of simple reflections $s$ such that $s\cdot A \in \fA_\lambda$ and $s\cdot A>A$. By symmetry, it suffices to show that if $s\in \mathfrak{I}(A)$ then $s \in \mathfrak{I}(\gamma(\mathbf{t})(A))$, i.e. $s \cdot \gamma(\mathbf{t})(A) \in \fA_\lambda$ and $s \cdot \gamma(\mathbf{t})(A)>\gamma(\mathbf{t})(A)$. First, note that $s \cdot \gamma(\mathbf{t})(A)=\gamma(\mathbf{t})(s \cdot A)$ (when $s\cdot A \in \fA_\lambda$) since $\gamma$-action is defined in terms of the right action of $\affSn$. Thus the first part is clear. For the second part, it suffices to show that $\gamma$-action preserves the order $\geq$ on $\fA_\lambda$. From the definition of $\geq$, it suffices to show that $\gamma$-action preserves the function $d:\fA_\lambda \times \fA_\lambda \rightarrow \bZ$. But this follows from \cite[2.12(c)]{lus97}.
\end{proof}


We are ready to define $\qwg_\lambda=(V, m, \tau)$ as follows. First we set $V= \RSYT(\lambda)$ which is identified with the $\gamma(\cT)$-orbits of $\fA_\lambda$ under the bijection $\fA_\lambda/\gamma(\cT) \simeq \fund_\lambda \xrightarrow{\Upsilon} \RSYT(\lambda)$. We also set $\tau = \ades$. Then for any $A \in \fA_\lambda$ in the $\gamma(\cT)$-orbit parametrized by $T$, we have $s_i \in \mathfrak{I}(A)$ if and only if $i \in \ades(T)$ by Lemma \ref{lem:tauinv} and \ref{lem:gammatau}. Finally, for $T, T' \in \RSYT(\lambda)$ we define $m(T, T')= \sum_{B}\mu(A,B)$, where $A$ is an element in the $\gamma$-orbit parametrized by $T$ and the sum is over all $B$ in the $\gamma$-orbit parametrized by $T'$. We claim that this is well-defined. Indeed, even if each $\gamma$-orbit contains infinitely many alcoves in general, $\mu(A,B)$ is zero for all but finitely many $B$ because of the result of \cite{var04} and \cite[Consequence 13.8]{lus97}. Furthermore, as $\mu: \fA_\lambda \times \fA_\lambda \rightarrow \bZ$ is invariant under $\gamma$-action, $m(T, T')$ does not depend on the choice of $A$.

It is not hard to show that $\qwg_\lambda$ satisfies the defining conditions of a $\affSn$-graph described in \ref{sec:defwgraph} provided that so does $\pwg_\lambda$. Thus $\qwg_\lambda$ is a $\affSn$-graph. Also, it defines a finite-dimensional representation of the Hecke algebra of $\affSn$ constructed in \cite[0.3]{lus97} where the homomorphism $\bZ[q^{\pm\frac{1}{2}}][\cT] \rightarrow \mathbf{K}$ therein ($\mathbf{K}$ is a field of characteristic 0) corresponds to the trivial representation of $\cT$.

\subsection{Properties of $\qwg_\lambda$ under nonnegativity assumption on $\mu$}
It is conjectured \cite[Conjecture 13.16]{lus97} that coefficients of $p_{A,B}$ are nonnegative integers for any $A, B \in \fA_\lambda$, which in particular implies that $\mu(A,B)\geq 0$. (To the authors' best knowledge it is still open.) Here, we assume the nonnegativity of $\mu$-function and discuss some properties of $\pwg_\lambda$ and $\qwg_\lambda$.

\begin{lem} Suppose that $\mu(A,B)\geq 0$ for any $A, B\in \fA_\lambda$. Then $\pwg_\lambda=(\fA_\lambda, \mu, \mathfrak{I})$ is admissible and $\qwg_\lambda=(\RSYT(\lambda), m, \ades)$ is nb-admissible.
\end{lem}
\begin{proof} First $\im \mu \subset \bN$ by assumption, which also implies that $\im m \subset \bN$. Also, $\pwg_\lambda$ (resp. $\qwg_\lambda$) satisfies the Simplicity Rule by \cite[Remark 4.3]{ste08}, which implies that $\mu(A,B)=\mu(B,A)$ (resp. $m(T, T')=m(T', T)$) whenever $\mathfrak{I}(A)$ and $\mathfrak{I}(B)$ (resp. $\ades(T)$ and $\ades(T')$) are not comparable. Finally, $\pwg_\lambda$ is bipartite as a result of \cite[Proposition 11.12]{lus97}; one may choose the color of each vertex $A\in\fA_\lambda$ to be the residue of $d(A, A_{id})$ modulo 2.
\end{proof}

The next proposition describes the simple underlying graph of $\qwg_\lambda$.
\begin{prop} Suppose that $\mu(A,B)\geq 0$ for any $A, B \in \fA_\lambda$. Then $U(\qwg_\lambda) \simeq \adeg_\lambda$.
\end{prop}
\begin{proof} Suppose we are given $T, T' \in \RSYT(\lambda)$ and let $\cO_T$ and $\cO_{T'}$ be $\gamma$-orbits in $\fA_\lambda$ parametrized by $T$ and $T'$ respectively. If there exists an undirected edge between $T$ and $T'$ in $\qwg_\lambda$, or $m(T,T')=m(T',T)=1$, then $\sum_{B\in \cO_{T'}}\mu(A_0, B)=\sum_{A\in \cO_{T}}\mu(B_0, A)=1$ where $A_0 \in \cO_T$ and $B_0 \in \cO_{T'}$ are arbitrary. Also $\ades(T)$ and $\ades(T')$ are incomparable, or equivalently $\mathfrak{I}(A)$ and $\mathfrak{I}(B)$ are incomparable for any $A \in \cO_T$ and $B \in \cO_{T'}$. Since $\im \mu \subset \bN$, there exists a unique $B \in \cO_{T'}$ such that $\mu(A_0, B)=1$, which we may set to be $B_0$. Then as $\pwg_\lambda$ is admissible we have $\mu(B_0, A_0)=1$ as well, i.e. there exists an undirected edge between $A_0$ and $B_0$ in $\pwg_\lambda$. From the definition of $\mu$, this is only possible if there exists a simple reflection $s\in \affSn$ such that $B_0 = s\cdot A_0$. Thus by Lemma \ref{lem:upsilontab} it follows that $T$ and $T'$ are connected by a single Knuth move.

Conversely, this time let us assume that $T, T' \in \RSYT(\lambda)$ are connected by a single Knuth move. Then for any $A \in \cO_T$, there exists a simple reflection $s\in \affSn$ such that $s\cdot A \in \cO_{T'}$ again by Lemma \ref{lem:upsilontab}. Thus by \cite[Corollary 11.7]{lus97} together with the fact that $\pwg_\lambda$ satisfies the Simplicity Rule, we have $\mu(A, B)=\mu(B, A)=1$. It follows that $m(T, T'), m(T', T)\geq 1$ by nonnegativity assumption of $\mu$, which implies that $m(T,T')=m(T',T)=1$ as $\qwg_\lambda$ satisfies the Simplicity Rule.
\end{proof}

Now the following theorem is a natural consequence.
\begin{thm} Suppose that $\mu(A,B)\geq 0$ for any $A, B \in \fA_\lambda$. Then for any two-row partition $\lambda$, we have $\awg_\lambda\simeq \qwg_\lambda$.
\end{thm}
\begin{proof} If $\lambda$ consists of two rows of unequal length, then it follows from Theorem \ref{thm:unequniq}. Thus suppose that $\lambda$ consists of two equal rows. By Theorem \ref{thm:eqmax}, it suffices to show that there exists an embedding $\awg_\lambda \rightarrow \qwg_\lambda$. By Corollary \ref{cor:eqmin}, it suffices to show that if there exists a directed edge $T\rightarrow T'$ (of weight 1) in $\awg_\lambda$ for $T$ and $T'$ in different simple components then the same directed edge appears in $\qwg_\lambda$ (with weight 1). To this end, let $\cO_T$ and $\cO_{T'}$ be the $\gamma$-orbits in $\fA_\lambda$ parametrized by $T$ and $T'$ respectively. As $T\rightarrow T'$ is always a move of the first kind by Lemma \ref{lem:arctworow}, there exists a simple reflection $s\in \affSn$ such that $T'=s\cdot T$. This means that for any $A\in \cO_T$ we have $s\in \mathfrak{I}(A)$ and $s\cdot A \in \cO_{T'}$. Now \cite[Corollary 11.7]{lus97} and \cite[Lemma 11.9]{lus97} imply that $m(T, T') = \sum_{B \in \cO_{T'}} \mu(A, B) = \mu(A, s\cdot A) = 1$, thus the result follows.
\end{proof}

\bibliographystyle{amsalpha}
\bibliography{wgraph}

\end{document}